\documentclass[a4paper, reqno, 11pt]{amsart}
\usepackage[a4paper, bottom=3cm, left=3cm, right=2.5cm]{geometry}

\date{}

\usepackage{mathrsfs,amssymb,amsmath,amsthm}
\usepackage[amsstyle]{cases}
\usepackage{stmaryrd}
\usepackage{graphicx,mathtools,enumerate}
\usepackage{hyperref}
\allowdisplaybreaks

\newtheorem{theorem}{Theorem}[section]
\newtheorem{proposition}[theorem]{Proposition}

\newtheorem{lemma}[theorem]{Lemma}

\theoremstyle{definition}

\theoremstyle{remark}
\newtheorem{remark}[theorem]{Remark}

\numberwithin{equation}{section}

\newcommand{\vare}{\varepsilon}
\newcommand{\id}{\mathcal{I}}
\newcommand{\laplace}{\Delta}
\newcommand{\dive}{\operatorname{div}}
\newcommand{\trace}{\operatorname{Tr}}
\newcommand{\tizeta}{\tilde{\zeta}}
\newcommand{\tiOmega}{\tilde{\Omega}}
\newcommand{\tiomega}{\tilde{\omega}}
\newcommand{\tiv}{\tilde{v}}
\newcommand{\tiw}{\tilde{w}}
\newcommand{\tiG}{\tilde{G}}
\newcommand{\tiX}{\tilde{X}}
\newcommand{\tiq}{\tilde{q}}
\newcommand{\tif}{\tilde{f}}
\newcommand{\tig}{\tilde{g}}
\newcommand{\tih}{\tilde{h}}
\newcommand{\tiF}{\tilde{F}}
\newcommand{\tiK}{\tilde{K}}
\newcommand{\tiH}{\tilde{H}}

\begin{document}

\title{Splash singularity for the free boundary incompressible viscous MHD}

\author{Chengchun Hao}
\email{hcc@amss.ac.cn}

\author{Siqi Yang}
\email{yangsiqi@amss.ac.cn}

\address{Academy of Mathematics \& Systems Science, Chinese Academy of Sciences, Beijing 100190, China}
\address{Hua Loo-Keng Key Laboratory of Mathematics,
	Chinese Academy of Sciences, Beijing 100190, China}
\address{School of Mathematical Sciences,
	University of Chinese Academy of Sciences, 
	Beijing 100049, China
}

\begin{abstract}
    In this paper, we prove the existence of smooth initial data for the two-dimensional free boundary incompressible viscous magnetohydrodynamics (MHD) equations, for which the interface remains regular but collapses into a splash singularity (self-intersects in at least one point) in finite time. The existence of the splash singularities is guaranteed by a local existence theorem, in which we need suitable spaces for the modified magnetic field and modification of the velocity and the pressure such that the modified initial velocity is zero, and a stability result which allows us to construct a class of initial velocities and domains for an arbitrary initial magnetic field. It turns out that the presence of the magnetic field does not prevent the viscous fluid to form splash singularities for certain smooth initial data.
\end{abstract}


\maketitle

\tableofcontents

\section{Introduction}\label{Introduction}
In the present paper, we are concerned with the formation of splash singularities for two-dimensional incompressible viscous magnetohydrodynamics (MHD) without  magnetic diffusivity. It consists of solving a bounded variable domain $ \Omega(t)\subset\mathbb{R}^2 $
filled with incompressible viscous electrically conducting homogeneous plasma, the density of which is a positive constant, together with the vector field of velocity $ u(t, x) = (u^1, u^2)^\top, $ the scalar
pressure $ q(t,x) $ and the magnetic field $ H(t,x)=(H^1,H^2)^\top  $ satisfying the system of  MHD. The boundary $ \partial\Omega(t) $ of $ \Omega(t) $ is the free surface of the plasma.

The problem can be represented in the following form (cf. \cite{padula2010nonlinear,Padula2011}). In the plasma region,  the incompressible viscous MHD equations read
\begin{align*}
	\begin{cases} 
	\partial_t u+u \cdot \nabla u-\dive T(u,q)=\dive T_M(H), & \text { in } \Omega(t), \\
	\partial_t H+u \cdot \nabla H=H\cdot \nabla u,  & \text { in } \Omega(t), \\
	\operatorname{div} u=0,\quad \operatorname{div} H=0, & \text { in } \Omega(t), 
	\end{cases}
\end{align*}
where the time $ t>0; $  $ u\cdot \nabla  $ and $ H\cdot \nabla $ are directional derivatives;  $ T(u,q)=-q\id+\nu S(u) $ is the viscous stress tensor; $\id$ is the $2\times 2$ identity matrix; $ \nu $ is the kinematic viscosity;
$ S(u)=\nabla u +(\nabla u)^\top  $ is the doubled rate-of-strain tensor; $ (\nabla u)_{ij}= \partial_j u^i; $ $ T_M(H)=\mu(H\otimes H-\frac{1}{2}|H|^2\id) $ is the magnetic stress tensor;  $ (H\otimes H)_{ij}=H^i H^j; $ $ \mu $ is the magnetic permeability. We assume $ \nu=1 $ and $ \mu=1 $ for simplicity.

The plasma surface is free to move and the presence of the kinematic viscosity leads to the following condition that must be satisfied on the surface, i.e.,
\begin{align*}
	(T(u,q)+T_M(H)) n=0,\text { on } \partial \Omega(t), 
\end{align*}
where  $ n $ is the outward normal to $ \partial\Omega(t) $. 

For convenience, we denote 
\begin{align*}
	p=q+\frac 12|H|^2.
\end{align*}
The system can be rewritten as
\begin{equation}\label{mhd}
	\begin{cases}
		\partial_t u+u \cdot \nabla u-\laplace u+\nabla p=H\cdot \nabla H, & \text { in } \Omega(t), \\
		\partial_t H+u \cdot \nabla H=H\cdot \nabla u,  & \text { in } \Omega(t), \\ 
		\operatorname{div} u=0,\quad \operatorname{div} H=0, & \text { in } \Omega(t), \\
		(-p\id+\nabla u+(\nabla u)^\top +H\otimes H) n=0, & \text { on } \partial \Omega(t), \\ 
		u(0, \cdot)=u_0, \quad H(0, \cdot)=H_0, & \text { in } \Omega_0,
	\end{cases}
\end{equation}
where $ n=(n^1, n^2)^\top  $ is the unit outer normal to $ \partial  \Omega(t) $;
$\Omega_0, u_0$ and $H_0$ are prescribed initial data which satisfy the compatibility conditions
\begin{equation}\label{equinitialdata}
	\begin{cases}
		\operatorname{div} u_0=0, \quad \operatorname{div} H_0=0, & \text { in }  \Omega_0, \\ 
		n_0^{\perp} ( (\nabla u_0+\nabla u_0^\top  )+ H_0 \otimes H_0 ) n_0=0, & \text { on } \partial \Omega_0,
	\end{cases}
\end{equation}
where $ n_0=(n_0^1, n_0^2)^\top  $ is the unit outer normal to $ \partial \Omega_0 $ and $ n_0^{\perp}=(-n_0^2,n_0^1) $.

We first state the main result of this paper as follows.

\begin{theorem}
	There exists a bounded domain $ \Omega  $  such that for any divergence free $ H_0\in H^k(\Omega)$, there exists  a   solution $ (u,p,H) $ for the viscous MHD equations \eqref{mhd} in $ [0, t^*) $ for some $ t^*>0 $ and the surface $ \partial\Omega(t) $ self-intersects at time $  t^* $ in at least  one point which creates a splash singularity.
\end{theorem}


The analogous result for the two-dimensional inviscid water wave equation was studied by  Castro \emph{et al.} in \cite{castro2012splash}. By proving a   local existence theorem  and a structural  stability result in the  transformed
domain, they exhibited smooth initial data   for which the smoothness of the interface breaks down in finite time, i.e., collapses in a splash singularity. At the same time, their stability theorem shows that a sufficiently small
perturbation of initial data which leads to a splash singularity  still generates a
splash singularity. In their proofs, a conformal map was used to transform the equation  in $ \Omega(t) $  to a new domain $ \tiOmega(t) $ and the transformed water wave equations behave much like the original equations. By adapting the  energy estimates in \cite{Ambrose_2005} and  \cite{cordoba2010interface}, the existence of solutions was proved in the transformed domain for times in $ [t^*-\vare, t^*+\vare], $ with a given splash domain $ \Omega(t^*) $. Then, they chose a suitable initial velocity to complete the proofs.

Except for the splash singularity, the splat singularity defined in \cite{castro2013finite} is another singularity that can arise in  the water waves, i.e., the solutions collapse along an arc in finite time but remain smooth otherwise. It turned out that  there exist smooth initial data  for which the smoothness
of the interface breaks down in finite time into a splash singularity or a
splat singularity  in \cite{castro2013finite}. 

In the presence of viscosity, the strategy for the inviscid case could not work since
the equations failed to be solved backward in time, i.e., the existence of solutions in the transformed domain for times $ \left[t^*-\vare,t^*\right)  $ is  rather difficult. For this problem, Castro \emph{et al.} utilized the
transformation to the tilde domain in a new way and proved that there exist  solutions to the viscous water wave equation that remain smooth for short time  but form a splash singularity at a finite time in \cite{castro2019splash}. 
For the three-dimensional case, Coutand and Shkoller have proved that the three-dimensional  free-surface incompressible Euler equations with regular initial geometries and velocity fields that satisfy the Taylor sign condition have solutions that form a finite-time splash
or splat singularity in \cite{coutand2014finite}. They do not assume that the fluid
is irrotational, and as such, their method can be used for   other fluid interface problems, including compressible flows, plasma, as well as the inclusion of surface tension
effects. 
  For the two-dimensional or three-dimensional  viscous water
wave, they also proved that  given a
sufficiently smooth initial boundary  which is close to self-intersection  and a divergence-free velocity field
designed to push the boundary towards self-intersection, the  free-surface   will indeed self-intersect in finite time in \cite{Coutand2019}.

In \cite{di2020splash},  the two-dimensional free boundary viscoelastic fluid model of the Oldroyd-B type  at a high  Weissenberg number  was studied by applying the classical conformal mapping
method and  the existence of splash singularities was proved. It turned out that the action of the viscoelastic deformation does not prevent the existence of splash singularities.




In many important physical situations, magnetic fields are essential  (cf. \cite{ST83, Z88, CG68}). Examples include  solar flares  in astrophysics  \cite{CG68}. 
For the incompressible inviscid MHD in bounded domain,  Hao and Luo established  a priori estimates for the free boundary problem in \cite{Hao_2014} under the Taylor-type sign condition, and the ill-posedness for the two-dimensional case in \cite{Hao_2019} when the Taylor-type sign condition is violated.  Luo and Zhang obtained a priori estimates for the low regularity solution in the case when the domain has small volume in \cite{LZ20}. A local existence result was established in \cite{GW19}, for which the detailed proof is given in an initial flat domain of the form $\mathbb{T}^2 \times (0,1)$, for a two-dimensional period box $\mathbb{T}^2$ in $x_1$ and $x_2$. With the same set-up of the initial domain, the local well-posedness is obtained in \cite{GLZ22} by Gu, Luo and Zhang with surface tension. For the case where the magnetic field is zero on the free boundary and in vacuum, in the three-dimensional space with infinite and finite depth settings, Lee proved the local existence and uniqueness of the free boundary problem of
incompressible viscous-diffusive MHD flow in \cite{Lee_2017}, he also proved in \cite{Lee_2018} a local unique solution for the free boundary MHD without kinetic viscosity and magnetic diffusivity via zero kinetic viscosity-magnetic diffusivity limit.


To our knowledge, we are not aware of any previous mathematical research on the splash singularity for the MHD equations with free surface. To prove the existence of splash singularity for the  two-dimensional incompressible viscous MHD model, we reduce the original system to a system in conformal Lagrangian coordinates to form a  fixed boundary that is not  self-intersecting. However, to prove the local existence theorem, we must estimate the magnetic field additionally.
We have found that the choice of spaces to which the magnetic field belongs could not be treated in the same way as the viscoelastic fluid model in \cite{di2020splash}. Indeed, the estimates of the magnetic field involve some product of the iterative solutions, but the insufficient regularity of the functional space for spatial variables makes it impossible to use the bilinear inequalities as in the viscoelastic case. For this reason, we need to choose the  indices of the functional spaces carefully. Nevertheless, we could not simply raise the regularities for all spatial variables since the spaces to which the velocity and the pressure belong are space-time coupled Beale spaces. If we raise the regularities for all variables, the initial second-order derivatives of the solutions must be zero and  consequently, some key lemmas fail. In fact, we can choose suitable spaces for the modified magnetic field and  flux and  do not change the spaces for the velocity and the pressure to keep some lemmas valid in the present spaces. To obtain the desired estimates, we will also prove the estimates of the flux (cf. Lemma \ref{x-omega}) and the magnetic field (cf. Lemma \ref{G-G0}) in the present spaces to bound to the iterative sequence and show that the sequence is Cauchy. Moreover, to apply  Lemma \ref{lem4.1}, we still need a modification of the velocity and the pressure such that the modified initial velocity is zero. Due to the presence of a magnetic field,  we extend the analysis for the choice of the approximated solution  made in \cite{castro2019splash}. Finally, we show that the presence of the  magnetic field does not prevent the fluid to form splash singularities either.

The key element for the proof is the use of the classical method of conformal
mapping, see \cite{cordoba2010interface}, and the definition of the splash curve and  the splash singularity can be found in \cite{castro2013finite} and \cite{castro2019splash}. This method has been used recently for this kind of problem in \cite{castro2019splash}. We introduce the map
\begin{align*}
	P(z) =\sqrt{z-\alpha},\quad \text{ for } z \in \mathbb{C}\setminus \Gamma,
\end{align*}
which is defined as an analytic branch of $ \sqrt{z-\alpha}, $ where $  \Gamma $ is a branch cut, passed through the splash point (see Fig. \ref{fig:0}). 
\begin{figure}[h]
 	\centering
 	\includegraphics[width=0.8\linewidth]{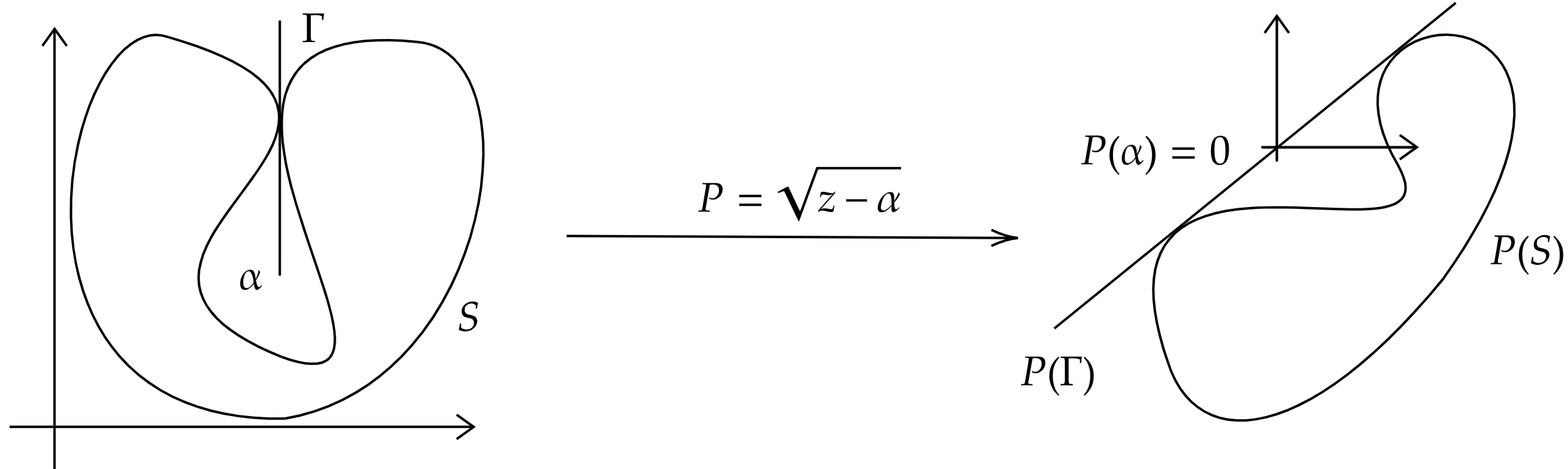}
 	\caption{Conformal mapping $ P. $}
 	\label{fig:0}
 \end{figure}
We take $ z \in \mathbb{C}\setminus \Gamma  $ to make $ P(z) $ an analytic function and to have $ P^{-1}(w)=w^2+\alpha, $ an entire function. The idea to prove our theorem is to reduce system \eqref{mhd}, in Eulerian coordinates, to a system in Lagrangian coordinates to form a fixed boundary, as in \cite{beale1981initial}. The second key observation
regards the behavior of the magnetic field at the Lagrangian boundary.
As a result, it shows that the way the viscokinematic deformation acts on the boundary does
not prevent the natural tendency of the fluid to form splash singularities. The idea hidden
in the proof is inspired by the geometric construction in \cite{castro2019splash}, as explained below. 
\begin{remark}
	Another type of splash scenario is illustrated in Fig. \ref{fig:00}. The proof can be easily adapted by replacing $ P(z)=\sqrt{z-\alpha} $ by a branch of $ \sqrt{\frac{z-\alpha}{z-\beta}} $  with suitable $ \alpha, \beta $ and the branch cut.
\end{remark}
\begin{figure}[h]
	\centering
	\includegraphics[width=0.8\linewidth]{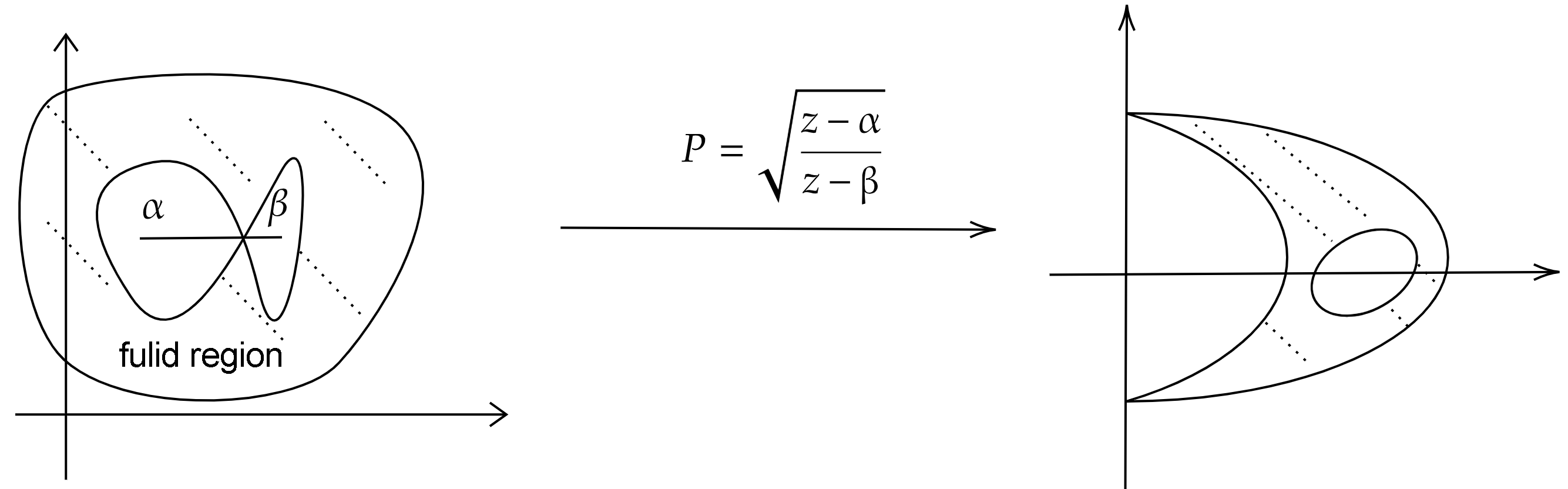}
	\caption{Another type of splash scenario.}
	\label{fig:00}
\end{figure}

From now on, we assume that $ \alpha=0 $, then $ P(z)=\sqrt{z}  $   for   $ z \in \mathbb{C}\setminus \Gamma.$

\begin{itemize}
	\item Let the initial domain $ \Omega_0 $ be a non-regular domain as in Fig. \ref{fig:1}. We use the conformal map $ P $ to get $ \tiOmega_0=P(\Omega_0), $ a non-splash type domain.
		
\begin{figure}[h]
	\centering
	\includegraphics[width=0.8\linewidth]{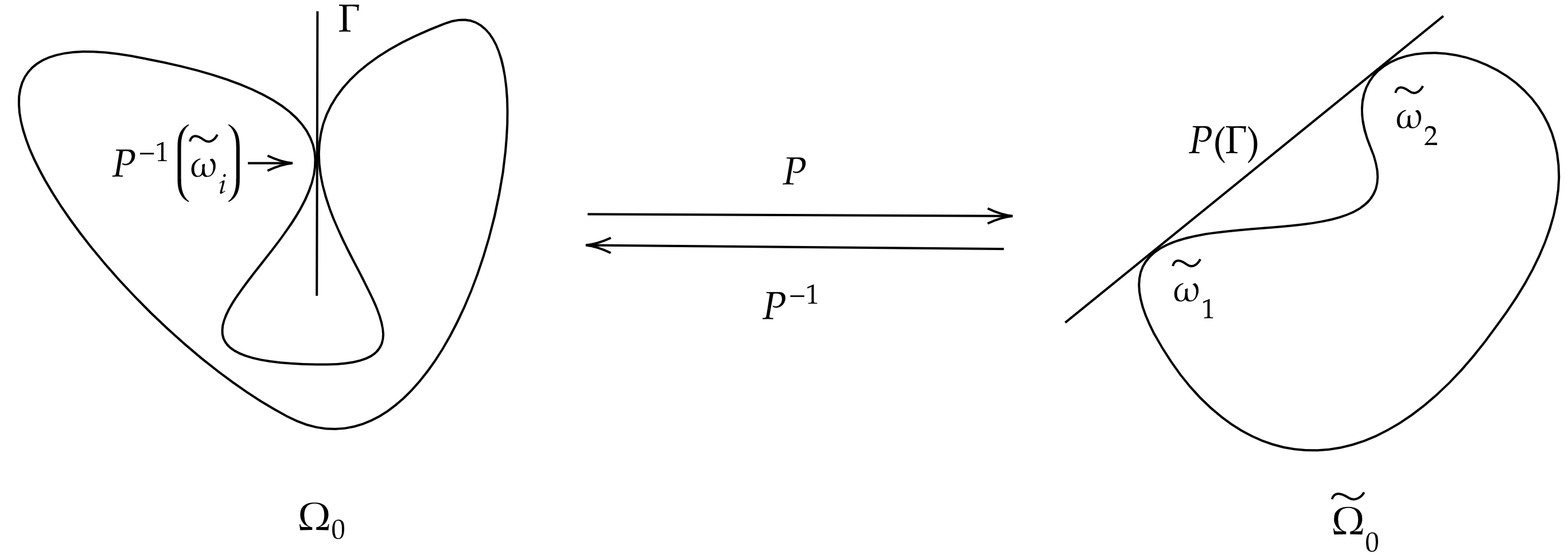}
	\caption{Initial domains $ \Omega_0 $ and $ \tiOmega_0. $}
	\label{fig:1}
\end{figure}

	\item For smooth initial data $ (\tiOmega_0, \tiv_0, \tiG_0)  $, we  prove the local existence  in Section \ref{Local existence of smooth solutions} and obtain a solution $  (\tiOmega(t), \tiv(t,\cdot), \tiq(t,\cdot), \tiG(t,\cdot) ) $ for $ t\in[0,T], $ with $ T>0. $ 	

	\item By a suitable choice of the initial velocity (Section \ref{Existence of splash singularity}) as in Fig. \ref{fig:2} such that $  \tiv_{0}(\tiomega_1)\cdot \tilde{n}_0(\tiomega_1)>0  $ and $  \tiv_{0}(\tiomega_2)\cdot \tilde{n}_0(\tiomega_2)>0, $
	 we have for $ \bar{t}\in (0, T) $ so small  that $ P^{-1}(\partial\tiOmega(\bar{t})) $ is self-intersecting  as in Fig. \ref{fig:3}. This solution lives only in the complex plane so it cannot be reversed into a solution in the non-tilde domain by $ P^{-1}. $ Hence, it is   insufficient to prove the existence of splash singularity.

	 \begin{figure}[h]
	 	\centering
	 	\includegraphics[width=0.8\linewidth]{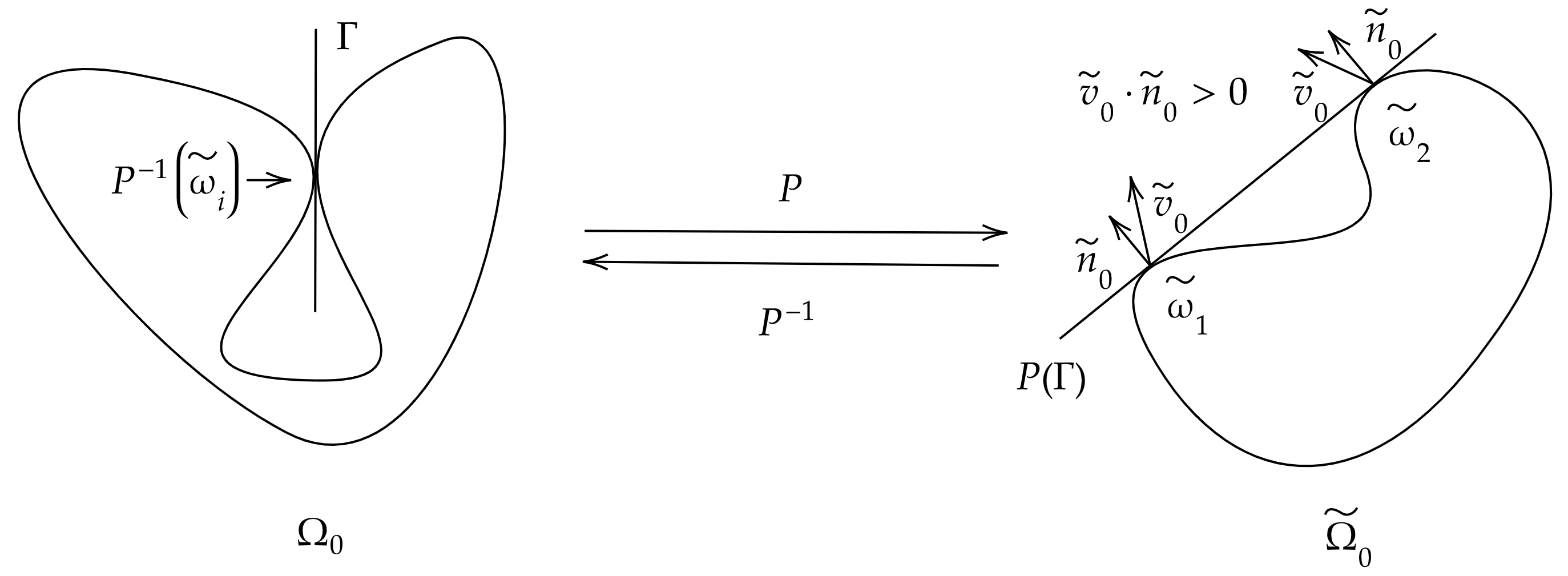}
	 	\caption{A suitable choice of the initial velocity.}
	 	\label{fig:2}
	 \end{figure}
\begin{figure}[h]
	\centering
	\includegraphics[width=0.8\linewidth]{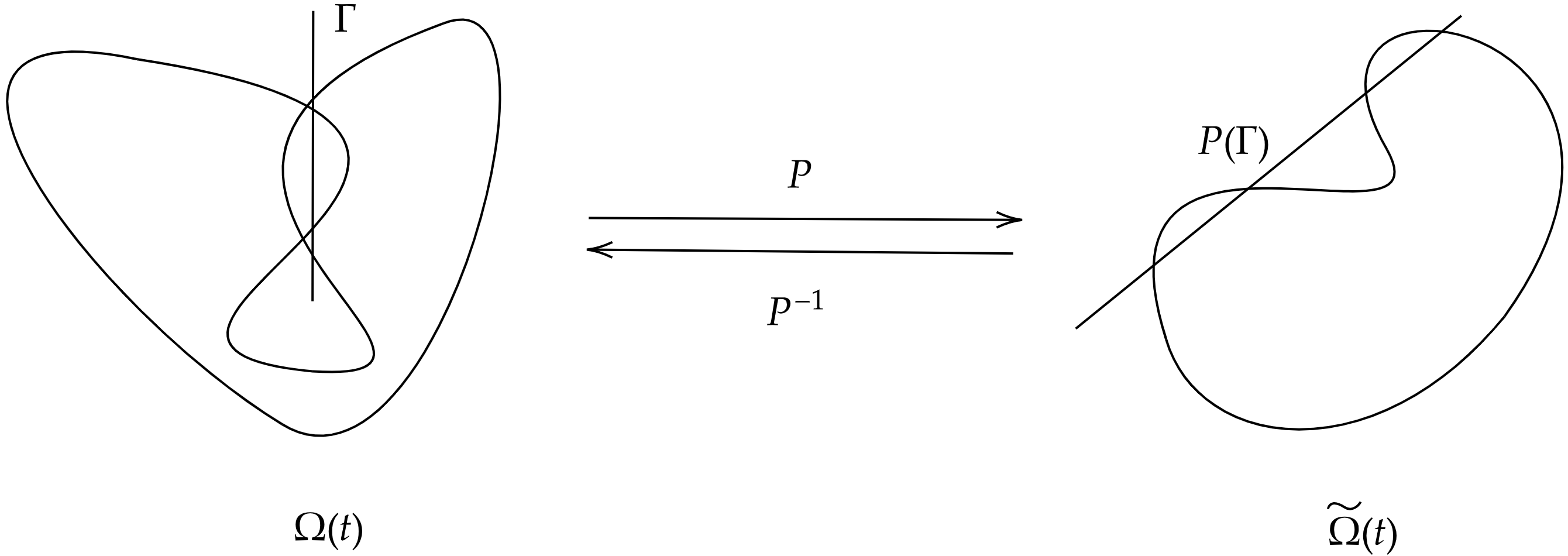}
	\caption{Solutions in the  tilde domain.}
	\label{fig:3}
\end{figure}

	\item To solve the problem in the non-tilde domain, we take a one-parameter family of initial data
	$ \{ (\tiOmega_\vare(0), \tiv_{\vare, 0}^\prime, \tiG_{\vare, 0}^\prime) :\vare>0 \}  $ with $ \tiOmega_{\vare}(0)=\tiOmega_0+\vare b $ and $ |b| = 1, $ such that $ P^{-1}(\partial\tiOmega_{\vare}(0)) $
	is regular as in Fig. \ref{fig:4}, then there exists a local-in-time smooth solution $  (\tiOmega_\vare(t), \tiv_{\vare}^\prime(t,\cdot), \tiq_{\vare}^\prime(t,\cdot), \tiG_{\vare}^\prime(t,\cdot))$,  and by the inverse mapping there exists a local-in-time smooth solution $ (\Omega_\vare(t), u_{\vare}^\prime(t,\cdot), p_{\vare}^\prime(t,\cdot),$  $ H_{\vare}^\prime(t,\cdot))$ in the non-tilde domain.
	\begin{figure}[h]
	\centering
	\includegraphics[width=0.8\linewidth]{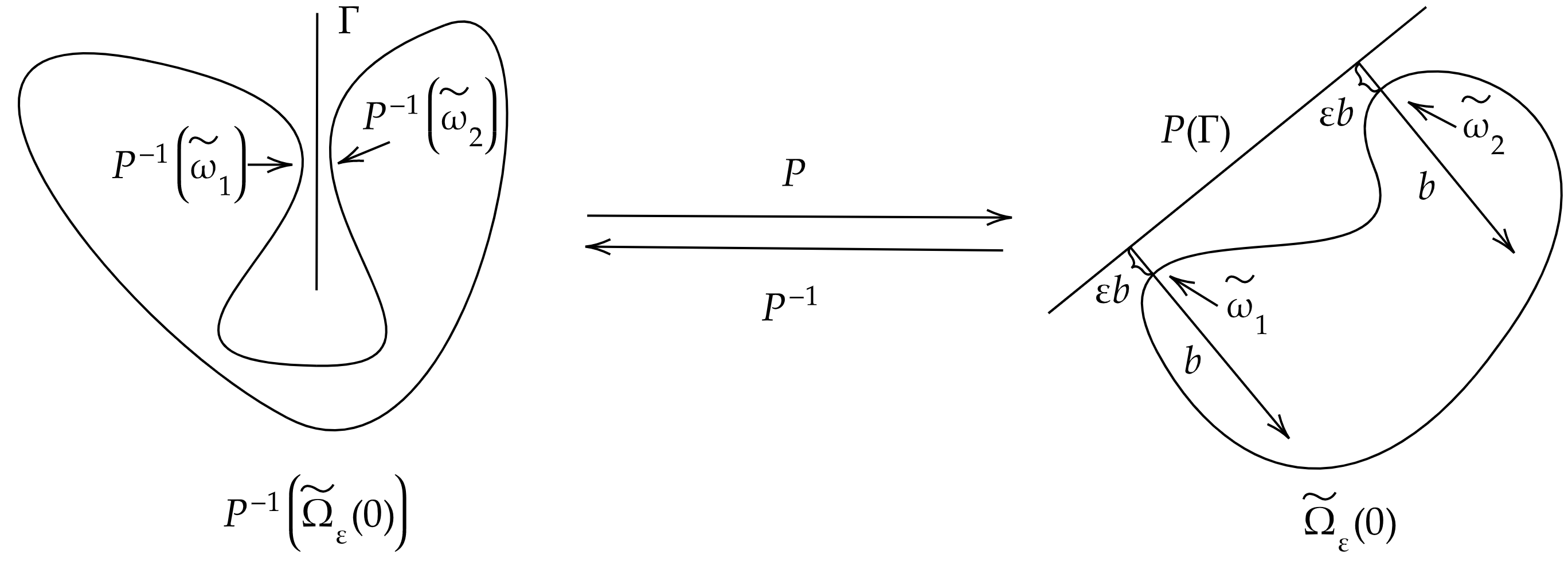}
	\caption{Regular domain $ P^{-1}(\partial\tiOmega_{\vare}(0)). $ }
	\label{fig:4}
\end{figure}	
	\item Then, for sufficiently small $ \vare>0, $ the stability result (Section \ref{Stability estimates}) compares the difference between $ (\tiOmega_\vare(t),$   $ \tiv_{\vare}^\prime(t,\cdot),$ $\tiq_{\vare}^\prime(t,\cdot),$ $\tiG_{\vare}^\prime(t,\cdot)) $ and $ (\tiOmega(t),$  $\tiv(t,\cdot),$  $\tiq(t,\cdot)$, $\tiG(t,\cdot))$ by shifting the solution $ (\tiOmega_\vare(t), \tiv_{\vare}^\prime(t,\cdot),$  $ \tiq_{\vare}^\prime(t,\cdot), \tiG_{\vare}^\prime(t,\cdot))  $ to $ (\tiOmega(t),$ $\tiv_{\vare}(t,\cdot),$ $\tiq_{\vare}(t,\cdot),$ $\tiG_{\vare}(t,\cdot))  $ which is also defined in $ \tiOmega(t) $ and we get $ \operatorname{dist}(\partial\tiOmega_{\vare}(\bar{t}), \partial\tiOmega(\bar{t})) $  $\le C \vare$. Thus, we have $ \operatorname{dist}(P^{-1}(\partial\tiOmega_{\vare}(\bar{t})),$  $ P^{-1} (\partial\tiOmega(\bar{t}))) \le C\vare $ and so $ P^{-1}(\partial\tiOmega_{\vare}(\bar{t})) $ self-intersects as in Fig. \ref{fig:5}.
	\begin{figure}[h]
	\centering
	\includegraphics[width=0.8\linewidth]{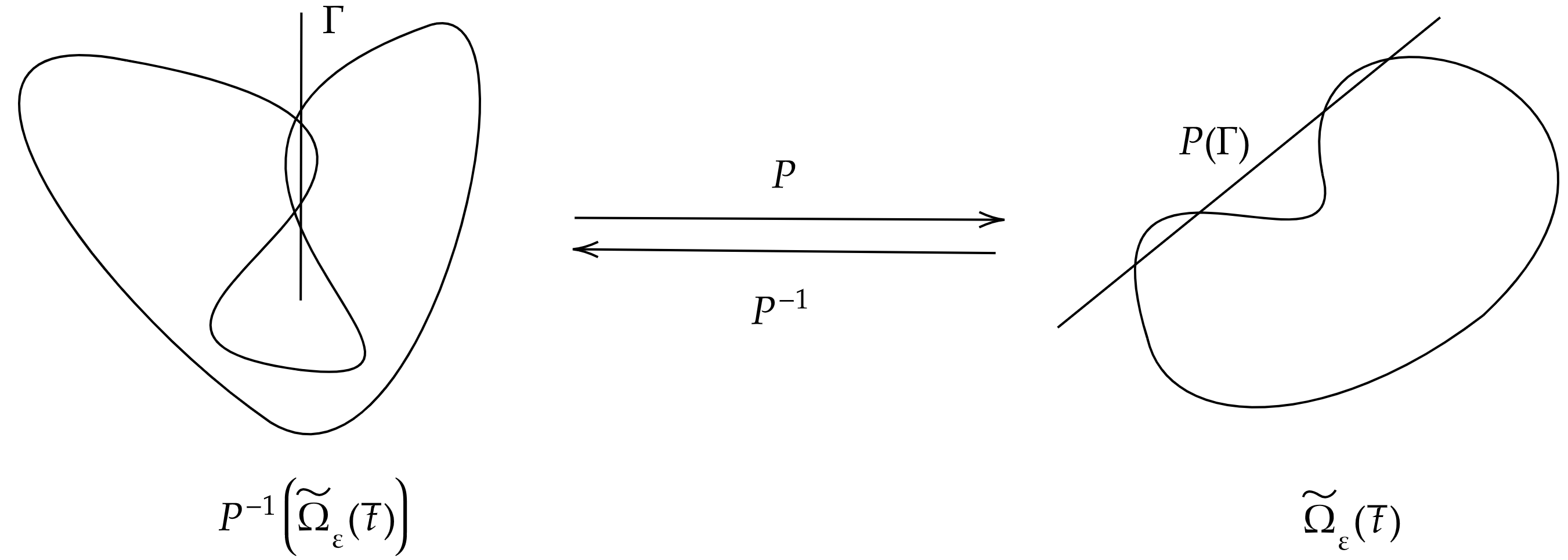}
	\caption{ Self-intersecting domain $ P^{-1}(\tiOmega_{\vare}(\bar{t})). $ }
	\label{fig:5}
\end{figure}
	\item Since $ P^{-1}(\partial\tiOmega_{\vare}(0)) $ is regular and $ P^{-1}(\partial\tiOmega_{\vare}(\bar{t})) $ is self-intersecting, there exists a time $ t^*_\vare $ such that $ P^{-1}(\partial\tiOmega_{\vare}(t^*_\vare)) $ has a splash singularity as in Fig. \ref{fig:6}.
	\begin{figure}[h]
		\centering
		\includegraphics[width=0.8\linewidth]{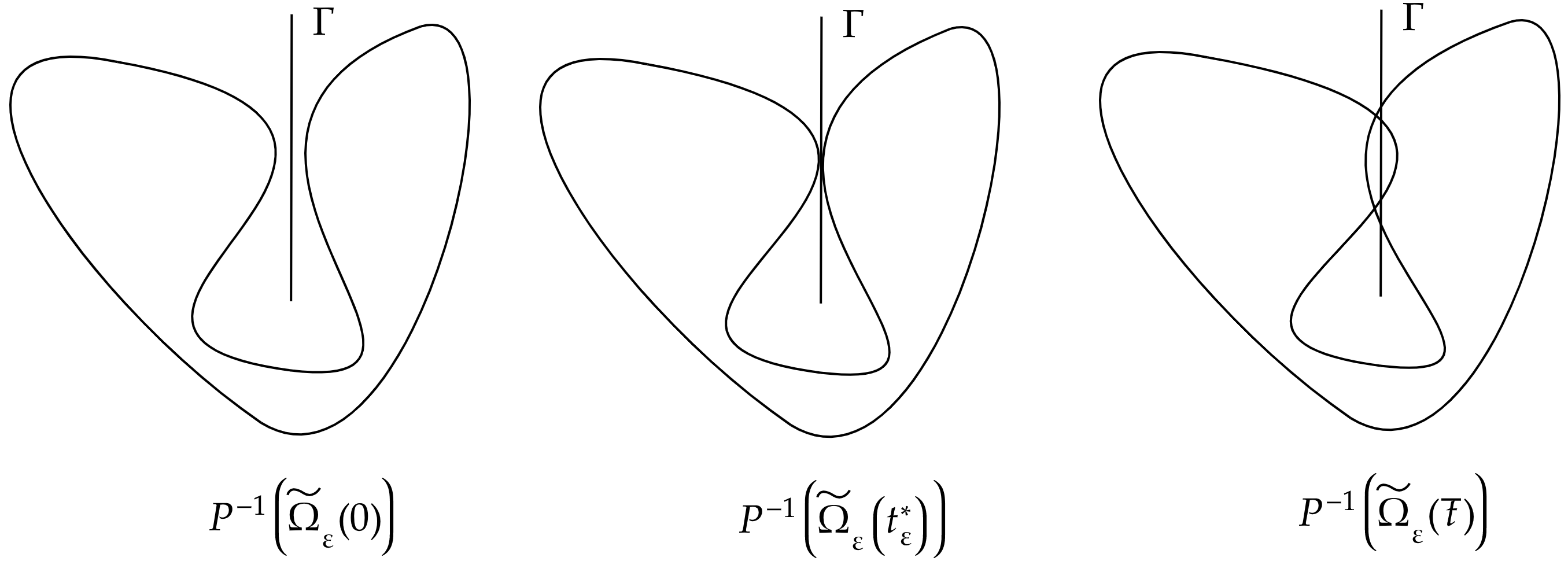}
		\caption{The existence of a plash singularity.}
		\label{fig:6}
	\end{figure}
\end{itemize}


The paper is organized as follows. In Section \ref{Function spaces and preliminary lemmas}, we introduce the function spaces for the estimates.  In Section \ref{Coordinate transformations}, we define all the variables and deal with the transformations from $ \Omega(t) $ into a fixed domain $ \tiOmega_0 $ by using the conformal map and the Lagrangian coordinates. In Section \ref{Local existence of smooth solutions}, we solve the viscous MHD model by constructing an iterative Cauchy sequence for $ T>0 $  small enough. To  show that the iterative sequence is Cauchy, we prove some estimates of the flux and the magnetic field with the help of lemmas given in the appendix. This local existence result is one of the main ingredients for proving the existence of a splash singularity. In Section \ref{Stability estimates}, we show the stability estimate which is another key ingredient. In Section \ref{Existence of splash singularity}, we construct a suitable initial velocity according to the initial magnetic field and other initial data such that the splash occurs. Therefore, we  obtain a finite-time splash singularity for the MHD model even in the presence of the kinematic viscosity. In the appendix, we list some estimates and  key lemmas for readers to consult.

\section{Beale spaces}\label{Function spaces and preliminary lemmas}
We introduce the Beale spaces which we will use later.
\par As in the classical paper \cite{beale1981initial}, the space $ H^s_{(0)}([0,T]) $ is defined as the interpolation between $ L^2([0,T]) $ and $ H^1_{(0)}([0,T])  $  for $ 0<s<1 .$ The operator used to interpolate is $ S=1-\partial^2_t $ with domain $ D(S)=\{ u\in H^2[0,T]:u(0)=\partial_tu(T)=0 \}. $ The details for this construction can be found in \cite{beale1981initial}. 

The norm of this space can be explicitly formulated. Note that the operator $ S $ has eigenvectors $\big\{\sin (\frac{(2 n+1) \pi}{2 T} t) \sqrt{\frac{2}{T}}\big\}_{n=0}^{\infty}$ with eigenvalues $ \big\{1+\frac{(2 n+1)^2 \pi^2}{4 T^2} \big\}_{n=0}^{\infty} $ which is also an orthogonal basis of $L^2([0, T]).$  Then, the space  $H_{(0)}^s([0, T])$ consists of functions $u \in L^2([0, T])$ such that
\begin{align*}
		\|v\|_{H_{(0)}^s}^2=\sum_{n=0}^{\infty} (v_n^s )^2 (\frac{(2 n+1) \pi}{2 T} )^{2 s}<\infty, 
\end{align*}
where
\begin{align*}
	v_n^s=\int_0^T  v(t) \sin  (\frac{(2 n+1) \pi}{2 T} t ) \sqrt{\frac{2}{T}} d t.
\end{align*}
For $s>\frac{1}{2},$  $v \in H_{(0)}^s([0, T])$ implies $v(0)=0$.
 
For  $ m=1,2,3, \ldots$ and $0 \leq s<1, $  the space $H_{(0)}^{m+s}([0, T])  $ with larger exponent is regarded as the subspace of $ \{ u \in H^m([0, T]): (\partial_t^k u )(0)=0, k=0, \ldots, m-1 \}$ with $\partial_t^m u \in H_{(0)}^s([0, T])$. We equip this space with the norm for fractional derivatives in time which we will use in this paper
\begin{align*}
	\|v\|_{H_{(0)}^{m+s}([0, T])}^2=\|v\|_{L^2([0, T])}^2+ \|\partial_t v \|_{L^2([0, T])}^2+\ldots+ \|\partial_t^m v \|_{H_{(0)}^s([0, T])}^2 .
\end{align*}

We also introduce the space $H^s([0, T])$ for $0<s<1$, which is defined as the interpolation of $H^1([0, T])$ and $L^2([0, T])$ with $S=1-\partial_t^2$ and domain $D(S)= \{ u \in H^2([0, T]): (\partial_t u )(0)= (\partial_t u )(T)=0 \}$. In this case, $ S $ has eigenvectors 
$$ \bigg\{\frac{1}{\sqrt{T}}, \big\{\cos  (\frac{n \pi}{T} t ) \sqrt{\frac{2}{T}} \big\}_{n=1}^{\infty} \bigg\}$$
 with eigenvalues $ \big\{1+\frac{n^2 \pi^2}{T^2} \big\}_{n=0}^{\infty}$ which  is a basis of $L^2([0, T]).$  Thus, we define
\begin{align*}
	\|u\|_{H^s([0, T])}^2=\sum_{n=0}^{\infty} \left(1+\frac{n^2 \pi^2}{T^2} \right)^s (u_n^c )^2,
\end{align*}
where
\begin{align*}
	u_0^c=\int_0^T  \frac{u(t)}{\sqrt{T}} d t, \quad u_n^c=\int_0^T  u(t) \cos  \left(\frac{n \pi}{T} t \right) \sqrt{\frac{2}{T}} d t, \quad n \geq 1.
\end{align*}
 For larger exponents, we regard $H^{m+s}([0, T]), m=1,2,3, \ldots,$ and $ 0<s<1$ as the subspace of $H^m([0, T])$ with $\partial_t^m u \in$ $H^s([0, T])$. It happens that $H_{(0)}^{m+s}([0, T])= \{u \in H^{m+s}([0, T]): (\partial_t^k u)(0)= $ $0, k=0,1, \ldots, m\}$, for $s>\frac{1}{2}$ and $H_{(0)}^{m+s}([0, T])= \{ u \in H^{m+s}([0, T]): (\partial_t^k u )(0)= $ $0, k=0,1, \ldots, m-1\}$, for $s<\frac{1}{2}$.

Once we have defined the spaces $ H^s $, we introduce the spaces we will use to solve
the free boundary MHD equations
\begin{align*}
	\mathcal{K}_{(0)}^s([0, T] ; \Omega)&:=L^2([0, T] ; H^s(\Omega)) \cap H_{(0)}^{\frac{s}{2}}([0, T] ; L^2(\Omega))\nonumber \\
	\mathcal{K}_{p r(0)}^s([0, T] ; \Omega)&:=\big\{q \in L^{\infty}([0, T] ; \dot{H}^1(\Omega)): \nabla q \in \mathcal{K}_{(0)}^{s-1}([0, T] ; \Omega),\\
	&\qquad\qquad\qquad\qquad\qquad\qquad  q \in \mathcal{K}_{(0)}^{s-\frac{1}{2}}([0, T] ; \partial \Omega)\big\},\nonumber\\ 
	\overline{\mathcal{K}}_{(0)}^s([0, T] ; \Omega)&:=L^2([0, T] ; H^s(\Omega)) \cap H_{(0)}^{\frac{s+1}{2}}([0, T] ; H^{-1}(\Omega)),\nonumber \\
	\mathcal{A}^{s, \gamma}([0, T] ; \Omega)&:=L_{\frac{1}{4}}^{\infty}([0, T] ; H^{s}(\Omega)) \cap H_{(0)}^2([0, T] ; H^\gamma(\Omega)), \nonumber
\end{align*}
for $ \vare>0 $  small enough and $ \gamma $ satisfying $ s-1-\vare<\gamma<s-1, $
with 
\begin{align*}
	\|u\|_{L^\infty_{\frac 14}H^s}:=\sup_{t\in[0,T]}t^{-\frac 14}\|u(t)\|_{H^s}.
\end{align*}
These spaces are important for the use of the embedding theorems and
interpolation estimates to get constants independent of time,  see
\cite{beale1981initial} and \cite{lions1972hilbert} for details.

\section{Coordinate transformations}\label{Coordinate transformations}

In the two-dimensional free boundary  incompressible viscous MHD model \eqref{mhd},
$ \Omega_0 $ is a domain in $ \mathbb{R}^2$;   $\Omega(t)=X(t, \Omega_0)$ moves according to the flux $ X $ associated to the velocity, which solves
\begin{align*}
	\begin{cases}
		\dfrac{d}{d t} X(t, \omega)=u(t, X(t, \omega)), \\
		X(0, \omega)=\omega, \quad \text{ in } \Omega_0.
	\end{cases}
\end{align*}

\par 
We apply the conformal map $P$ and   change the coordinates from $\Omega(t)  $  to  $ \tiOmega(t)=P(\Omega(t))$. 

\begin{remark}
Let $ \omega $ and $ \tiomega $ denote the typical points in $ \Omega_0 $ and $ \tiOmega_0 $, respectively. $ \alpha $ and $ \beta $ denote the typical points in $ \Omega(t) $ and $ \tiOmega(t), $ respectively.
$     H=(H^1, H^2)^\top , u=(u^1, u^2)^\top, X=(X^1, X^2)^\top  $ and $ \nabla p=(\partial_1 p, \partial_2 p)^\top $ are column vectors. The gradient of $ u $ is defined by
\begin{align*}
	\nabla u= 
	\left( 
	\begin{matrix}
	\partial_1u^1	& \partial_2u^1  \\
	\partial_1u^2	& \partial_2u^2
	\end{matrix} 
	\right),
\end{align*}
and $ \nabla u^i=(\partial_1 u^i, \partial_2 u^i) $ is a row vector. 
\end{remark}

The transformed velocity field, pressure and magnetic field are defined by
\begin{align*}
	\begin{cases}
		\tilde{u}(t,\beta):=u(t,P^{-1}(\beta)),\\
		\tilde{p}(t,\beta):=p(t,P^{-1}(\beta)),\\
		\tilde{H}(t,\beta):=H(t,P^{-1}(\beta)),
	\end{cases}	
\end{align*}
hence
\begin{align*}
	\begin{cases}
		\tilde{u}(t,P(\alpha))=u(t,\alpha),\\
		\tilde{p}(t,P(\alpha))=p(t,\alpha),\\
		\tilde{H}(t,P(\alpha))=H(t,\alpha).
	\end{cases}	
\end{align*}

Define
\begin{align*}
	J_{kj}(\cdot):=(\partial_jP^k)\circ P^{-1}(\cdot),
\end{align*}
and for the derivatives of the transformed velocity field, we have
\begin{align*}
	\sum_{k=1}^{2}J_{kj}(\beta)(\partial_k\tilde{u}^i)(t,\beta)=(\partial_ju^i)(t,P^{-1}(\beta)),
\end{align*}
i.e.,
\begin{align*}
	\sum_{k=1}^{2}J_{kj} \partial_k\tilde{u}^i=(\partial_ju^i) \circ P^{-1}.
\end{align*}

For the transformation in conformal coordinates, we have the following lemma.
\begin{lemma}
	Let $ P $ be the conformal map described above, $ J=(J_{ij}) $ and $Q^2(\cdot)=\left|\dfrac{dP}{dz}\circ P^{-1}(\cdot)\right|^2 $. Under this transformation,  system \eqref{mhd} becomes
	\begin{equation}\label{MHD conformal coordinate}
		\begin{cases}
			\partial_t \tilde{u}+ \nabla \tilde{u} J \tilde{u} - Q^2\laplace \tilde{u}+ J^\top \nabla \tilde{p}=\nabla \tilde{H} J \tilde{H}, &  \text { in } \tiOmega(t), \\
			\partial_t \tilde{H} + \nabla \tilde{H} J \tilde{u} =\nabla \tilde{u} J \tilde{H}, & \text { in } \tiOmega(t), \\ 
			\trace(\nabla \tilde{u} J) =0, \quad
			\trace(\nabla \tilde{H} J) =0, & 
			\text { in } \tiOmega(t), \\ 
			(-\tilde{p} \id+ (\nabla \tilde{u} J+(\nabla \tilde{u}J)^\top )+ \tilde{H}\tilde{H}^\top ) J^{-1} \tilde{n}=0, & 
			\text { on } \partial \tiOmega(t), \\ 
			\tilde{u}(t)(0, \cdot)=\tilde{u}_0, \quad \tilde{H}(t)(0, \cdot)=\tilde{H}_0, & 
			\text { in } \tiOmega
			_0,
		\end{cases}
	\end{equation}
	or in components
	\begin{align*}
		\begin{cases}\displaystyle
			\partial_t \tilde{u}^i+\sum_{k,j=1}^{2} \partial_k\tilde{u}^i J_{kj} \tilde{u}^j - Q^2\laplace \tilde{u}^i+\sum_{k=1}^{2}  \partial_k\tilde{p}J_{ki}=\sum_{k,j=1}^{2}\partial_k\tilde{H}^i J_{kj} \tilde{H}^j, &  \text { in } \tiOmega(t), \\
			\displaystyle	\partial_t \tilde{H}^i +\sum_{k,j=1}^{2} \partial_k\tilde{H}^i J_{kj} \tilde{u}^j =\sum_{k,j=1}^{2}\partial_k\tilde{u}^i J_{kj} \tilde{H}^j, & \text { in } \tiOmega(t), \\ 
			\displaystyle \sum_{k,i=1}^{2}\partial_k\tilde{u}^i J_{ki} =0, \quad
			\sum_{k,i=1}^{2}\partial_k\tilde{H}^i J_{ki} =0, & 
			\text { in } \tiOmega(t), \\ 
			\displaystyle \left(-\tilde{p} \id+ (\nabla \tilde{u} J+(\nabla \tilde{u}J)^\top )+ \tilde{H}\tilde{H}^\top \right) J^{-1} \tilde{n}=0, & 
			\text { on } \partial \tiOmega(t), \\ 
		\displaystyle	\tilde{u}(t)(0, \cdot)=\tilde{u}_0, \quad \tilde{H}(t)(0, \cdot)=\tilde{H}_0, & 
			\text { in } \tiOmega
			_0,
		\end{cases}
	\end{align*}
	 where $ \tilde{n}=-\Lambda J|_{\partial\tiOmega(t)}\Lambda n $ and $ \Lambda=\left( \begin{array}{cc}
		0	&  -1\\
		1	&  0
	\end{array}\right). $
\end{lemma}
\begin{proof}
	By direct calculation, $ u\cdot \nabla H$ becomes $ \nabla \tiH J\tilde{u}; $  $\nabla p$  becomes $ J^\top\nabla\tilde{p};$  $ \laplace u $ becomes $ Q^2\laplace \tilde{u} $ by using that $ P=P_1+i P_2 $ satisfies the Cauchy-Riemann equations
	\begin{align*}
		\begin{cases}
			\partial_1 P_1=\partial_2 P_2,\\
			\partial_2 P_1=-\partial_1 P_2.
		\end{cases}
	\end{align*}
The other terms are calculated similarly.
\end{proof}

Now, we pass from Eulerian to Lagrangian coordinates. For $ \tiomega\in \tiOmega_0, $ we define $ \tiX(t, \tiomega)\in \tiOmega(t) $ by
\begin{align*}
	\begin{cases}
		\dfrac{d}{d t} \tiX(t, \tiomega)=J(\tiX(t,\tiomega))\tilde{u}(t, \tiX(t, \tiomega)), & \text { in } \tiOmega(t), \\ 
		\tiX(0, \tiomega)=\tiomega,& \text { in } \tiOmega(0), \\ 
	\end{cases}
\end{align*}
i.e.,
\begin{align*}
	\begin{cases}
		\dfrac{d}{d t} \tiX^k(t, \tiomega)=J_{kj}(\tiX(t,\tiomega))\tilde{u}^j(t, \tiX(t, \tiomega)),& \text { in } \tiOmega(t), \\
		\tiX^k(0, \tiomega)=\tiomega^k,& \text { in } \tiOmega(0).
	\end{cases}
\end{align*}
The Lagrangian variables are defined as  
\begin{equation}\label{Lagrangian coordinate}
	\begin{cases}
		\tiv(t,\tiomega):=\tilde{u}(t,\tiX(t,\tiomega)),\\
		\tiq(t,\tiomega):=\tilde{p}(t,\tiX(t,\tiomega)),\\
		\tiG(t,\tiomega):=\tilde{H}(t,\tiX(t,\tiomega)).
	\end{cases}
\end{equation}

Define
\begin{align*}
	\left( \begin{matrix}
	\tizeta_{11}	& \tizeta_{12}  \\
	\tizeta_{21}	& \tizeta_{22}
	\end{matrix}\right)(t,\tiomega):= 
	\left( \begin{matrix}
		\partial_1\tiX^1
			& 	\partial_2\tiX^1  \\
			\partial_1\tiX^2	& 	\partial_2\tiX^2
	\end{matrix}\right)^{-1}(t,\tiomega),
\end{align*}
and for the transformation in Lagrangian coordinates, we have the following lemma.
\begin{lemma}
	Under the change of coordinates \eqref{Lagrangian coordinate}, system \eqref{MHD conformal coordinate} becomes
	\begin{equation}\label{MHD Lagrangian coordinate}
		\begin{cases} \displaystyle
			\partial_t \tiv^i - Q^2(\tiX)\sum_{j,k,m=1}^{2}\partial_m(\partial_k\tiv^i\tizeta_{kj})\tizeta_{mj} +\sum_{k,l=1}^{2} \partial_l\tiq\tizeta_{lk}J_{ki}(\tiX)\\
			\displaystyle\quad\quad=\sum_{j,k,l=1}^{2}\partial_l \tiG^i \tizeta_{lk}J_{kj}(\tiX) \tiG^j, &  \text { in } \tiOmega_0, \\
			\displaystyle\partial_t \tiG^i =\sum_{j,k,l=1}^{2}\partial_j \tiv^i \tizeta_{jk} J_{kl}(\tiX) \tiG^l, & \text { in } \tiOmega_0, \\
			\trace(\nabla \tiv \tizeta J(\tiX)) =0, \quad
			\trace(\nabla \tiG \tizeta J(\tiX)) =0, & 
			\text { in } \tiOmega_0, \\ 
			(-\tiq \id+ (\nabla \tiv \tizeta J(\tiX)+(\nabla \tiv \tizeta J(\tiX))^\top )+ \tiG\tiG^\top ) (J(\tiX))^{-1} \nabla_\Lambda\tiX\tilde{n}_0=0, & 
			\text { on } \partial \tiOmega_0, \\ 
			\tiv(t)(0, \cdot)=\tiv_0, \quad \tiG(t)(0, \cdot)=\tiG_0, & 
			\text { in } \tiOmega_0,
		\end{cases}
	\end{equation}
 	on the fixed domain $ [0,T]\times \tiOmega_0 $ where $ \nabla_{\Lambda}\tiX=-\Lambda\nabla \tiX \Lambda $ with $ \Lambda=\left( \begin{array}{cc}
 	0	&  -1\\
 	1	&  0
 	\end{array}\right). $
\end{lemma}
\begin{proof}
	By calculating the derivative of both ends of \eqref{Lagrangian coordinate} and with the help of $ \tizeta $, we obtain immediately the system in Lagrangian coordinates.
\end{proof}

We define 
\begin{align*}
	(\nabla(\nabla \tiv \tizeta)\tizeta)^i:=\sum_{j,k,m=1}^{2}\partial_m(\partial_k\tiv^i\tizeta_{kj})\tizeta_{mj}.
\end{align*}
For this expression, we can not simply write it in the form of matrix multiplication, so we adopt this notation. Note that the notation $ (\nabla(\nabla \tiv \tizeta)\tizeta)^i $  does not represent the usual matrix multiplication. Formally,  all we can  tell from it is that  the first-order derivatives of $ \tizeta, \tiv $ and the second-order derivatives of $ \tiv $ are involved, which is enough for us to estimate.

Used this notation, system \eqref{MHD Lagrangian coordinate} becomes
\begin{equation}\label{MHD Lagrangian coordinate2}
	\begin{cases}
		\partial_t \tiv - Q^2(\tiX)\nabla(\nabla \tiv \tizeta)\tizeta + J(\tiX)^\top \tizeta^\top \nabla\tiq=\nabla\tiG\tizeta J(\tiX)\tiG, &  \text {in } \tiOmega_0, \\
			\partial_t \tiG =\nabla \tiv \tizeta J(\tiX) \tiG, & \text {in } \tiOmega_0, \\ 
		\trace(\nabla \tiv \tizeta J(\tiX)) =0, \quad 
		\trace(\nabla \tiG \tizeta J(\tiX)) =0, & 
		\text {in } \tiOmega_0, \\ 
		(-\tiq \id+ (\nabla \tiv \tizeta J(\tiX)+(\nabla \tiv \tizeta J(\tiX))^\top )+ \tiG\tiG^\top ) (J(\tiX))^{-1} \nabla_\Lambda\tiX \tilde{n}_0=0, & 
		\text {on } \partial \tiOmega_0, \\ 
		\tiv(t)(0, \cdot)=\tiv_0, \quad \tiG(t(0, \cdot)=\tiG_0, & 
		\text {in } \tiOmega_0.
	\end{cases}
\end{equation}

\section{Local existence of smooth solutions for system \eqref{MHD Lagrangian coordinate2}}\label{Local existence of smooth solutions}

\par 
To prove a local existence result for  system \eqref{MHD Lagrangian coordinate2}, we define the iterative sequence $ \{ (\tiX^{(n)},\tiv^{(n)},\tiq^{(n)},$  $\tiG^{(n)}):n=0,1,\cdots \} $ and we show that $ \{ (\tiX^{(n)}-\hat{X},\tiv^{(n)}-\phi,\tiq^{(n)}-\tiq_\phi,\tiG^{(n)}-\hat{G}):n=0,1,\cdots \} $ is Cauchy in  suitable spaces if $ T>0 $ is sufficiently small. Here we have introduced
\begin{align*}
	 \hat{X}=\tiomega+t J \tiv_0,\quad \hat{G}=\tiG_0+t  \nabla \tiv_0 J \tiG_0,
\end{align*}  
and the construction of $ (\phi, \tiq_\phi) $ will be clear in the following discussion.

We separate the iteration for $ (\tiv, \tiq) $ from the iteration for $ \tiX $ and $ \tiG $. For $ (\tiv, \tiq), $ we have
\begin{equation}\label{iterative equation of v}
	\begin{cases}
		\partial_t \tiv^{(n+1)}-Q^2 \Delta \tiv^{(n+1)}+ J^\top \nabla \tiq^{(n+1)}=\tif^{(n)}, & \text{ in }  (0,T)\times\tiOmega_0, \\
		\trace (\nabla \tiv^{(n+1)} J )=\tig^{(n)}, & \text{ in }  (0,T)\times\tiOmega_0, \\
		{ [-\tiq^{(n+1)} \id+ ( (\nabla \tiv^{(n+1)} J  )+ (\nabla \tiv^{(n+1)} J)^\top  ) ] J^{-1} \tilde{n}_0=\tih^{(n)}}, & \text{ on }  (0,T)\times\partial\tiOmega_0, \\
		\tiv(0, \tiomega)=\tiv_0(\tiomega),& \text{ on }  \{ t=0  \}\times\tiOmega_0,
	\end{cases}
\end{equation}
where $\tif^{(n)}, \tig^{(n)} $ and  $ \tih^{(n)}$ collect all the nonlinear terms at the $n$-th step, namely
\begin{align}
		\tif^{(n)}= & -Q^2 \Delta \tiv^{(n)}+ J^\top \nabla \tiq^{(n)}+Q^2(\tiX^{(n)}) \nabla(\nabla\tiv^{(n)}\tizeta^{(n)})\tizeta^{(n)}\nonumber \\
		&-J(\tiX^{(n)})^\top  \tizeta^{(n)\top} \nabla \tiq^{(n)} +\nabla \tiG^{(n)}\tizeta^{(n)} J(\tiX^{(n)})\tiG^{(n)},\label{tifn} \\
		\tig^{(n)}= & \trace (\nabla \tiv^{(n)} J)-\trace (\nabla \tiv^{(n)} \tizeta^{(n)} J (\tiX^{(n)} ),\label{tign}\\
		\tih^{(n)}= & -\tiq^{(n)}J^{-1} \tilde{n}_0+\tiq^{(n)}(J(\tiX^{(n)}))^{-1} \nabla_{\Lambda} \tiX^{(n)} \tilde{n}_0 \nonumber\\
		&+((\nabla \tiv^{(n)} J) +(\nabla \tiv^{(n)} J)^\top )J^{-1} \tilde{n}_0\nonumber \\
		& - ( (\nabla \tiv^{(n)} \tizeta^{(n)} J (\tiX^{(n)} ) )+ (\nabla \tiv^{(n)} \tizeta^{(n)} J (\tiX^{(n)} ) )^\top  ) J (\tiX^{(n)} ) ^{-1} \nabla_{\Lambda} \tiX^{(n)} \tilde{n}_0\nonumber\\
		&-\tiG^{(n)}\tiG^{(n)\top}J(\tiX^{(n)})^{-1} \nabla_{\Lambda} \tiX^{(n)} \tilde{n}_0.\label{tihn}
\end{align}
For the magnetic field $ \tiG, $ we consider the following ODE
\begin{equation}\label{iterative equation of G}
	\begin{cases}
		\partial_t \tiG^{(n+1)}(t, \tiomega)=\nabla \tiv^{(n)}(t, \tiomega) \tizeta^{(n)}(\tiomega) J(\tiX^{(n)}(t, \tiomega) )   \tiG^{(n)}(t,\tiomega), \\
		\tiG(0, \tiomega)=\tiG_0(\tiomega), \quad \text { in } \tiOmega_0 .
	\end{cases}
\end{equation}
Finally, the flux $ \tiX $ satisfies
\begin{equation}\label{iterative equation of X}
	\begin{cases}
		\dfrac{d}{d t} \tiX^{(n+1)}(t, \tiomega)=J  (\tiX^{(n)}(t, \tiomega) ) \tiv^{(n)}(t, \tiomega), \\ 
		\tiX(0, \tiomega)=\tiomega, \quad \text { in } \tiOmega_0 .
	\end{cases}
\end{equation}

We study  systems \eqref{iterative equation of v}, \eqref{iterative equation of G} and \eqref{iterative equation of X} separately. 
Similar to the form of   linear system \eqref{iterative equation of v}, we consider the following system with zero initial velocity
\begin{equation}\label{system lemma 4.1}
	\begin{cases}
		\partial_t \tiv-Q^2 \Delta \tiv+ J^\top \nabla \tiq=\tif, & \text{ in }  (0,T)\times\tiOmega_0, \\
		\trace (\nabla \tiv J )=\tig, & \text{ in }  (0,T)\times\tiOmega_0, \\
		 [-\tiq \id+ ( (\nabla \tiv J  )+ (\nabla \tiv J)^\top  ) ] J^{-1} \tilde{n}_0=\tih, & \text{ on }  (0,T)\times\partial\tiOmega_0, \\
		\tiv(0, \tiomega)=0, & \text{ in }  \{ t=0 \}\times\tiOmega_0,
	\end{cases}
\end{equation}
and the  initial data satisfy the compatibility conditions
\begin{equation}\label{compatibility lemma4.1}
	\begin{cases}
		\trace (\nabla \tiv_0 J )=\tig(0), & \text{ in }   \tiOmega_0, \\
		(J^{-1} \tilde{n}_0)^{\perp} [ ( (\nabla \tiv J  )+ (\nabla \tiv J)^\top  ) ] J^{-1} \tilde{n}_0=\tih(0) (J^{-1} \tilde{n}_0)^{\perp}, & \text{ on }   \partial\tiOmega_0. \\
	\end{cases}
\end{equation}

We define the following functional space of the solution 
\begin{align*}
	X_0:=\{ (\tiv, \tiq)\in \mathcal{K}_{(0)}^{s+1}([0,T];\tiOmega_0)\times\mathcal{K}_{pr(0)}^s ([0,T];\tiOmega_0)\},
\end{align*}
the functional space of data $  (\tif, \tig, \tih)  $
\begin{align*}
	Y_0:=\{ (\tif, \tig, \tih)\in
	 \mathcal{K}_{(0)}^{s-1}([0,T];\tiOmega_0)\times \bar{\mathcal{K}}_{(0)}^{s}([0,T];\tiOmega_0)\times&\mathcal{K}_{(0)}^{s-\frac 12}([0,T];\partial\tiOmega_0) :\\
	 &\quad\eqref{compatibility lemma4.1} \text{ are satisfied} \},
\end{align*}
and a linear operator  $ L:X_0\to Y_0, $ related to  system \eqref{system lemma 4.1} by
\begin{equation}\label{L}
	L(\tiv, \tiq)=(\tif, \tig, \tih, 0).
\end{equation}

For system \eqref{system lemma 4.1}, we recall the following result:
\begin{lemma}\label{lem4.1}(\cite[Theorem 4.1]{castro2019splash}).
For system \eqref{system lemma 4.1}, 	the operator $ L $ defined in \eqref{L} is invertible for $ 2<s<\frac 52. $ Moreover, $ \|L^{-1}\| $ is bounded uniformly if $ T $ is bounded above and the following estimate holds
\begin{align*}
	\|(\tiv, \tiq)\|_{X_0}\le C\|(\tif, \tig, \tih)\|_{Y_0}.
\end{align*} 
\end{lemma} 
\par To apply  Lemma \ref{lem4.1}, we need a modification of the velocity and the pressure such that the modified terms belong to $ X_0 $ and the modified initial velocity is zero. For this reason, we extend the analysis for the choice of the approximated solution  made in \cite[Section 5]{castro2019splash}. For our problem, we  need to choose $\phi$   such  that $  \tiv^{(n)}(0)=\tiv_0=\phi(0) $ and $  (\partial_t \tiv^{(n)})(0)=(\partial_t \phi)(0)$, for all $ n\ge 0$. 

Indeed, if we set $ t=0 $ in \eqref{MHD Lagrangian coordinate2}, for $ \tiv,  $ we have
\begin{align*}
	(\partial_t \tiv)(0) - Q^2\laplace\tiv_0 + J^\top  \nabla\tiq(0)=\nabla\tiG_0 J\tiG_0 , 
\end{align*}
in  $ \tiOmega_0.$ Therefore, $ \phi $  should satisfy
\begin{align*}
	(\partial_t \phi)(0)= Q^2\laplace\tiv_0 - J^\top  \nabla\tiq(0)+\nabla\tiG_0 J\tiG_0 , \quad \phi(0)=\tiv_0.
\end{align*}
A reasonable choice of $ \phi $ is
\begin{align*}
	\phi(t)=\tiv_0+t(   Q^2\laplace\tiv_0 - J^\top  \nabla\tiq(0)+\nabla\tiG_0 J\tiG_0).
\end{align*}
Note that in the above expression, we need to specify $ \nabla \tiq (0) $ and we need to investigate the iterative system for $ \tiv $ and $ \tiq $ in the original domain $ \Omega_0. $

From the coordinate transformation through the inverse of the conformal mapping $ P, $ we obtain the iterative system  of the velocity and the pressure in $ \Omega_0. $ More precisely, we define
\begin{align*}
	\begin{cases}
		v^{(n)}(t,P^{-1}(\tiomega)):=\tiv^{(n)}(t,\tiomega),\\
		q^{(n)}(t,P^{-1}(\tiomega)):=\tiq^{(n)}(t,\tiomega),\\
		f^{(n)}(t,P^{-1}(\tiomega)):=\tif^{(n)}(t,\tiomega),\\
		g^{(n)}(t,P^{-1}(\tiomega)):=\tig^{(n)}(t,\tiomega),\\
		h^{(n)}(t,P^{-1}(\tiomega)):=\tih^{(n)}(t,\tiomega),
	\end{cases}
\end{align*} 
then we have
\begin{equation*}
	\begin{cases}
		\partial_t v^{(n+1)}- \laplace v^{(n+1)}+ \nabla q^{(n+1)}=f^{(n)}, & \text{ in }  (0,T)\times\Omega_0, \\
		\dive v^{(n+1)}=g^{(n)}, & \text{ in }  (0,T)\times\Omega_0, \\
		{ [-q^{(n+1)} \id+ ( (\nabla v^{(n+1)} )+ (\nabla v^{(n+1)} )^\top  ) ] {n}_0= h^{(n)}}, & \text{ on }  (0,T)\times\partial\Omega_0, \\
		v(0,\omega)=v_0(\omega),& \text{ on }  \{ t=0  \}\times\Omega_0.
	\end{cases}
\end{equation*}
Note that
\begin{align}
	\partial_t g^{(n)}&=\partial_t \dive v^{(n+1)}\nonumber\\
	&=\laplace g^{(n)}-\laplace q^{(n+1)}+\dive f^{(n)},\nonumber
\end{align}
therefore, for $ t=0, $ we obtain
\begin{equation}\label{equq0}
	\begin{cases}
		-\laplace q^{(n+1)}(0)=(\partial_t g^{(n)})(0)-\laplace g^{(n)}(0)-\dive f^{(n)}(0), &\text{ in } \Omega_0,\\
		q^{(n+1)}(0)n_0=(\nabla v_0+\nabla v_0^\top)n_0-h^{(n)} n_0, &\text{ on } \partial\Omega_0.
	\end{cases}
\end{equation}

From \eqref{tifn}, \eqref{tign} and  \eqref{tihn},   we obtain 
\begin{align*}
	\begin{cases}
		f^{(n)}(0)=\nabla G_0 G_0,\\
		g^{(n)}(0)=0,\\
		h^{(n)}(0)=-G_0G_0^\top n_0,\\		
		\dive f^{(n)}(0)=\trace(\nabla G_0 \nabla G_0),
	\end{cases}
\end{align*}
where we have used $  G_0=H_0$ and $ \dive H_0=0. $
Note that $ \partial_t \tizeta^{(n)}=-\tizeta^{(n)}  \nabla \partial_t\tiX^{(n)}\tizeta^{(n)}, $ we obtain from \eqref{tign} that
\begin{align*}
	(\partial_t \tig^{(n)})(0)=\trace(\nabla\tiv_0J\nabla\tiv_0J)
\end{align*}
by direct calculation,  therefore,
\begin{align*}
	(\partial_t  g^{(n)})(0)=\trace(\nabla v_0 \nabla v_0 ). 
\end{align*}
We conclude that $ q^{(n+1)}(0) $ satisfies
\begin{align*}
	\begin{cases}
		-\laplace q^{(n+1)}(0)=\trace(\nabla v_0 \nabla v_0 )-\trace(\nabla G_0\nabla G_0), &\text{ in } \Omega_0,\\
		q^{(n+1)}(0) n_0=(\nabla v_0+\nabla v_0^\top +G_0 G_0^\top)n_0, &\text{ on } \partial\Omega_0,
	\end{cases}
\end{align*}
which is independent of  $ n $.

Back to the domain $ \tiOmega_0, $ we obtain
\begin{align*}
	\begin{cases}
		-Q^2 \Delta \tiq^{(n+1)}(0)=\trace (\nabla \tiv_0 J  \nabla \tiv_0 J  )-\trace (\nabla \tiG_0 J  \nabla \tiG_0 J  ), & \text { in } \tiOmega_0, \\ 
		\tiq^{(n+1)}(0) J^{-1} \tilde{n}_0=( \nabla \tiv_0 J+( \nabla \tiv_0 J)^\top +\tiG_0 \tiG_0^\top )J^{-1} \tilde{n}_0, & \text { on }  \partial \tiOmega_0,
	\end{cases}
\end{align*}
for all $ n. $
Therefore,   $ \phi $  and  $\tiq_\phi $ are well-defined and satisfy
\begin{align*}
	\phi:=\tiv_0+t \hat{\phi}:=\tiv_0+t (Q^2 \Delta \tiv_0- J^\top \nabla \tiq_\phi+\nabla \tiG_0 J \tiG_0  ),
\end{align*}
and
\begin{align*}
	\begin{cases}
		-Q^2 \Delta \tiq_\phi=\trace (\nabla \tiv_0 J  \nabla \tiv_0 J  )-\trace (\nabla \tiG_0 J  \nabla \tiG_0 J  ), & \text { in } \tiOmega_0, \\ 
		\tiq_\phi J^{-1} \tilde{n}_0=( \nabla \tiv_0 J+( \nabla \tiv_0 J)^\top +\tiG_0 \tiG_0^\top )J^{-1} \tilde{n}_0, & \text { on }  \partial \tiOmega_0.
	\end{cases}
\end{align*}

Once we have defined $ \phi  $ and $ \tiq_\phi $, we define the modified velocity $ \tiw^{(n)}  $ and the modified pressure $ \tiq_w^{(n)} $ by 
\begin{align*}
	\begin{cases}
		\tiw^{(n)}:=\tiv^{(n)}-\phi,\\
		\tiq_w^{(n)}:=\tiq^{(n)}-\tiq_\phi.
	\end{cases}
\end{align*}
We write  system \eqref{iterative equation of v} in terms of $ \tiw^{(n)} $ and $ \tiq_w^{(n)} $ as the following 
\begin{equation}\label{equ4.9}
	\begin{cases}
		\partial_t \tiw^{(n+1)}-Q^2 \Delta \tiw^{(n+1)}+ J^\top \nabla \tiq_w^{(n+1)}\\
		 \quad =\tif^{(n)}-\partial_t \phi+Q^2 \Delta \phi- J^\top \nabla \tiq_\phi, & \text{in }  (0,T)\times\tiOmega_0, \\
		\trace (\nabla \tiw^{(n+1)} J )=\tig^{(n)}-\trace(\nabla \phi J), & \text{in }  (0,T)\times\tiOmega_0, \\
		 [-\tiq_w^{(n+1)} \id+ ((\nabla \tiw^{(n+1)} J )+ (\nabla \tiw^{(n+1)} J )^\top  )] J^{-1} \tilde{n}_0\\
		 \quad =\tih^{(n)}+\tiq_\phi J^{-1} \tilde{n}_0 - ( (\nabla \phi J )+ (\nabla \phi J )^\top  ) J^{-1} \tilde{n}_0, & \text{on }  (0,T)\times\partial\tiOmega_0, \\
		\tiw^{(n+1)}(0,\cdot)=0, & \text{in }  \{ t=0 \}\times\tiOmega_0,
	\end{cases}
\end{equation}
where
\begin{align*}
		\tif^{(n)}= & -Q^2 \Delta (\tiw^{(n)}+\phi)+ J^\top \nabla (\tiq_w^{(n)}+\tiq_\phi)+Q^2(\tiX^{(n)}) \nabla(\nabla(\tiw^{(n)}+\phi)\tizeta^{(n)})\tizeta^{(n)}\nonumber \\
		& -J(\tiX^{(n)})^\top  \tizeta^{(n)\top} \nabla (\tiq_w^{(n)}+\tiq_\phi)+\nabla \tiG^{(n)}\tizeta^{(n)} J(\tiX^{(n)})\tiG^{(n)}, \nonumber\\
		\tig^{(n)}= & \trace (\nabla (\tiw^{(n)}+\phi) J)-\trace (\nabla (\tiw^{(n)}+\phi) \tizeta^{(n)} J (\tiX^{(n)} ),\nonumber \\
		\tih^{(n)}= & -(\tiq_w^{(n)}+\tiq_\phi)J^{-1} \tilde{n}_0+(\tiq_w^{(n)}+\tiq_\phi)(J(\tiX^{(n)}))^{-1} \nabla_{\Lambda} \tiX^{(n)} \tilde{n}_0 \nonumber\\
		&+((\nabla (\tiw^{(n)}+\phi) J) +(\nabla (\tiw^{(n)}+\phi) J)^\top )J^{-1} \tilde{n}_0\\
		& - \big( (\nabla (\tiw^{(n)}+\phi) \tizeta^{(n)} J (\tiX^{(n)} ) )\nonumber \\
		&\qquad + (\nabla (\tiw^{(n)}+\phi) \tizeta^{(n)} J (\tiX^{(n)} ) )^\top  \big) J (\tiX^{(n)} ) ^{-1} \nabla_{\Lambda} \tiX^{(n)} \tilde{n}_0\nonumber\\
		&-\tiG^{(n)}\tiG^{(n)\top}J(\tiX^{(n)})^{-1} \nabla_{\Lambda} \tiX^{(n)} \tilde{n}_0.\nonumber
\end{align*}
Furthermore, for the magnetic field, we have
\begin{align}\label{equ4.10}
		&\tiG^{(n+1)}(t, \tiomega)\nonumber \\ =&\tiG_0+\int_0^t\nabla(\tiw^{(n)}(\tau,\tiomega)+\phi(\tau,\tiomega))\tizeta^{(n)}(\tiomega)J(\tiX^{(n)}(\tau,\tiomega))\tiG^{(n)}(\tau,\tiomega)d\tau,
\end{align}
and for the flux, we obtain
\begin{equation}\label{equ4.11}
	\tiX^{(n+1)}(t, \tiomega)=\tiomega+\int_0^t J(\tiX^{(n)}(\tau,\tiomega))(\tiw^{(n)}(\tau,\tiomega)+\phi(\tau,\tiomega))d\tau .
\end{equation}

To prove the  local existence theorem, we need to prove the following estimates of the flux and the magnetic field.
\begin{lemma}\label{x-omega} 	Let $2<s<\frac{5}{2}, 1<\gamma< s-1, \delta,\mu>0 $ small enough and $ \tilde{X}-\tiomega-\int_{0}^{t}J \nabla \phi d\tau\in \mathcal{A}^{s+1,\gamma+1} $. Then, for $ T>0 $ small enough we have
	\begin{align}
		\|\tilde{X}-\tiomega\|_{L^{\infty} H^{s+1}}  &\leq T^{\frac{1}{4}}\left\|\tilde{X}-\tiomega-\int_{0}^{t}J \nabla \phi d\tau \right\|_{\mathcal{A}^{s+1, \gamma+1}}+C(\tilde{v}_0, \tiG_0) T,\nonumber\\
		\|\tilde{X}-\tiomega\|_{L_{\frac 14}^{\infty} H^{s+1}}  &\leq \left\|\tilde{X}-\tiomega-\int_{0}^{t}J \nabla \phi d\tau \right\|_{\mathcal{A}^{s+1, \gamma+1}}+C(\tilde{v}_0, \tiG_0) T^{\frac 34},\nonumber\\
		\|\tilde{X}-\tiomega\|_{H_{(0)}^1 H^{\gamma+1}}  &\leq C T^\delta \left\|\tilde{X}-\tiomega-\int_{0}^{t}J \nabla \phi d\tau \right\|_{\mathcal{A}^{s+1,\gamma+1}}+C(\tilde{v}_0, \tiG_0) T^{\frac{1}{2}}, \nonumber\\
		\|\tilde{X}-\tiomega\|_{H_{(0)}^{\frac{s-1}{2}} H^{2+\mu}} 
		&\leq C T^\delta \left\|\tilde{X}-\tiomega-\int_{0}^{t}J \nabla \phi d\tau \right\|_{\mathcal{A}^{s+1,\gamma+1}}+C(\tilde{v}_0, \tiG_0) T^{\frac{1}{2}},\nonumber\\
		\|\tilde{X}-\tiomega\|_{H_{(0)}^{\frac{s}{2}-\frac 14} H^{2+\mu}} 
		&\leq C T^\delta \left\|\tilde{X}-\tiomega-\int_{0}^{t}J \nabla \phi d\tau \right\|_{\mathcal{A}^{s+1,\gamma+1}}+C(\tilde{v}_0, \tiG_0) T^{\frac{1}{8}},\nonumber\\
		\|\tilde{X}-\hat{X}\|_{ \mathcal{A}^{s+1,\gamma+1} } 
		&\leq  \left\|\tilde{X}-\tiomega-\int_{0}^{t}J \nabla \phi d\tau \right\|_{ \mathcal{A}^{s+1,\gamma+1} }+C(\tilde{v}_0, \tiG_0) T^{\frac{1}{2}},\nonumber\\
		\|\tiX\|_{L^\infty H^{s+1}}&\le T^{\frac{1}{4}} \left\|\tilde{X}-\tiomega-\int_{0}^{t}J \nabla \phi d\tau \right\|_{\mathcal{A}^{s+1, \gamma+1}}+ C(\tiv_0,\tiG_0).\nonumber
	\end{align}
\end{lemma}
\begin{proof}
	Without loss of generality, we may	assume $ T<1. $ We have
	\begin{align*}
		\|\tilde{X}-\tiomega\|_{L^{\infty} H^{s+1}} & \leq \left\| \tilde{X}-\tiomega-\int_{0}^{t}J \nabla \phi d\tau  \right\|_{L^{\infty} H^{s+1}}+ \left\|\int_{0}^{t}J \nabla \phi d\tau \right\|_{L^{\infty} H^{s+1}}\nonumber \\
		& \leq T^{\frac{1}{4}} \left\|\tilde{X}-\tiomega-\int_{0}^{t}J \nabla \phi  d\tau \right\|_{L^\infty_{\frac{1}{4}} H^{s+1}}+ \left\|\int_{0}^{t}J (\nabla \tiv_0+\tau\nabla \hat{\phi}) d\tau \right\|_{L^{\infty} H^{s+1}}\nonumber\\
		& \leq  T^{\frac{1}{4}} \left\|\tilde{X}-\tiomega-\int_{0}^{t}J \nabla \phi d\tau \right\|_{\mathcal{A}^{s+1, \gamma+1}}+C(\tilde{v}_0, \tiG_0) T,\nonumber
	\end{align*}
and
	\begin{align*}
		\|\tilde{X}-\tiomega\|_{L_{\frac 14}^{\infty} H^{s+1}} & \leq \left\|\tilde{X}-\tiomega-\int_{0}^{t}J \nabla \phi  d\tau \right\|_{L_{\frac 14}^{\infty} H^{s+1}}+ \left\|\int_{0}^{t}J \nabla \phi  d\tau \right\|_{L_{\frac 14}^{\infty} H^{s+1}}\nonumber \\
		& \leq \left\|\tilde{X}-\tiomega-\int_{0}^{t}J \nabla \phi d\tau \right\|_{L_{\frac 14}^{\infty} H^{s+1}}+ \left\|\int_{0}^{t}J (\nabla \tiv_0+\tau\nabla \hat{\phi}) d\tau \right\|_{L_{\frac 14}^{\infty} H^{s+1}}\nonumber\\
		& \leq  \left\|\tilde{X}-\tiomega-\int_{0}^{t}J \nabla \phi d\tau \right\|_{\mathcal{A}^{s+1, \gamma+1}}+C(\tilde{v}_0, \tiG_0) T^\frac34 .\nonumber
	\end{align*}
	
For some $ \delta>0 $ small enough and $ \frac 12<\eta<1  $, we have from Lemmas \ref{lem3.3} and \ref{lem3.7} that
	\begin{align}
		&\|\tilde{X}-\tiomega\|_{H_{(0)}^1 H^{\gamma+1}}\nonumber\\
		\leq&  \left\|\tilde{X}-\tiomega-\int_{0}^{t}J \nabla \phi d\tau \right\|_{H_{(0)}^1 H^{\gamma+1}}+ \left\|\int_{0}^{t}J \nabla \phi  d\tau \right\|_{H_{(0)}^1 H^{\gamma+1}}\nonumber \\
		\leq&  \left\|\int_0^t \partial_\tau (\tilde{X}-\tiomega-\int_{0}^{\tau}J \nabla \phi d\xi ) d\tau  \right\|_{H_{(0)}^{1+\eta-\delta} H^{\gamma+1}}+C(\tilde{v}_0)\|t\|_{H_{(0)}^1}+C(\tilde{v}_0, \tiG_0)\|t^2\|_{H_{(0)}^1} \nonumber\\
		\leq & C T^\delta \left\|\tilde{X}-\tiomega-\int_{0}^{t}J \nabla \phi d\tau \right\|_{H_{(0)}^{1+\eta} H^{\gamma+1}}+C(\tilde{v}_0, \tiG_0) T^{\frac{1}{2}}\nonumber \\
		\leq&  C T^\delta \left\|\tilde{X}-\tiomega-\int_{0}^{t}J \nabla \phi d\tau \right\|_{\mathcal{A}^{s+1,\gamma+1}}+C(\tilde{v}_0, \tiG_0) T^{\frac{1}{2}},\nonumber
	\end{align}
and	
	\begin{align*}
		&\|\tilde{X}-\tiomega\|_{H_{(0)}^{\frac{s-1}{2}} H^{2+\mu}}\nonumber\\
		\leq& \left\|\tilde{X}-\tiomega-\int_{0}^{t}J \nabla \phi d\tau \right\|_{H_{(0)}^{\frac{s-1}{2}} H^{2+\mu}}+ \left\|\int_{0}^{t}J \nabla \phi d\tau \right\|_{H_{(0)}^{\frac{s-1}{2}} H^{2+\mu}}\nonumber \\
		\leq & \left\|\int_0^t \partial_\tau (\tilde{X}-\tiomega-\int_{0}^{\tau}J \nabla \phi  d\xi ) d\tau  \right\|_{H_{(0)}^{\frac{s-1}{2}+\delta-\delta} H^{2+\mu}}\\
		&+C(\tilde{v}_0)\|t\|_{H_{(0)}^{\frac{s-1}{2}}}+C(\tilde{v}_0, \tiG_0)\|t^2\|_{H_{(0)}^{\frac{s-1}{2}}}\nonumber \\
		\leq& C T^\delta \left\|\tilde{X}-\tiomega-\int_{0}^{t}J \nabla \phi  d\tau \right\|_{H_{(0)}^{\frac{s-1}{2}+\delta}H^{2+\mu}}+C(\tilde{v}_0, \tiG_0) T^{\frac{1}{2}}\nonumber \\
		\leq& C T^\delta \left\|\tilde{X}-\tiomega-\int_{0}^{t}J \nabla \phi d\tau \right\|_{\mathcal{A}^{s+1,\gamma+1}}+C(\tilde{v}_0, \tiG_0) T^{\frac{1}{2}} .\nonumber
	\end{align*}
	Similarly, we have  
	\begin{align*}
		  \|\tilde{X}-\tiomega\|_{H_{(0)}^{\frac{s}{2}-\frac 14} H^{2+\mu}}\leq C T^\delta \left\|\tilde{X}-\tiomega-\int_{0}^{t}J \nabla \phi d\tau \right\|_{\mathcal{A}^{s+1,\gamma+1}}+C(\tilde{v}_0, \tiG_0) T^{\frac{1}{8}}, 
	\end{align*}
	\begin{align*}
		\|\tilde{X}-\hat{X}\|_{ \mathcal{A}^{s+1,\gamma+1} } & \leq \left\|\tilde{X}-\tiomega-\int_{0}^{t}J \nabla \phi d\tau \right\|_{ \mathcal{A}^{s+1,\gamma+1} }+ \left\|\int_{0}^{t} \tau J \nabla \hat{\phi} d\tau \right\|_{ \mathcal{A}^{s+1,\gamma+1} }\nonumber \\
		&\leq \left\|\tilde{X}-\tiomega-\int_{0}^{t}J \nabla \phi d\tau\right\|_{ \mathcal{A}^{s+1,\gamma+1} }+C(\tilde{v}_0, \tiG_0) T^{\frac{1}{2}} ,\nonumber
	\end{align*}
and
	\begin{align*}
		\|\tiX\|_{L^\infty H^{s+1}}\le & \|\tiX-\tiomega\|_{L^\infty H^{s+1}}+\|\tiomega\|_{L^\infty H^{s+1}}\nonumber\\
		\le &T^{\frac{1}{4}} \left\|\tilde{X}-\tiomega-\int_{0}^{t}J \nabla \phi d\tau \right\|_{\mathcal{A}^{s+1, \gamma+1}}+ C(\tiv_0,\tiG_0).
	\end{align*}
\end{proof}

\begin{lemma}\label{G-G0} 	Let $2<s<\frac{5}{2}, 1<\gamma< s-1,  \delta,\mu>0 $ small enough and $ \tiG -\tiG_0-\int_{0}^{t}\nabla \phi J \tiG_0d\tau \in \mathcal{A}^{s,\gamma} $. Then, for $ T>0 $ small enough we have
	\begin{align}
		\|\tiG -\tiG_0\|_{L^\infty H^s} &\leq T^{\frac{1}{4}} \left\|\tiG -\tiG_0-\int_{0}^{t}\nabla \phi J \tiG_0 d\tau \right\|_{\mathcal{A}^{s,\gamma}}+ C(\tiv_0,\tiG_0)T,\nonumber\\
		\|\tiG -\tiG_0\|_{L_{\frac14}^\infty H^s} &\leq \left\|\tiG -\tiG_0-\int_{0}^{t}\nabla \phi J \tiG_0 d\tau \right\|_{\mathcal{A}^{s,\gamma}}+ C(\tiv_0,\tiG_0)T^{\frac 34},\nonumber\\
		\|\tiG -\tiG_0\|_{H^1_{(0)}H^\gamma}  &\leq C T^\delta \left\|\tiG -\tiG_0-\int_{0}^{t}\nabla \phi J \tiG_0 d\tau  \right\|_{\mathcal{A}^{s,\gamma}}+C(\tilde{v}_0, \tiG_0) T^{\frac{1}{2}},\nonumber \\
		\|\tiG -\tiG_0\|_{H^{\frac{s-1}{2}}_{(0)}H^{1+\mu}} 
		&\leq C T^\delta \left\|\tiG -\tiG_0-\int_{0}^{t}\nabla \phi J \tiG_0 d\tau  \right\|_{\mathcal{A}^{s,\gamma}}+C(\tilde{v}_0, \tiG_0) T^{\frac{1}{2}},\nonumber\\
		\|\tiG-\tiG_0\|_{H_{(0)}^{\frac{s}{2}-\frac 14}H^{1+\mu}}  &\leq C T^\delta \left\|\tiG -\tiG_0-\int_{0}^{t}\nabla \phi J \tiG_0 d\tau \right\|_{\mathcal{A}^{s,\gamma}}+C(\tilde{v}_0, \tiG_0) T^{\frac 18},\nonumber\\
		\|\tilde{G}-\hat{G}\|_{\mathcal{A}^{s,\gamma} } &\leq  \left\|\tilde{G}-\tiG_0-\int_{0}^{t}\nabla \phi J \tiG_0 d\tau \right\|_{ \mathcal{A}^{s,\gamma} }+C(\tilde{v}_0, \tiG_0) T^{\frac{1}{2}},\nonumber \\
		\|\tiG\|_{L^\infty H^s} & 
		\le T^{\frac{1}{4}} \left\|\tiG-\tiG_0-\int_{0}^{t}\nabla \phi J \tiG_0 d\tau \right\|_{\mathcal{A}^{s,\gamma}}+ C(\tiv_0,\tiG_0).\nonumber
	\end{align}
\end{lemma}
\begin{proof}
	Without loss of generality, we may	assume $ T<1 $. We obtain
	\begin{align}
		\|\tiG -\tiG_0\|_{L^\infty H^s}=& \left\|\tiG -\tiG_0-\int_{0}^{t}\nabla \phi J \tiG_0 d\tau  \right\|_{L^\infty H^s}+\left\|\int_{0}^{t}\nabla \phi J \tiG_0 d\tau \right\|_{L^\infty H^s}\nonumber\\
		\le &T^{\frac{1}{4}} \left\|\tiG -\tiG_0-\int_{0}^{t}\nabla \phi J \tiG_0 d\tau \right\|_{L_{\frac 14}^\infty H^s}+\left\|\int_{0}^{t}\nabla \phi J \tiG_0 d\tau \right\|_{L^\infty H^s}\nonumber\\
		\le &T^{\frac{1}{4}} \left\|\tiG -\tiG_0-\int_{0}^{t}\nabla \phi J \tiG_0 d\tau \right\|_{\mathcal{A}^{s,\gamma}}+ C(\tiv_0,\tiG_0)T.\nonumber
	\end{align}
Similarly, we have
\begin{align*}
		\|\tiG -\tiG_0\|_{L_{\frac14}^\infty H^s} \leq \left\|\tiG -\tiG_0-\int_{0}^{t}\nabla \phi J \tiG_0 d\tau \right\|_{\mathcal{A}^{s,\gamma}}+ C(\tiv_0,\tiG_0)T^{\frac 34}.
\end{align*}
	
	For some $ \delta>0 $ small enough and $ \frac 12<\eta<1  $, we have from Lemmas \ref{lem3.3} and \ref{lem3.7} that
	\begin{align*}
		\|\tiG -\tiG_0\|_{H^1_{(0)}H^\gamma}=& \left\|\tiG -\tiG_0-\int_{0}^{t}\nabla \phi J \tiG_0d\tau \right\|_{H^1_{(0)}H^\gamma}+\left\|\int_{0}^{t}\nabla \phi J \tiG_0d\tau  \right\|_{H^1_{(0)}H^\gamma}\nonumber\\
		\le &  \left\|\int_0^t \partial_\tau (\tiG -\tiG_0-\int_{0}^{\tau}\nabla \phi J \tiG_0 d\xi ) d\tau  \right\|_{H_{(0)}^{1+\eta-\delta} H^{\gamma}}\\
		&+C(\tilde{v}_0, \tiG_0)(\|t\|_{H_{(0)}^1}+\|t^2\|_{H_{(0)}^1}) \nonumber\\
		\leq &C T^\delta \left\|\tiG -\tiG_0-\int_{0}^{t}\nabla \phi J \tiG_0 d\tau \right\|_{H_{(0)}^{1+\eta} H^{\gamma}}+C(\tilde{v}_0, \tiG_0) T^{\frac{1}{2}}\nonumber \\
		\leq &C T^\delta \left\|\tiG -\tiG_0-\int_{0}^{t}\nabla \phi J \tiG_0 d\tau \right\|_{\mathcal{A}^{s,\gamma}}+C(\tilde{v}_0, \tiG_0) T^{\frac{1}{2}},\nonumber 
	\end{align*}
and 
	\begin{align*}
		&\|\tiG -\tiG_0\|_{H^{\frac{s-1}{2}}_{(0)}H^{1+\mu}}\nonumber\\
		=& \left\|\tiG -\tiG_0-\int_{0}^{t}\nabla \phi J \tiG_0 d\tau \right\|_{H^{\frac{s-1}{2}}_{(0)}H^{1+\mu}}+\left \|\int_{0}^{t}\nabla \phi J \tiG_0 d\tau  \right\|_{H^{\frac{s-1}{2}}_{(0)}H^{1+\mu}}\nonumber\\
		\le &  \left\|\int_0^t \partial_\tau (\tiG -\tiG_0-\int_{0}^{\tau}\nabla \phi J \tiG_0 d\xi ) d\tau \right\|_{H^{\frac{s-1}{2}+\delta-\delta}_{(0)}H^{1+\mu}}\\
		&+C(\tilde{v}_0, \tiG_0)(\|t\|_{H_{(0)}^1}+\|t^2\|_{H_{(0)}^1})\nonumber \\
		\leq &C T^\delta \left\|\tiG -\tiG_0-\int_{0}^{t}\nabla \phi J \tiG_0 d\tau \right\|_{H_{(0)}^{\frac{s-1}{2}+\delta} H^{1+\mu}}+C(\tilde{v}_0, \tiG_0) T^{\frac{1}{2}}\nonumber \\
		\leq &C T^\delta \left\|\tiG -\tiG_0-\int_{0}^{t}\nabla \phi J \tiG_0 d\tau \right\|_{\mathcal{A}^{s,\gamma}}+C(\tilde{v}_0, \tiG_0) T^{\frac{1}{2}} .\nonumber
	\end{align*}
	Similarly, we have 
\begin{align*}
		 \|\tiG-\tiG_0\|_{H_{(0)}^{\frac{s}{2}-\frac 14}H^{1+\mu}} \leq C T^\delta \left\|\tiG -\tiG_0-\int_{0}^{t}\nabla \phi J \tiG_0 d\tau \right\|_{\mathcal{A}^{s,\gamma}}+C(\tilde{v}_0, \tiG_0) T^{\frac 18},   
\end{align*}
	\begin{align*}
		\|\tilde{G}-\hat{G}\|_{ \mathcal{A}^{s,\gamma} } & \leq \left\|\tilde{G}-\tiG_0-\int_{0}^{t}\nabla \phi J \tiG_0 d\tau \right\|_{ \mathcal{A}^{s,\gamma} }+ \left\|\int_{0}^{t} \tau  \nabla \hat{\phi} J\tiG_0 d\tau  \right\|_{ \mathcal{A}^{s,\gamma} }\nonumber \\
		&\leq  \left\|\tilde{G}-\tiG_0-\int_{0}^{t}\nabla \phi J \tiG_0 d\tau \right\|_{ \mathcal{A}^{s,\gamma} }+C(\tilde{v}_0, \tiG_0) T^{\frac{1}{2}} ,
	\end{align*}
and
	\begin{align*}
		\|\tiG\|_{L^\infty H^s}\le & \|(\tiG-\tiG_0)\|_{L^\infty H^s}+\|\tiG_0\|_{L^\infty H^s}\nonumber\\
		\le &T^{\frac{1}{4}} \left\|\tiG-\tiG_0-\int_{0}^{t}\nabla \phi J \tiG_0 d\tau \right\|_{\mathcal{A}^{s,\gamma}}+ C(\tiv_0,\tiG_0).
	\end{align*}
\end{proof}
\par
Note that we assume that 
\begin{equation*}
 \tiX-\tiomega-\int_{0}^{t}J\nabla\phi d\tau\in\mathcal{A}^{s+1,\gamma+1},\quad \tiG-\tiG_0-\int_{0}^{t}\nabla\phi J \tiG_0 d\tau\in\mathcal{A}^{s,\gamma}, 
\end{equation*}	
and  so we obtain that $ \tiX-\hat{X} $ and $ \tiG-\hat{G} $  belong to the same spaces  $ \mathcal{A}^{s+1,\gamma+1} $ and $ \mathcal{A}^{s,\gamma}, $  respectively, since $ \phi=\tiv_0+t\hat{\phi} $ and $ \partial_t(\int_0^t J\tau \nabla \hat{\phi}d\tau)|_{t=0}=0. $ However, $ \tiX-\tiomega $  does not belong to $ \mathcal{A}^{s+1,\gamma+1} $ because the initial value of its time derivative does not vanish. Similarly, $ \tiG-\tiG_0 $  does not belong to $ \mathcal{A}^{s,\gamma} $ either. Therefore, for the flux and the  magnetic field, we need to adopt these technical modifications to ensure that the modified quantities belong to suitable spaces. In this way, we are able to obtain the desired estimates  and prove the local existence theorem.

Now we prove the following local existence theorem.
\begin{theorem}	\label{thm1}
 Let $2<s<\frac{5}{2} $ and $ 1<\gamma<s-1$. If $\tiv(0)=\tiv_0$ and $\tiG(0)=\tiG_0$ are in $H^k (\tiOmega_0)$ for $k$ large enough,  there exist a sufficiently small $T>0$ and a solution $ (\tiX-\hat{X}, \tiw, \tiq_w, \tiG-\hat{G}) \in$ $\mathcal{A}^{s+1, \gamma+1} \times \mathcal{K}_{(0)}^{s+1} \times \mathcal{K}_{p r(0)}^s \times \mathcal{A}^{s, \gamma}$ in $(0, T] \times \tiOmega_0$ satisfying
\begin{align}
	&\left\| \tiX-\tiomega-\int_{0}^{t}J\nabla \phi d\tau \right\|_{\mathcal{A}^{s+1,\gamma+1}}\le N,\nonumber\\
	& \| (\tiw,\tiq_w)-L^{-1}( \tilde{f}_\phi^L+\tilde{f}_{G_0}, \bar{g}_\phi^L, \tilde{h}_\phi^L+\tilde{h}_{G_0}, 0 )\|_{\mathcal{K}^{s+1}_{(0)}\times \mathcal{K}^{s}_{pr(0)}}\le N,\nonumber\\
	&\left\| \tiG-\tiG_0-\int_{0}^{t}\nabla \phi J \tiG_0 d\tau \right\|_{\mathcal{A}^{s,\gamma}}\le N,\nonumber
\end{align}
where
\begin{align}
	N:=&  \left\|\int_{0}^{1} \tau J\nabla 
	\hat{\phi} d\tau\right\|_{\mathcal{A}^{s+1,\gamma+1}([0,1];\tiOmega_0)}+ \left\|\int_{0}^{1} \tau \nabla 
	\hat{\phi}J \tiG_0 d\tau \right\|_{\mathcal{A}^{s,\gamma}([0,1];\tiOmega_0)}\nonumber\\ &+\|L^{-1}(\tilde{f}_\phi^L+\tilde{f}_{G_0}, \bar{g}_\phi^L, \tilde{h}_\phi^L+\tilde{h}_{G_0}, 0)\|_{\mathcal{K}_{(0)}^{s+1}([0,1];\tiOmega_0)\times \mathcal{K}_{pr(0)}^s([0,1];\tiOmega_0)}.\nonumber 
\end{align}

\end{theorem}
\begin{remark}
	 In the definition of $ N, $ we have taken the integral from $ 0 $ to $ 1 $ since we merely need to prove the existence of solutions for  $ T>0 $   small enough. From now on,  we assume that $ T $ is less than $ 1 $ for simplicity.
\end{remark}

In this theorem we recall $ \hat{X}=\tiomega+t J \tiv_0 $ and $\hat{G}=\tiG_0+t  \nabla \tiv_0 J \tiG_0$. Combining these two terms, we have $\tiX-\hat{X}\in \mathcal{A}^{s+1, \gamma+1}$ and $\tiG-\hat{G}\in\mathcal{A}^{s, \gamma} $. Also, we have introduced some terms which will be defined later for technical modifications. To prove Theorem \ref{thm1}, we need to show the following estimates.
\begin{proposition}\label{pro1}
	Let $2<s<\frac{5}{2} $ and $ 1<\gamma<s-1.$ If $ ( (\tiw^{(0)}, \tiq^{(0)}_w), \tiX^{(0)}, \tiG^{(0)} )=((0,0), \hat{X}, \hat{G})   $ and $ N $ is defined as above, for $ T>0 $ small enough depending on $ N, \tiv_0 $ and $ \tiG_0 $, we have
\begin{align*}
	&\tiX^{(n)}-\hat{X}\in \mathcal{A}^{s+1,\gamma+1},\quad (\tiw^{(n)},\tiq_w^{(n)}) \in \mathcal{K}^{s+1}_{(0)}\times \mathcal{K}^{s}_{pr(0)},\quad \tiG^{(n)}-\hat{G}\in \mathcal{A}^{s,\gamma},\nonumber \\
	&\left\| \tiX^{(n)}-\tiomega-\int_{0}^{t}J\nabla \phi d\tau \right\|_{\mathcal{A}^{s+1,\gamma+1}}\le N,\nonumber\\
	&\| (\tiw^{(n)},\tiq_w^{(n)})-L^{-1}( \tilde{f}_\phi^L+\tilde{f}_{G_0}, \bar{g}_\phi^L, \tilde{h}_\phi^L+\tilde{h}_{G_0}, 0 )\|_{\mathcal{K}^{s+1}_{(0)}\times \mathcal{K}^{s}_{pr(0)}}\le N,\nonumber\\
	&\left\| \tiG^{(n)}-\tiG_0-\int_{0}^{t}\nabla \phi J \tiG_0 d\tau \right\|_{\mathcal{A}^{s,\gamma}}\le N,\nonumber
\end{align*}
where $ n\ge 0. $

Moreover, we have
\begin{align*}
		 &\|\tiw^{(n+1)}-\tiw^{(n)}\|_{\mathcal{K}_{(0)}^{s+1}}+\|\tiq_w^{(n+1)}-\tiq_w^{(n)}\|_{\mathcal{K}_{p r(0)}^s}+\|\tiX^{(n+1)}-\tiX^{(n)}\|_{\mathcal{A}^{s+1, \gamma+1}}\nonumber \\
		&+\|\tiG^{(n+1)}-\tiG^{(n)}\|_{\mathcal{A}^{s, \gamma}}\nonumber\\
		\le & C(N,\tiv_0,\tiG_0) T^{\delta_1}   \left( \|\tiw^{(n)}-\tiw^{(n-1)}\|_{\mathcal{K}_{(0)}^{s+1}}+\|\tiq_w^{(n)}-\tiq_w^{(n-1)}\|_{\mathcal{K}_{p r(0)}^s}\right) \nonumber \\			
		& +C(N,\tiv_0,\tiG_0) T^{\delta_2}\|\tiX^{(n)}-\tiX^{(n-1)}\|_{\mathcal{A}^{s+1, \gamma+1}}\\
		&+C(N,\tiv_0,\tiG_0) T^{\delta_3}\|\tiG^{(n)}-\tiG^{(n-1)}\|_{\mathcal{A}^{s, \gamma}}, \nonumber
\end{align*}
where $ \delta_i>0, i=1,2,3 $ and $ n\ge 1. $
In particular, $ ( (\tiw^{(n)}, \tiq^{(n)}_w), \tiX^{(n)}-\hat{X}, \tiG^{(n)}-\hat{G} ) $ is a Cauchy sequence  for $ T>0 $ small enough and the limit $ ( (\tiw, \tiq_w), \tiX-\hat{X}, \tiG-\hat{G} ) $ satisfies
\begin{align*}
	&\tiX-\hat{X}\in \mathcal{A}^{s+1,\gamma+1},\quad (\tiw,\tiq_w) \in \mathcal{K}^{s+1}_{(0)}\times \mathcal{K}^{s}_{pr(0)},\quad \tiG-\hat{G}\in \mathcal{A}^{s,\gamma},\nonumber \\
	&\left\| \tiX-\tiomega-\int_{0}^{t}J\nabla \phi d\tau \right\|_{\mathcal{A}^{s+1,\gamma+1}}\le N,\nonumber\\
	& \| (\tiw,\tiq_w)-L^{-1}( \tilde{f}_\phi^L+\tilde{f}_{G_0}, \bar{g}_\phi^L, \tilde{h}_\phi^L+\tilde{h}_{G_0}, 0 )\|_{\mathcal{K}^{s+1}_{(0)}\times \mathcal{K}^{s}_{pr(0)}}\le N,\nonumber\\
	&\left\| \tiG-\tiG_0-\int_{0}^{t}\nabla \phi J \tiG_0 d\tau \right\|_{\mathcal{A}^{s,\gamma}}\le N.\nonumber
\end{align*}
Therefore, system \eqref{MHD Lagrangian coordinate2} is solved.
\end{proposition}
\par  We will show that if the initial iterative sequence satisfies the conditions in the first part of the following lemma and  propositions, all subsequent quantities  belong to the predetermined balls uniformly for $ T>0 $ small enough. Furthermore,  the iterative sequence is  Cauchy  and its limit solves system \eqref{MHD Lagrangian coordinate2}. 

We start by writing system \eqref{equ4.9} in the following form
\begin{align*}
	&L(\tiw^{(n+1)}, \tiq^{(n+1)}_w)\nonumber\\
	=&\big( \tif^{(n)}-\partial_t \phi+Q^2 \Delta \phi- J^\top \nabla \tiq_\phi,\; \tig^{(n)}-\trace(\nabla \phi J),\;\\
	&\qquad \tih^{(n)}+\tiq_\phi J^{-1} \tilde{n}_0- ( (\nabla \phi J )+ (\nabla \phi J )^\top  ) J^{-1} \tilde{n}_0,\; 0 \big)\nonumber\\
	=&(\tilde{f}^{(n)}-\tilde{f}_{G_0}, \bar{g}^{(n)}, \tilde{h}^{(n)}-\tilde{h}_{G_0},0)+(\tilde{f}_\phi^L+\tilde{f}_{G_0}, \bar{g}_\phi^L, \tilde{h}_\phi^L+\tilde{h}_{G_0}, 0),\nonumber
\end{align*}
where
\begin{align}
		\tilde{f}_{G_0}&:=\nabla \tilde{G}_0 J \tilde{G}_0,\nonumber\\
		\tilde{h}_{G_0}&:=-\tilde{G}_0 \tilde{G}_0^\top  J^{-1} \tilde{n}_0,\nonumber\\
		\tilde{f}_\phi^L & :=-\partial_t \phi+Q^2 \Delta \phi- J^\top  \nabla \tilde{q}_\phi, \nonumber\\
		\bar{g}^{(n)} & :=\tilde{g}^{(n)}+\operatorname{Tr}(\nabla \phi \tilde{\zeta}_\phi J_\phi)-\operatorname{Tr}(\nabla \phi J), \nonumber\\
		\bar{g}_\phi^L & :=-\operatorname{Tr}(\nabla \phi \tilde{\zeta}_\phi J_\phi), \nonumber\\
		\tilde{h}_\phi^L & :=\tilde{q}_\phi J^{-1} \tilde{n}_0-(\nabla \phi J+(\nabla \phi J)^\top )J^{-1} \tilde{n}_0,\nonumber
\end{align}
with
\begin{align}
	\tilde{\zeta}_\phi(t, \cdot)&:=\id+t\left( \partial_\tau \tizeta^{(n)}(\tau,\cdot)|_{\tau=0}\right)=\mathcal{I}-t \nabla(J \tilde{v}_0)(\cdot),\nonumber\\ (J_\phi)_{i j}(t, \cdot)&:=J_{ij}(\cdot)+t\left(\partial_\tau (J_{ij}(\tiX^{(n)}(\tau ,\cdot)) )\big|_{\tau =0}\right)=J_{i j}+t\sum_{k,l=1}^{2} (\partial_k J_{i j} J_{k l} \tilde{v}_{0}^{l}) (\cdot).\nonumber
\end{align}  
In this way, we have
\begin{align}
	(\tilde{f}^{(n)}-\tilde{f}_{G_0}, \bar{g}^{(n)}, \tilde{h}^{(n)}-\tilde{h}_{G_0} )(0, \cdot)=&0,\nonumber\\  (\tilde{f}_\phi^L+\tilde{f}_{G_0}, \bar{g}_\phi^L, \tilde{h}_\phi^L+\tilde{h}_{G_0})(0, \cdot)=&0,\nonumber\\
	(\partial_t\bar{g}^{(n)})(0, \cdot)=&0,\nonumber\\
	(\partial_t\bar{g}_\phi^L)(0, \cdot)=&0,\nonumber\\ (\partial_t(\tig^{(n)}-\trace(\nabla \phi J) ))(0, \cdot)=&0.\nonumber
\end{align}
These technical modifications are essential since now we have $(\tilde{f}^{(n)}-\tilde{f}_{G_0}$, $\bar{g}^{(n)}$, $\tilde{h}^{(n)}-\tilde{h}_{G_0} ) \in Y_0$ and also $(\tilde{f}_\phi^L+\tilde{f}_{G_0}, \bar{g}_\phi^L, \tilde{h}_\phi^L+\tilde{h}_{G_0}) \in Y_0$, thus we can apply Lemma \ref{lem4.1} to get the independence of time for all constants.

\par  In order to prove Proposition \ref{pro1},  we define the following notations:
\begin{align}
	B(X):=&\left\{\tiX : \tiX -\hat{X}\in \mathcal{A}^{s+1,\gamma+1}, \left\| \tiX -\tiomega-\int_{0}^{t}J\nabla \phi d\tau \right\|_{\mathcal{A}^{s+1,\gamma+1}}\le N \right\},\nonumber\\
	B(w,q):=&\big\{(\tiw ,\tiq_w ): (\tiw ,\tiq_w ) \in \mathcal{K}^{s+1}_{(0)}\times \mathcal{K}^{s}_{pr(0)},\nonumber\\
	& \| (\tiw ,\tiq_w )-L^{-1}( \tilde{f}_\phi^L+\tilde{f}_{G_0},\bar{g}_\phi^L, \tilde{h}_\phi^L+\tilde{h}_{G_0}, 0 )\|_{\mathcal{K}^{s+1}_{(0)}\times \mathcal{K}^{s}_{pr(0)}}\le N \big\},\nonumber\\
	B(G):=&\left\{\tiG : \tiG -\hat{G}\in \mathcal{A}^{s,\gamma}, \left\| \tiG -\tiG_0-\int_{0}^{t}\nabla \phi J \tiG_0 d\tau \right\|_{\mathcal{A}^{s,\gamma}}\le N \right\},\nonumber
\end{align}
where $ N $ is defined as in Theorem \ref{thm1}.
\par 
We investigate the equations  separately.  For \eqref{equ4.11}, we use the result obtained in \cite[Proposition 5.3]{castro2019splash}.
\begin{lemma}\label{lem1}
	Let $ 2<s<\frac52, 1<\gamma<s-1.$ For $ T>0 $ small enough depending on $  N,\tiv_0, $  the following results hold:
	\begin{enumerate}[(1)]
		\item For $ n\ge 0, $ if $ \tiX^{(n)}\in B(X) $ and $ (\tiw^{(n)},\tiq_w^{(n)})\in B(w,q) $, then we have  $ \tiX^{(n+1)}\in B(X). $
		\item For $ n\ge 1, $ if $ \tiX^{(n-1)}, \tiX^{(n)}, \tiX^{(n+1)} \in B(X) $ and $ (\tiw^{(n-1)},\tiq_w^{(n-1)})$, $(\tiw^{(n)},\tiq_w^{(n)}), (\tiw^{(n+1)},$  $ \tiq_w^{(n+1)})$  $\in B(w,q)$, then we have the estimate   
		\begin{align*}
			&\|\tiX^{(n+1)}-\tiX^{(n)}\|_{\mathcal{A}^{s+1,\gamma+1}}\\
			\le& C(\tiv_0, N) T^\delta ( \|\tiw^{(n)}-\tiw^{(n-1)}\|_{\mathcal{K}^{s+1}_{(0)}} + \|\tiX^{(n)}-\tiX^{(n-1)}\|_{\mathcal{A}^{s+1,\gamma+1}})
		\end{align*}
		for some $ \delta>0$. 
	\end{enumerate}
\end{lemma}

For the magnetic field, we can show the following result.
\begin{proposition}\label{pro2}
	For $ 2<s<\frac52, 1<\gamma<s-1,$ and $ T>0 $ small enough depending on $  N,\tiv_0 $ and $ \tiG_0, $ we have
	\begin{enumerate}[(1)]
		\item For $ n\ge 0, $ if $ \tiX^{(n)}\in B(X), (\tiw^{(n)},\tiq_w^{(n)})\in B(w,q)  $ and $ \tiG^{(n)}\in B(G), $ then $ \tiG^{(n+1)}\in B(G). $
		\item For $ n\ge 1, $  if $ \tiX^{(n-1)}, \tiX^{(n)}, \tiX^{(n+1)} \in B(X), (\tiw^{(n-1)},\tiq_w^{(n-1)}), (\tiw^{(n)},\tiq_w^{(n)}), (\tiw^{(n+1)},$ $\tiq_w^{(n+1)})\in B(w,q) $ and $ \tiG^{(n-1)}, \tiG^{(n)}, \tiG^{(n+1)}\in B(G), $ then
		\begin{align*}
			\|\tiG^{(n+1)}-\tiG^{(n)}\|_{\mathcal{A}^{s,\gamma}}\le & C(N,\tiv_0,\tiG_0) T^\delta (\|\tiG^{(n)}-\tiG^{(n-1)}\|_{\mathcal{A}^{s,\gamma}}\nonumber\\
			&+ \|\tiw^{(n)}-\tiw^{(n-1)}\|_{\mathcal{K}^{s+1}_{(0)}} + \|\tiX^{(n)}-\tiX^{(n-1)}\|_{\mathcal{A}^{s+1,\gamma+1}})\nonumber
		\end{align*}
		for some $ \delta>0. $
	\end{enumerate}
\end{proposition}

\begin{proof}

\textbf{Part 1.}

To prove $ \tiG^{(n+1)}\in B(G), $ we  define $ I_1,I_2 $ and $ I_3 $ by	
\begin{align*}
			&\left\|\tilde{G}^{(n+1)}-\tilde{G}_0-\int_0^t  \nabla \phi J \tilde{G}_0 d\tau \right\|_{\mathcal{A}^{s, \gamma}} \\
			\le & \left\| \int_0^t  \nabla \tilde{w}^{(n)} \tilde{\zeta}^{(n)}  J(\tilde{X}^{(n)}) \tilde{G}^{(n)} d\tau \right\|_{\mathcal{A}^{s, \gamma}}\nonumber \\
			& \quad +\left\|\int_0^t \nabla \tiv_0( \tizeta^{(n)}J(\tiX^{(n)})\tiG^{(n)}-J \tiG_0 ) d\tau \right\|_{\mathcal{A}^{s, \gamma}}\\
			& \quad +\left\|\int_0^t \tau \nabla \hat{\phi}( \tizeta^{(n)}J(\tiX^{(n)})\tiG^{(n)}-J \tiG_0 ) d\tau\right\|_{\mathcal{A}^{s, \gamma}}\nonumber \\
			&=: \left\|I_1\right\|_{\mathcal{A}^{s, \gamma}}+\left\|I_2\right\|_{\mathcal{A}^{s, \gamma}}+\left\|I_3\right\|_{\mathcal{A}^{s, \gamma}} .\nonumber
\end{align*}
Recall $ \mathcal{A}^{s, \gamma} $ involves two functional spaces and we
start with the estimate in $L_{\frac{1}{4}}^{\infty} H^s$. To use the estimates we have proved and lemmas in the appendix, we need to split these terms properly. 
	
For the first term, we have
\begin{align*}
	\left\|I_1 \right\|_{L^\infty_{\frac 14}H^s}
	\le& \sup_{t\in [0,T]}t^{-\frac14} \int_{0}^{t}\left\|\nabla \tiw^{(n)} \tizeta^{(n)} J(\tiX^{(n)}) \tiG^{(n)} \right\|_{H^s} d\tau \nonumber\\
	\le & \sup_{t\in [0,T]}t^{-\frac14} t^{\frac12} \|\nabla \tiw^{(n)} \tizeta^{(n)} J(\tiX^{(n)}) \tiG^{(n)}  \|_{L^2([0,t]; H^s)}\nonumber\\
	\le & T^{\frac14}  \|\nabla \tiw^{(n)} \tizeta^{(n)} J(\tiX^{(n)}) \tiG^{(n)}\|_{L^2([0,T]; H^s)}\nonumber\\
	= & T^{\frac14}  \|\nabla \tiw^{(n)} \tizeta^{(n)} J(\tiX^{(n)}) (\tiG^{(n)}-\tiG_0)\|_{L^2H^s}+T^{\frac14} \|\nabla \tiw^{(n)} \tizeta^{(n)} J(\tiX^{(n)}) \tiG_0\|_{L^2H^s}\nonumber\\
	=&: I_{1,1}+ I_{1,2},\nonumber
\end{align*}
where we have used Minkowski's, Cauchy's inequalities and the definition of  $L_{\frac{1}{4}}^{\infty} H^s.$
We just show the estimate of $ I_{1,1} $ since $ \tiG_0 $  is the initial data in $ I_{1,2} $. Since $ \tiX^{(n)}\in B(X), (\tiw^{(n)},\tiq_w^{(n)})\in B(w,q), $ and $ \tiG^{(n)}\in B(G), $ we apply   Lemmas \ref{G-G0}, \ref{jx-j} and \ref{zeta-} to obtain
\begin{align*}
	I_{1,1}\le & T^{\frac14}  \|\nabla \tiw^{(n)}\|_{L^2H^s}
	 \| \tizeta^{(n)}\|_{L^\infty H^s}
	 \| J(\tiX^{(n)})\|_{L^\infty H^s}
	 \| \tiG^{(n)}-\tiG_0\|_{L^\infty H^s}\nonumber\\
	 \le& T^{\frac14} C(\tiv_0)  \| \tiw^{(n)}\|_{\mathcal{K}^{s+1}_{(0)}}
	 \| \tiG^{(n)}-\tiG_0\|_{L^\infty H^s}\nonumber\\
	 \le& T^{\frac12} C(N, \tiv_0,\tiG_0).\nonumber
\end{align*}
Similarly, we have $ I_{1,2}\le T^{\frac14} C(N, \tiv_0,\tiG_0)$ and therefore $ \sum_{i=1}^{2}I_{1,i}\le T^{\frac14} C(N, \tiv_0,\tiG_0).$

For the second term, we have
\begin{align*}
	\left\|I_2 \right\|_{L^\infty_{\frac 14}H^s}
	\le& \sup_{t\in [0,T]}t^{-\frac14} \int_{0}^{t}\left\|\nabla \tiv_0( \tizeta^{(n)}J(\tiX^{(n)})\tiG^{(n)}-J \tiG_0 ) \right\|_{H^s} d\tau \nonumber\\
	\le & \sup_{t\in [0,T]}t^{-\frac14} t^{\frac12} \| \nabla \tiv_0( \tizeta^{(n)}J(\tiX^{(n)})\tiG^{(n)}-J \tiG_0 )\|_{L^2([0,t]; H^s)}\nonumber\\
	\le & T^{\frac14}  \|\nabla \tiv_0( \tizeta^{(n)}J(\tiX^{(n)})\tiG^{(n)}-J \tiG_0 )\|_{L^2([0,T]; H^s)}\nonumber\\
	= & T^{\frac14}  \|\nabla \tiv_0 \tizeta^{(n)}(J(\tiX^{(n)})-J)\tiG^{(n)}\|_{L^2H^s}+T^{\frac14} \|\nabla \tiv_0( \tizeta^{(n)}-\id)J\tiG^{(n)}\|_{L^2H^s}\nonumber\\
	&+T^{\frac14} \|\nabla \tiv_0J(\tiG^{(n)}-\tiG_0)\|_{L^2H^s}\nonumber\\
	=&: I_{2,1}+ I_{2,2}+ I_{2,3}.\nonumber
\end{align*}
We merely show the estimate of $ I_{2,1} $ since $ I_{2,2} $ and $ I_{2,3} $ are similar with the help of Lemmas \ref{x-omega}, \ref{G-G0} and 
\begin{align}
	\|\tizeta^{(n)}-\id\|_{L^\infty H^s}=&\|\tizeta^{(n)}\nabla(\tiomega-\tiX^{(n)}(t,\tiomega))\|_{L^\infty H^{s}}\nonumber\\
	\le & \|\tizeta^{(n)}\|_{L^\infty H^s} \|\tiomega-\tiX^{(n)}(t,\tiomega)\|_{L^\infty H^{s+1}}\nonumber\\
	\le & T^{\frac 14}C(N, \tiv_0, \tiG_0).\nonumber
\end{align}
From  Lemmas \ref{G-G0},  \ref{jx-j} and  \ref{zeta-}, we have
\begin{align}
	I_{2,1}\le & T^{\frac14}  \|\nabla \tiv_0\|_{L^2H^s}
	\| \tizeta^{(n)}\|_{L^\infty H^s}
	\| J(\tiX^{(n)})-J\|_{L^\infty H^s}
	\| \tiG^{(n)}\|_{L^\infty H^s}\nonumber\\
	\le & T^{\frac34}  \|\nabla \tiv_0\|_{L^\infty H^s}
	\| \tizeta^{(n)}\|_{L^\infty H^s}
	\| J(\tiX^{(n)})-J\|_{L^\infty H^s}
	\| \tiG^{(n)}\|_{L^\infty H^s}\nonumber\\
	\le& T^{\frac34} C(N, \tiv_0,\tiG_0).\nonumber
\end{align}
Thus, we conclude that  $ \sum_{i=1}^{3}I_{2,i}\le T^{\frac34} C(N, \tiv_0,\tiG_0).$

For the third term, we replace $ \nabla\tiv_0 $ by $ \tau \nabla\hat{\phi} $ in the above discussion  to obtain  $ \|I_3 \|_{L^\infty_{\frac 14}H^s}\le T^{\frac74} C(N, \tiv_0,\tiG_0)$ where we have used 
$$ \|\tau \nabla \hat{\phi}\|_{L^2H^s}\le  \|\tau \|_{L^2}C(\tiv_0, \tiG_0)
\le  T^{\frac 32} C(\tiv_0, \tiG_0). $$

To sum up, we obtain
\begin{equation}\label{equpro2.2-1}
	\left\|\tilde{G}^{(n+1)}-\tilde{G}_0-\int_0^t  \nabla \phi J \tilde{G}_0 d\tau \right\|_{L^\infty_{\frac 14}H^s} \le \sum_{i=1}^{3}\|I_i\|_{L^\infty_{\frac 14}H^s}\le T^{\frac 14} C(N, \tiv_0, \tiG_0).
\end{equation}
\par We need to obtain a similar result in $ H^2_{(0)}H^\gamma $. For the first term, we have from Lemma \ref{lem3.3} that
\begin{align}
		&\left\|I_1 \right\|_{H^2_{(0)}H^\gamma}\le  \|\nabla \tiw^{(n)} \tizeta^{(n)} J(\tiX^{(n)}) \tiG^{(n)}\|_{H^1_{(0)}H^{\gamma}}.\nonumber
\end{align}
We need to expand  the above expression in the norm into $ 8 $ terms as follows:
\begin{align}
	 \nabla\tiw^{(n)}[(\tizeta^{(n)}-\id) +\id] [(J(\tiX^{(n)})-J)+J]  [(\tiG^{(n)}-\tiG_0)+\tiG_0].\nonumber
\end{align}
Now we show the estimate of $\nabla \tiw^{(n)}(\tizeta^{(n)}-\id)  (J(\tiX^{(n)})-J) (\tiG^{(n)}-\tiG_0) $ and the others are similar or easier since $ \tiG_0 $ is the initial data. Since $ \tiX^{(n)}\in B(X)$, $(\tiw^{(n)},\tiq_w^{(n)})$ $\in B(w,q), $ and $ \tiG^{(n)}\in B(G), $ we apply Lemmas \ref{x-omega},  \ref{G-G0}, \ref{jx-j}, \ref{zeta-} and \ref{lem3.6} with $ \gamma>1 $ to obtain
\begin{align*}
 &\|\nabla\tiw^{(n)}(\tizeta^{(n)}-\id)  (J(\tiX^{(n)})-J) (\tiG^{(n)}-\tiG_0)\|_{H^1_{(0)}H^{\gamma}}\nonumber\\
 \le& \|\nabla\tiw^{(n)}\|_{H^1_{(0)}H^{\gamma}}\|\tizeta^{(n)}-\id   \|_{H^1_{(0)}H^{\gamma}}\|  J(\tiX^{(n)})-J  \|_{H^1_{(0)}H^{\gamma}}\|
  \tiG^{(n)}-\tiG_0 \|_{H^1_{(0)}H^{\gamma}}\nonumber\\
 \le& C(N,\tiv_0, \tiG_0) \|\tiX^{(n)}-\tiomega   \|_{H^1_{(0)}H^{\gamma+1}}^2T^\delta\nonumber\\
 \le& T^{3\delta} C(N,\tiv_0, \tiG_0) \nonumber
\end{align*}
for some $  \delta>0 $   small enough. Thus, we have
\begin{align*}
	\left\|I_1 \right\|_{H^2_{(0)}H^\gamma} \le T^{\delta_1} C(N,\tiv_0, \tiG_0)
\end{align*}
for some $\delta_1>0. $
\par For the second term, we have
\begin{align*}
	&\left\|I_2  \right\|_{H^2_{(0)}H^\gamma}\le  \|\nabla \tiv_0( \tizeta^{(n)}J(\tiX^{(n)})\tiG^{(n)}-J \tiG_0 )\|_{H^1_{(0)}H^{\gamma}}.\nonumber
\end{align*}
We split the above expression into $ 7 $ terms as before
\begin{align*}
&\nabla \tiv_0( \tizeta^{(n)}J(\tiX^{(n)})\tiG^{(n)}-J \tiG_0 )\nonumber\\
=& \nabla \tiv_0[(\tizeta^{(n)}-\id) +\id] [(J(\tiX^{(n)})-J)+J]  [(\tiG^{(n)}-\tiG_0)+\tiG_0]-\nabla \tiv_0J \tiG_0\nonumber
\end{align*}
since the term $ \nabla \tiv_0J \tiG_0 $ has been eliminated.
In the same manner as the previous one, we conclude that
\begin{equation*}
	\left\|I_2\right\|_{H^2_{(0)}H^\gamma} \le T^{\delta_2} C(N,\tiv_0, \tiG_0) \nonumber
\end{equation*} 
for some $ \delta_2>0, $ where we have noted that $ \nabla\tiv_0 $ is the initial data.
\par For the third term, we have
\begin{align*}
	&\left\|I_3 \right\|_{H^2_{(0)}H^\gamma}\le  \| \tau \nabla \hat{\phi}( \tizeta^{(n)}J(\tiX^{(n)})\tiG^{(n)}-J \tiG_0 )\|_{H^1_{(0)}H^{\gamma}}.\nonumber
\end{align*}
As before, we need to split the above expression into $ 7 $ terms
\begin{align*}
	&\tau \nabla \hat{\phi}( \tizeta^{(n)}J(\tiX^{(n)})\tiG^{(n)}-J \tiG_0 )\nonumber\\
	=& \tau \nabla \hat{\phi}[(\tizeta^{(n)}-\id) +\id] [(J(\tiX^{(n)})-J)+J]  [(\tiG^{(n)}-\tiG_0)+\tiG_0]-\tau \nabla \hat{\phi}J \tiG_0.
\end{align*}
We point out that $ \hat{\phi} $ does not depend on time but just on the initial data $ \tiv_0 $ and $ \tiG_0 $. Besides, we note that $
	\|\tau\|_{H^1_{(0)}}\le  CT^{\frac 12}. $ Therefore, 
we obtain
\begin{align*}
	\left\|I_3 \right\|_{H^2_{(0)}H^\gamma} \le T^{\delta_3} C(N,\tiv_0, \tiG_0)
\end{align*}
for some $ \delta_3>0, $ and 
\begin{align*}
	\left\|\tilde{G}^{(n+1)}-\tilde{G}_0-\int_0^t  \nabla \phi J \tilde{G}_0 d\tau \right\|_{H^2_{(0)}H^\gamma}\le \sum_{i=1}^{3}\|I_i\|_{H^2_{(0)}H^\gamma} \le T^{\delta} C(N, \tiv_0, \tiG_0)
\end{align*}
where $ \delta=\min_{i=1,2,3}\delta_i>0. $ 

Combining with \eqref{equpro2.2-1}, the proof of the first part is complete.

\medskip
\textbf{Part 2.}
\par We consider the  following difference
\begin{align*}
	\tiG^{(n+1)}&-\tiG^{(n+1)} =\int_0^t\nabla\tiv^{(n)}\tizeta^{(n)}J(\tiX^{(n)})\tiG^{(n)}-\nabla\tiv^{(n-1)}\tizeta^{(n-1)}J(\tiX^{(n-1)})\tiG^{(n-1)}d\tau.\nonumber
\end{align*}
We start with the estimate in $ L^\infty_{\frac 14}H^{s} $. From Minkowski's  and Cauchy's inequalities, we have
\begin{align*}
	&\|\tiG^{(n+1)}-\tiG^{(n+1)}\|_{L^\infty_{\frac 14}H^s}\nonumber\\
	=&\left\|\int_0^t\nabla\tiv^{(n)}\tizeta^{(n)}J(\tiX^{(n)})\tiG^{(n)}-\nabla\tiv^{(n-1)}\tizeta^{(n-1)}J(\tiX^{(n-1)})\tiG^{(n-1)} d\tau \right\|_{L^\infty_{\frac 14}H^s}\nonumber\\
	\le& \sup_{t\in [0,T]}t^{-\frac14} \int_{0}^{t}\left\|\nabla\tiv^{(n)}\tizeta^{(n)}J(\tiX^{(n)})\tiG^{(n)}-\nabla\tiv^{(n-1)}\tizeta^{(n-1)}J(\tiX^{(n-1)})\tiG^{(n-1)} \right\|_{H^s}d\tau\nonumber\\
	\le & T^{\frac14}  \|\nabla\tiv^{(n)}\tizeta^{(n)}J(\tiX^{(n)})\tiG^{(n)}-\nabla\tiv^{(n-1)}\tizeta^{(n-1)}J(\tiX^{(n-1)})\tiG^{(n-1)}\|_{L^2([0,T]; H^s)}\nonumber\\
	= &:  T^{\frac14} \|I_{3}\|_{L^2H^s}.\nonumber
\end{align*}
We  split $ I_3 $ as 
\begin{align*}
	I_3	=&(\nabla\tiv^{(n)}-\nabla\tiv^{(n-1)})\tizeta^{(n)}J(\tiX^{(n)})\tiG^{(n)}+\nabla\tiv^{(n-1)}(\tizeta^{(n)}-\tizeta^{(n-1)})J(\tiX^{(n)})\tiG^{(n)}\nonumber\\
	&+\nabla\tiv^{(n-1)}\tizeta^{(n-1)}(J(\tiX^{(n)})-J(\tiX^{(n-1)}))\tiG^{(n)}\\ &+\nabla\tiv^{(n-1)}\tizeta^{(n-1)}J(\tiX^{(n-1)})(\tiG^{(n)}-\tiG^{(n-1)})\nonumber\\
	=&:
	\sum_{i=1}^{4} I_{3,i}.\nonumber
\end{align*}
For $ I_{3,1},$ since $ \tiX^{(n)}\in B(X),  (\tiw^{(n-1)},\tiq_w^{(n-1)}), (\tiw^{(n)},\tiq_w^{(n)})\in B(w,q) $ and $ \tiG^{(n)}\in B(G), $ we apply Lemmas \ref{G-G0},  \ref{jx-j} and \ref{zeta-} to obtain
\begin{align}
	\|I_{3,1}\|_{L^2H^s}\le &\|\nabla\tiv^{(n)}-\nabla\tiv^{(n-1)}\|_{L^2H^s}\|\tizeta^{(n)}\|_{L^\infty H^s}\|J(\tiX^{(n)})\|_{L^\infty H^s}\|\tiG^{(n)}\|_{L^\infty H^s}\nonumber\\
	\le &\|\tiw^{(n)}-\tiw^{(n-1)}\|_{L^2H^{s+1}}C(N,\tiv_0, \tiG_0)\nonumber\\
	\le & C(N,\tiv_0, \tiG_0) \|\tiw^{(n)}-\tiw^{(n-1)}\|_{\mathcal{K}_{(0)}^{s+1}}.\nonumber
\end{align}
For $ I_{3,2}, $ from   Lemmas \ref{G-G0}, \ref{jx-j},  \ref{zetan-zetan-1} and  $ \tiv^{(n)}=\tiw^{(n)}+\tiv_0+t\hat{\phi}, $ we have
\begin{align*}
	\|I_{3,2}\|_{L^2H^s}\le &\|\nabla\tiv^{(n-1)}\|_{L^2H^s}\|\tizeta^{(n)}-\tizeta^{(n-1)}\|_{L^\infty H^s}\|J(\tiX^{(n)})\|_{L^\infty H^s}\|\tiG^{(n)}\|_{L^\infty H^s}\nonumber\\
	\le&(\|\tiw^{(n)}\|_{L^2H^{s+1}}+\|\tiv_0\|_{L^2H^{s+1}}+\|s\hat{\phi}\|_{L^2H^{s+1}})\\
	&\cdot\|\tiX^{(n)}-\tiX^{(n-1)}\|_{L^\infty H^{s+1}}C(N,\tiv_0, \tiG_0)\nonumber\\
	\le&C(N,\tiv_0, \tiG_0)T^{\frac 14}\|\tiX^{(n)}-\tiX^{(n-1)}\|_{L_{\frac 14}^\infty H^{s+1}}\nonumber\\
	\le &T^{\frac 14} C(N,\tiv_0, \tiG_0) \|\tiX^{(n)}-\tiX^{(n-1)}\|_{\mathcal{A}^{s+1,\gamma+1}}.\nonumber
\end{align*}
For $ I_{3,3}, $  the estimate can be obtained in the same manner as $ I_{3,1} $ and $ I_{3,2} $ with the help of Lemmas \ref{G-G0},  \ref{jx-jy} and  \ref{zeta-}. Indeed, we have
\begin{align*}
	\|I_{3,3}\|_{L^2H^s}
	\le T^{\frac 14} C(N,\tiv_0, \tiG_0) \|\tiX^{(n)}-\tiX^{(n-1)}\|_{\mathcal{A}^{s,\gamma+1}}.
\end{align*}
For $ I_{3,4}, $ we have from Lemmas \ref{G-G0},  \ref{jx-j} and \ref{zeta-} that
\begin{align}
	\|I_{3,4}\|_{L^2H^s}\le &\|\nabla\tiv^{(n-1)}\|_{L^2H^s}\|\tizeta^{(n-1)}\|_{L^\infty H^s}\|J(\tiX^{(n-1)})\|_{L^\infty H^s}\|\tiG^{(n)}-\tiG^{(n-1)}\|_{L^\infty H^s}\nonumber\\
	\le&(\|\tiw^{(n)}\|_{L^2H^s}+\|\tiv_0\|_{L^2H^s}+\|t\hat{\phi}\|_{L^2H^s})\|\tiG^{(n)}-\tiG^{(n-1)}\|_{L^\infty H^s}C(N,\tiv_0, \tiG_0)\nonumber\\
	\le&T^{\frac 14}
	C(N,\tiv_0, \tiG_0)\|\tiG^{(n)}-\tiG^{(n-1)}\|_{L^\infty H^s}\nonumber\\
	\le &T^{\frac 14} C(N,\tiv_0, \tiG_0) \|\tiG^{(n)}-\tiG^{(n-1)}\|_{\mathcal{A}^{s,\gamma}}.\nonumber
\end{align}

Now we  estimate  $ \tiG^{(n)}-\tiG^{(n-1)} $ in $ H_{(0)}^2H^\gamma. $  From Lemma \ref{lem3.3}, we have 
\begin{align}
	&\|\tiG^{(n+1)}-\tiG^{(n+1)}\|_{H_{(0)}^2H^\gamma}\nonumber\\
	\le&\|\nabla\tiv^{(n)}\tizeta^{(n)}J(\tiX^{(n)})\tiG^{(n)}-\nabla\tiv^{(n-1)}\tizeta^{(n-1)}J(\tiX^{(n-1)})\tiG^{(n-1)}\|_{H_{(0)}^1H^\gamma}\nonumber\\
	\le &\|\nabla\tiw^{(n)}\tizeta^{(n)}J(\tiX^{(n)})\tiG^{(n)}-\nabla\tiw^{(n-1)}\tizeta^{(n-1)}J(\tiX^{(n-1)})\tiG^{(n-1)}\|_{H_{(0)}^1H^\gamma}\nonumber\\ &+\|\nabla\tiv_0\tizeta^{(n)}J(\tiX^{(n)})\tiG^{(n)}-\nabla\tiv_0\tizeta^{(n-1)}J(\tiX^{(n-1)})\tiG^{(n-1)}\|_{H_{(0)}^1H^\gamma}\nonumber\\
	&+\|t\hat{\phi}\tizeta^{(n)}J(\tiX^{(n)})\tiG^{(n)}-t\hat{\phi}\tizeta^{(n-1)}J(\tiX^{(n-1)})\tiG^{(n-1)}\|_{H_{(0)}^1H^\gamma}\nonumber\\	
	=&: \|I_{4,1}\|_{H_{(0)}^1H^\gamma}+\|I_{4,2}\|_{H_{(0)}^1H^\gamma}+\|I_{4,3}\|_{H_{(0)}^1H^\gamma}.\nonumber
\end{align}
We split  the term $ I_{4,1} $  as follows:
\begin{align*}
	I_{4,1}=&\nabla\tiw^{(n)}[(\tizeta^{(n)}-\tizeta^{(n-1)})+\tizeta^{(n-1)}] [(J(\tiX^{(n)})-J(\tiX^{(n-1)}))+J(\tiX^{(n-1)})]\tiG^{(n)}\nonumber\\
	&-\nabla\tiw^{(n-1)}\tizeta^{(n-1)}J(\tiX^{(n-1)})\tiG^{(n-1)}\nonumber\\
	=&\nabla\tiw^{(n)}[(\tizeta^{(n)}-\tizeta^{(n-1)}) (J(\tiX^{(n-1)})-J)+(\tizeta^{(n)}-\tizeta^{(n-1)}) J\\
	&+(\tizeta^{(n)}-\id)(J(\tiX^{(n)})-J(\tiX^{(n-1)}))\nonumber\\
	&+(J(\tiX^{(n)})-J(\tiX^{(n-1)}))+\tizeta^{(n-1)}J(\tiX^{(n-1)})]\tiG^{(n)}\\
	&-\nabla\tiw^{(n-1)}\tizeta^{(n-1)}J(\tiX^{(n-1)})\tiG^{(n-1)}\nonumber\\
	=& \nabla\tiw^{(n)}(\tizeta^{(n)}-\tizeta^{(n-1)}) (J(\tiX^{(n-1)})-J)[(\tiG^{(n)}-\tiG_0) +\tiG_0]\nonumber\\
	&+\nabla\tiw^{(n)}(\tizeta^{(n)}-\tizeta^{(n-1)}) J[(\tiG^{(n)}-\tiG_0) +\tiG_0]\nonumber\\
	&+ \nabla\tiw^{(n)} (\tizeta^{(n)}-\id)(J(\tiX^{(n)})-J(\tiX^{(n-1)}))[(\tiG^{(n)}-\tiG_0) +\tiG_0]\nonumber\\
	&+ \nabla\tiw^{(n)}(J(\tiX^{(n)})-J(\tiX^{(n-1)}))[(\tiG^{(n)}-\tiG_0) +\tiG_0]\nonumber\\
	&+\nabla\tiw^{(n)}\tizeta^{(n-1)}J(\tiX^{(n-1)})\tiG^{(n)}\nonumber\\
	&-\nabla\tiw^{(n-1)}\tizeta^{(n-1)}J(\tiX^{(n-1)})\tiG^{(n-1)}\nonumber\\
	=& \nabla\tiw^{(n)}(\tizeta^{(n)}-\tizeta^{(n-1)}) (J(\tiX^{(n-1)})-J)(\tiG^{(n)}-\tiG_0)\nonumber\\
	&+\nabla\tiw^{(n)}(\tizeta^{(n)}-\tizeta^{(n-1)}) (J(\tiX^{(n-1)})-J)\tiG_0\nonumber\\
	&+\nabla\tiw^{(n)}(\tizeta^{(n)}-\tizeta^{(n-1)}) J(\tiG^{(n)}-\tiG_0)\nonumber\\
	&+\nabla\tiw^{(n)}(\tizeta^{(n)}-\tizeta^{(n-1)}) J\tiG_0\nonumber\\
	&+ \nabla\tiw^{(n)} (\tizeta^{(n)}-\id)(J(\tiX^{(n)})-J(\tiX^{(n-1)}))(\tiG^{(n)}-\tiG_0) \nonumber\\
	&+ \nabla\tiw^{(n)} (\tizeta^{(n)}-\id)(J(\tiX^{(n)})-J(\tiX^{(n-1)}))\tiG_0\nonumber\\
	&+ \nabla\tiw^{(n)}(J(\tiX^{(n)})-J(\tiX^{(n-1)}))(\tiG^{(n)}-\tiG_0)\nonumber\\
	&+ \nabla\tiw^{(n)}(J(\tiX^{(n)})-J(\tiX^{(n-1)}))\tiG_0\nonumber\\
	&+(\nabla\tiw^{(n)}-\nabla\tiw^{(n-1)})(\tizeta^{(n-1)}-\id)(J(\tiX^{(n-1)})-J)(\tiG^{(n-1)}-\tiG_0)\nonumber\\
	&+(\nabla\tiw^{(n)}-\nabla\tiw^{(n-1)})(\tizeta^{(n-1)}-\id)J(\tiG^{(n-1)}-\tiG_0)\nonumber\\
	&+(\nabla\tiw^{(n)}-\nabla\tiw^{(n-1)})\id(J(\tiX^{(n-1)})-J)(\tiG^{(n-1)}-\tiG_0)\nonumber\\
	&+(\nabla\tiw^{(n)}-\nabla\tiw^{(n-1)})\id J(\tiG^{(n-1)}-\tiG_0)\nonumber\\
	&+(\nabla\tiw^{(n)}-\nabla\tiw^{(n-1)})(\tizeta^{(n-1)}-\id)(J(\tiX^{(n-1)})-J)\tiG_0\nonumber\\
	&+(\nabla\tiw^{(n)}-\nabla\tiw^{(n-1)})(\tizeta^{(n-1)}-\id)J\tiG_0\nonumber\\
	&+(\nabla\tiw^{(n)}-\nabla\tiw^{(n-1)})\id (J(\tiX^{(n-1)})-J)\tiG_0\nonumber\\
	&+(\nabla\tiw^{(n)}-\nabla\tiw^{(n-1)})\id J\tiG_0\nonumber\\
	&+\nabla\tiw^{(n)}(\tizeta^{(n-1)}-\id)(J(\tiX^{(n-1)})-J)(\tiG^{(n)}-\tiG^{(n-1)})\nonumber\\
	&+\nabla\tiw^{(n)}(\tizeta^{(n-1)}-\id)J(\tiG^{(n)}-\tiG^{(n-1)})\nonumber\\
	&+\nabla\tiw^{(n)}\id(J(\tiX^{(n-1)})-J)(\tiG^{(n)}-\tiG^{(n-1)})\nonumber\\
	&+\nabla\tiw^{(n)}\id J(\tiG^{(n)}-\tiG^{(n-1)})\nonumber\\
	=&:\sum_{i=1}^{20} I_{4,1,i}.\nonumber
\end{align*}
In this process, we avoid multiplying $  (\tizeta^{(n)}-\tizeta^{(n-1)}) $ and $ (J(\tiX^{(n)})-J(\tiX^{(n-1)}))$ first. Then, in each product,  we leave only one item as the difference between the iterative steps $ n-1 $ and $ n $. For the other terms, we  subtract their  values at $ t=0 $ to apply the estimates.

We will show the estimate of $ I_{4,1,1} $ which is the most complicated term.
Since $ \gamma>1, $ we have from $ \tiX^{(n-1)}, \tiX^{(n)}\in B(X),  (\tiw^{(n)},\tiq_w^{(n)})\in B(w,q), \tiG^{(n)}\in B(G), $ Lemmas  \ref{x-omega},  \ref{G-G0},  \ref{jx-j},  \ref{zetan-zetan-1} and \ref{lem3.6} that
\begin{align*}
	&\|I_{4,1,1}\|_{H_{(0)}^1H^\gamma}\nonumber\\
	=&\|\nabla\tiw^{(n)}(\tizeta^{(n)}-\tizeta^{(n-1)}) (J(\tiX^{(n-1)})-J)(\tiG^{(n)}-\tiG_0)\|_{H_{(0)}^1H^\gamma}\nonumber\\
	\le&\|\nabla \tiw^{(n)}\|_{H_{(0)}^1H^{\gamma}}\|\tizeta^{(n)}-\tizeta^{(n-1)}\|_{H_{(0)}^1H^\gamma}\| J(\tiX^{(n-1)})-J\|_{H_{(0)}^1H^\gamma}\|\tiG^{(n)}-\tiG_0\|_{H_{(0)}^1H^\gamma}\nonumber\\
	\le& T^\delta C(N, \tiv_0, \tiG_0)\|\tiX^{(n)}-\tiX^{(n-1)}\|_{H_{(0)}^1H^{\gamma+1}} \|\tiX^{(n-1)}-\tiomega\|_{H_{(0)}^1H^\gamma}\nonumber\\
	\le& T^{2\delta} C(N, \tiv_0, \tiG_0)\|\tiX^{(n)}-\tiX^{(n-1)}\|_{\mathcal{A}^{s+1,\gamma+1}},\nonumber
\end{align*}
where $ \delta>0 $ is sufficiently small.
Similarly, for the remaining terms, we have 
\begin{align*}
	\sum_{i=2}^{8}\|I_{4,1,i}\|_{H_{(0)}^1H^\gamma}\le & T^{\delta_{4,1,1}} C(N, \tiv_0, \tiG_0)\|\tiX^{(n)}-\tiX^{(n-1)}\|_{\mathcal{A}^{s+1,\gamma+1}},\nonumber	\\
	\sum_{i=9}^{16}\|I_{4,1,i}\|_{H_{(0)}^1H^\gamma}\le & T^{\delta_{4,1,2}} C(N, \tiv_0, \tiG_0)\|\tiw^{(n)}-\tiw^{(n-1)}\|_{\mathcal{K}^{s+1}_{(0)}},\nonumber\\
	\sum_{i=17}^{20}\|I_{4,1,i}\|_{H_{(0)}^1H^\gamma}\le & T^{\delta_{4,1,3}} C(N, \tiv_0, \tiG_0)\|\tiG^{(n)}-\tiG^{(n-1)}\|_{\mathcal{A}^{s,\gamma}}\nonumber
\end{align*}
for some $ \delta_{4,1,i}>0. $

Next, we divide $ I_{4,2} $ and $ I_{4,3} $ in the same way as $ I_{4,1}. $ Since $ \tiv_0 $ and $ \hat{\phi} $ depend only on the initial data, we obtain similar results by applying the estimates. Thus, for $ i=2$ and $3,$  we have
\begin{align}
	\|I_{4,i}\|_{H_{(0)}^1H^\gamma}\le   T^{\delta_{4,i}} C(N, \tiv_0, \tiG_0)(\|\tiX^{(n)}-\tiX^{(n-1)}\|_{\mathcal{A}^{s+1,\gamma+1}}+\|\tiG^{(n)}-\tiG^{(n-1)}\|_{\mathcal{A}^{s,\gamma}})\nonumber
\end{align}
for some $ \delta_{4,i}>0. $

Combining the estimates in $ L_{\frac 14}^\infty H^s $ and $ H_{(0)}^2H^\gamma, $ the proof is done since
	\begin{align}
	\|\tiG^{(n+1)}-\tiG^{(n)}\|_{\mathcal{A}^{s,\gamma}}\le & C(N,\tiv_0,\tiG_0) T^\delta (\|\tiG^{(n)}-\tiG^{(n-1)}\|_{\mathcal{A}^{s,\gamma}}+ \|\tiw^{(n)}-\tiw^{(n-1)}\|_{\mathcal{K}^{s+1}_{(0)}}\nonumber\\
	& + \|\tiX^{(n)}-\tiX^{(n-1)}\|_{\mathcal{A}^{s+1,\gamma+1}})\nonumber
\end{align}
for $ \delta=\min \{ \frac 14, \frac 12, \delta_{4,1,1}, \delta_{4,1,2}, \delta_{4,1,3}, \delta_{4,2}, \delta_{4,3} \}>0. $
\end{proof}

\begin{proposition}\label{pro3}
	For $ 2<s<\frac52$, $1<\gamma <s-1, $ and $ T>0 $ small enough depending on $  N,\tiv_0$  and $ \tiG_0, $ we have
	
	(1) For $ n\ge 0, $ if $ \tiX^{(n)}\in B(X), (\tiw^{(n)},\tiq_w^{(n)})\in B(w,q) $ and  $ \tiG^{(n)}\in B(G), $ then $ (\tiw^{(n+1)},\tiq_w^{(n+1)})\in B(w,q). $ 
	
	(2) For $ n\ge 1, $ if $ \tiX^{(n-1)}, \tiX^{(n)}, \tiX^{(n+1)} \in B(X)$, $(\tiw^{(n-1)},\tiq_w^{(n-1)}), (\tiw^{(n)},\tiq_w^{(n)}),$ $ (\tiw^{(n+1)},$ $\tiq_w^{(n+1)})\in B(w,q) $ and $ \tiG^{(n-1)}, \tiG^{(n)}, \tiG^{(n+1)}\in B(G), $ then 
	\begin{align}
		\|\tiw^{(n+1)}&-\tiw^{(n)}\|_{\mathcal{K}_{(0)}^{s+1}}+\|\tiq_w^{(n+1)}-\tiq_w^{(n)}\|_{\mathcal{K}_{pr(0)}^{s}}\le C(\tiv_0,\tiG_0) T^\delta (\|\tiG^{(n)}-\tiG^{(n-1)}\|_{\mathcal{A}^{s,\gamma}}\nonumber\\
		&+ \|\tiw^{(n)}-\tiw^{(n-1)}\|_{\mathcal{K}^{s+1}_{(0)}} +\|\tiq_w^{(n)}-\tiq_w^{(n-1)}\|_{\mathcal{K}^{s}_{pr(0)}}+ \|\tiX^{(n)}-\tiX^{(n-1)}\|_{\mathcal{A}^{s+1,\gamma+1}})\nonumber
	\end{align}
for some $ \delta>0. $
\end{proposition}
\begin{proof} 
	
\textbf{Part 1.}

From Lemma \ref{lem4.1}, we have
\begin{align}
&\|(\tiw^{(n+1)},\tiq_w^{(n+1)})-L^{-1}( \tilde{f}_\phi^L+\tilde{f}_{G_0}, \bar{g}_\phi^L, \tilde{h}_\phi^L+\tilde{h}_{G_0}, 0 )\|_{X_0}\nonumber\\
=&\|L^{-1}( \tilde{f}^{(n)}-\tilde{f}_{G_0}, \bar{g}^{(n)}, \tilde{h}^{(n)}-\tilde{h}_{G_0}, 0 )\|_{X_0}\nonumber\\
\le & C(\| \tilde{f}^{(n)}-\tilde{f}_{G_0}\|_{\mathcal{K}_{(0)}^{s-1}}+\| \bar{g}^{(n)}\|_{\bar{\mathcal{K}}_{(0)}^{s}}+\| \tilde{h}^{(n)}-\tilde{h}_{G_0} \|_{\mathcal{K}_{(0)}^{s-\frac 12}}).\label{equ4.13}
\end{align}
Thus, it is sufficient to estimate $ \| \tilde{f}^{(n)}-\tilde{f}_{G_0}\|_{\mathcal{K}_{(0)}^{s-1}}, \| \bar{g}^{(n)}\|_{\bar{\mathcal{K}}_{(0)}^{s}} $ and $ \| \tilde{h}^{(n)}-\tilde{h}_{G_0} \|_{\mathcal{K}_{(0)}^{s-\frac 12}}. $ 

\textbf{Estimate for $ \tif^{(n)} $}.
We rewrite $  \tif^{(n)} $ as
\begin{align}
		\tif^{(n)}
		= & -Q^2 \Delta (\tiw^{(n)}+\phi)+ J^\top \nabla (\tiq^{(n)})+Q^2(\tiX^{(n)}) \nabla(\nabla(\tiw^{(n)}+\phi)\tizeta^{(n)})\tizeta^{(n)}\nonumber\\
		&-J(\tiX^{(n)})^\top  \tizeta^{(n)\top} \nabla (\tiq^{(n)})\nonumber\\
		=&[-Q^2\laplace\tiw^{(n)}+Q^2(\tiX^{(n)})\nabla(\nabla \tiw^{(n)}\tizeta^{(n)})\tizeta^{(n)}]\nonumber\\
		&+[-Q^2\laplace\phi+Q^2(\tiX^{(n)})\nabla(\nabla \phi\tizeta^{(n)})\tizeta^{(n)}]\nonumber\\
		&+[J^\top\nabla \tiq^{(n)}-J(\tiX^{(n)})^\top \tizeta^{(n)\top}\nabla \tiq^{(n)}]+[\nabla\tiG^{(n)}\tizeta^{(n)}J(\tiX^{(n)})\tiG^{(n)}]\nonumber\\
		=&:\tif^{(n)}_w+\tif^{(n)}_\phi+\tif^{(n)}_q+\tif^{(n)}_G.\label{equ4.14}
\end{align}
Note that $ \tif^{(n)}_w, \tif^{(n)}_\phi $ and  $\tif^{(n)}_q $ have already been studied in \cite[Proposition 5.4]{castro2019splash} where the definition of
$ \tif^{(n)}_w $ and $ \tif^{(n)}_q  $
are the same. However, our definition of $ \phi $ is different from that in \cite{castro2019splash} since we have introduced the initial magnetic field  $ \tiG_0 $. It turns out that this distinction does not significantly affect the estimates for $ \tif^{(n)}_\phi $ in \cite{castro2019splash} but the
constant will depend on $ \tiG_0. $ For this motivation, it suffices to show the estimate of $ \tif^{(n)}_G- \tif_{G_0} $ in  $ L^2H^{s-1} $ and $ H_{(0)}^{\frac{s-1}{2}}L^2. $

For the estimate in $ L^2H^{s-1}, $ we split $ \tif^{(n)}_G- \tif_{G_0} $ as 
\begin{align*} 
	\tif^{(n)}_G- \tif_{G_0}=&\nabla\tiG^{(n)}\tizeta^{(n)}J(\tiX^{(n)})\tiG^{(n)}-\nabla\tiG_0J\tiG_0\nonumber\\
	=&[(\nabla\tiG^{(n)}-\nabla\tiG_0)+\nabla\tiG_0]\tizeta^{(n)}J(\tiX^{(n)})[(\tiG^{(n)}-\tiG_0)+\tiG_0]-\nabla\tiG_0J\tiG_0\nonumber\\
	=&(\nabla\tiG^{(n)}-\nabla\tiG_0)\tizeta^{(n)}J(\tiX^{(n)})(\tiG^{(n)}-\tiG_0)\\
	&+(\nabla\tiG^{(n)}-\nabla\tiG_0)\tizeta^{(n)}J(\tiX^{(n)})\tiG_0+\nabla\tiG_0\tizeta^{(n)}J(\tiX^{(n)})(\tiG^{(n)}-\tiG_0)\nonumber\\
	&+\nabla\tiG_0\tizeta^{(n)}(J(\tiX^{(n)})-J)\tiG_0+\nabla\tiG_0(\tizeta^{(n)}-\id) J\tiG_0\nonumber\\	
	=:&\sum_{i=1}^{5}d_i^f.\nonumber
\end{align*}
We merely show the estimate of $ d_1^f  $ and $ d_4^f. $ From $ \tiX^{(n)}\in B(X),  \tiG^{(n)}\in B(G), $ Lemmas \ref{G-G0}, \ref{jx-j} and  \ref{zeta-}, we have for $ d_1^f $ that
\begin{align}
	&\|d_1^f\|_{L^2H^{s-1}}\nonumber\\
	=&\|(\nabla\tiG^{(n)}-\nabla\tiG_0)\tizeta^{(n)}J(\tiX^{(n)})(\tiG^{(n)}-\tiG_0)\|_{L^2H^{s-1}}\nonumber\\
	\le & \|\nabla\tiG^{(n)}-\nabla\tiG_0\|_{L^2H^{s-1}} \|\tizeta^{(n)}\|_{L^\infty H^{s-1}}\|J(\tiX^{(n)})\|_{L^\infty H^{s-1}}\|\tiG^{(n)}-\tiG_0\|_{L^\infty H^{s-1}}\nonumber\\
	\le & T^{\frac 12+2\delta}C(N,\tiv_0, \tiG_0)\nonumber
\end{align}
for some $ \delta>0 $ small enough. 
For $ d_4^f, $ we apply Lemmas \ref{jx-j} and \ref{zeta-} to obtain
\begin{align}
	&\|\nabla\tiG_0\tizeta^{(n)}(J(\tiX^{(n)})-J)\tiG_0\|_{L^2H^{s-1}}\nonumber\\
	\le & \|\nabla \tiG_0\|_{L^2H^{s-1}} \|\tizeta^{(n)}\|_{L^\infty H^{s-1}}\|J(\tiX^{(n)})-J\|_{L^\infty H^{s-1}}\| \tiG_0 \|_{L^\infty H^{s-1}}\nonumber\\
	\le & T^{\frac 12}C(N,\tiv_0, \tiG_0).\nonumber
\end{align}
Similarly, we can estimate other terms.

For the estimate in $ H_{(0)}^{\frac{s-1}{2}}L^2, $ we use the following splitting:
\begin{align*}
	&\tif^{(n)}_G- \tif_{G_0}\nonumber\\
	=&\nabla\tiG^{(n)}\tizeta^{(n)}J(\tiX^{(n)})\tiG^{(n)}-\nabla\tiG_0J\tiG_0\nonumber\\
	=&[(\nabla\tiG^{(n)}-\nabla\tiG_0)+\nabla\tiG_0][(\tizeta^{(n)}-\id)+\id][(J(\tiX^{(n)})-J)+J][(\tiG^{(n)}-\tiG_0)+\tiG_0]\nonumber\\
	&-\nabla\tiG_0J\tiG_0=:\sum_{i=1}^{15}\bar{d}_i^f,\nonumber
\end{align*}
where we have fully expanded the first product into $ 16 $ terms.
We only show the estimate of $\bar{d}_1^f= (\nabla\tiG^{(n)}-\nabla\tiG_0)(\tizeta^{(n)}-\id)(J(\tiX^{(n)})-J)(\tiG^{(n)}-\tiG_0) $, and the others are similar or easier. From $ \tiX^{(n)}\in B(X), \tiG^{(n)}\in B(G), $ Lemmas   \ref{x-omega},  \ref{G-G0},  \ref{jx-j}, \ref{zeta-}, \ref{lem3.4} and  \ref{lem3.7}, we have
\begin{align}
	&\|\bar{d}_1^f\|_{H_{(0)}^{\frac{s-1}{2}}L^2}\nonumber\\
	\le&\|(\nabla\tiG^{(n)}-\nabla\tiG_0)(\tizeta^{(n)}-\id)(J(\tiX^{(n)})-J)(\tiG^{(n)}-\tiG_0)\|_{H_{(0)}^{\frac{s-1}{2}}L^2}\nonumber\\
	\le & \|(\nabla\tiG^{(n)}-\nabla\tiG_0)(\tizeta^{(n)}-\id)(J(\tiX^{(n)})-J)\|_{H_{(0)}^{\frac{s-1}{2}}L^2}\|\tiG^{(n)}-\tiG_0\|_{H_{(0)}^{\frac{s-1}{2}}H^{1+\eta}}\nonumber\\
	\le & \|\nabla\tiG^{(n)}-\nabla\tiG_0\|_{H_{(0)}^{\frac{s-1}{2}}L^2}\|(\tizeta^{(n)}-\id)(J(\tiX^{(n)})-J)\|_{H_{(0)}^{\frac{s-1}{2}}H^{1+\eta}}\|\tiG^{(n)}-\tiG_0\|_{H_{(0)}^{\frac{s-1}{2}}H^{1+\eta}}\nonumber\\
	\le & \|\tizeta^{(n)}-\id\|_{H_{(0)}^{\frac{s-1}{2}}H^{1+\eta}}\|J(\tiX^{(n)})-J\|_{H_{(0)}^{\frac{s-1}{2}}H^{1+\eta}}\|\tiG^{(n)}-\tiG_0\|^2_{H_{(0)}^{\frac{s-1}{2}}H^{1+\eta}}\nonumber\\
	\le & \|\tiX^{(n)}-\tiomega\|_{H_{(0)}^{\frac{s-1}{2}}H^{2+\eta}}C(N,\tiv_0,\tiG_0)(\|\tiX^{(n)}-\hat{X}\|_{H_{(0)}^{\frac{s-1}{2}+\delta}H^{1+\eta}}+T)T^{2\delta}\nonumber\\
	\le & T^{3\delta}C(N,\tiv_0,\tiG_0)(\|\tiX^{(n)}-\hat{X}\|_{\mathcal{A}^{s+1,\gamma+1}}+T)\nonumber\\
	\le &C(N,\tiv_0, \tiG_0)  T^{3\delta}\nonumber
\end{align}
for some $ \delta, \eta>0 $ small enough. We conclude that $ \|d_i^f\|_{H_{(0)}^{\frac{s-1}{2}}L^2} \le C(N,\tiv_0,\tiG_0) T^{\delta_i} $ for some $ \delta_i>0, i=1, \cdots, 15, $ and therefore
\begin{equation}\label{equ4.15}
	\|\tif^{(n)}_w\|_{\mathcal{K}_{(0)}^{s-1}}+\|\tif^{(n)}_\phi\|_{\mathcal{K}_{(0)}^{s-1}}+\|\tif^{(n)}_q\|_{\mathcal{K}_{(0)}^{s-1}}+\|\tif^{(n)}_G- \tif_{G_0}\|_{\mathcal{K}_{(0)}^{s-1}}\le T^{\delta}C(N,\tiv_0, \tiG_0)
\end{equation}
for some $ \delta>0 $.

\textbf{ Estimate for $ \bar{g}^{(n)} $.}
We split $ \bar{g}^{(n)} $ as in \cite{castro2019splash}:
\begin{align}
	\bar{g}^{(n)}  =&\trace (\nabla (\tiw^{(n)}+\phi) J)-\trace (\nabla (\tiw^{(n)}+\phi) \tizeta^{(n)} J (\tiX^{(n)} )+\operatorname{Tr}(\nabla \phi \tilde{\zeta}_\phi J_\phi)-\operatorname{Tr}(\nabla \phi J)\nonumber \\
	=& \trace(\nabla\tiw^{(n)}(J-J(\tiX^{(n)})))+\trace(\nabla\tiw^{(n)}(\id-\tizeta^{(n)})J(\tiX^{(n)}))\nonumber\\
	&+\trace(\nabla\phi\tizeta_\phi(J_\phi-J(\tiX^{(n)})))+\trace(\nabla\phi(\tizeta_\phi-\tizeta^{(n)})J(\tiX^{(n)})).\nonumber
\end{align}
For all computations we recall \cite{castro2019splash} and the final estimate will depend on both $ \tiv_0 $ and $ \tiG_0, $ i.e.,
\begin{equation}\label{equ4.16}
	\|\bar{g}^{(n)}\|_{\bar{\mathcal{K}}_{(0)}^{s}}\le T^\theta C(N,\tiv_0,\tiG_0)
\end{equation}
for some $ \theta>0. $

\textbf{ Estimate for $ \tih^{(n)} $.}
We rewrite $ \tih^{(n)}  $ as follows:
\begin{align}
	\tih^{(n)}
	=&  -\tiq^{(n)}J^{-1} \tilde{n}_0+\tiq^{(n)}J(\tiX^{(n)})^{-1} \nabla_{\Lambda} \tiX^{(n)} \tilde{n}_0\nonumber\\
	&+((\nabla (\tiw^{(n)}+\phi) J) +(\nabla (\tiw^{(n)}+\phi) J)^\top )J^{-1} \tilde{n}_0\nonumber \\
	& - ( (\nabla (\tiw^{(n)}+\phi) \tizeta^{(n)} J (\tiX^{(n)} ) )\nonumber \\
	&	+ (\nabla (\tiw^{(n)}+\phi) \tizeta^{(n)} J (\tiX^{(n)} ) )^\top  ) J (\tiX^{(n)} ) ^{-1} \nabla_{\Lambda} \tiX^{(n)} \tilde{n}_0\nonumber \\
	&-\tiG^{(n)}\tiG^{(n)\top}J(\tiX^{(n)})^{-1} \nabla_{\Lambda} \tiX^{(n)} \tilde{n}_0\nonumber \\
	= &\nabla \tiw^{(n)}\tilde{n}_0-\nabla \tiw^{(n)}\tizeta^{(n)} \nabla_{\Lambda} \tiX^{(n)} \tilde{n}_0 \nonumber \\
	&+(\nabla\tiw^{(n)}J)^\top  J^{-1} \tilde{n}_0-(\nabla\tiw^{(n)}J(\tiX^{(n)}))^\top J(\tiX^{(n)})^{-1} \nabla_{\Lambda} \tiX^{(n)} \tilde{n}_0 \nonumber \\	
	&+ \nabla\phi \tilde{n}_0-\nabla \tiw^{(n)}\tizeta^{(n)} \nabla_{\Lambda} \tiX^{(n)} \tilde{n}_0 \nonumber \\
	&+(\nabla\phi J)^\top  J^{-1} \tilde{n}_0-(\nabla\phi J(\tiX^{(n)}))^\top J(\tiX^{(n)})^{-1} \nabla_{\Lambda} \tiX^{(n)} \tilde{n}_0\nonumber \\
	&+(-\tiq^{(n)} J^{-1} \tilde{n}_0+\tiq^{(n)}J(\tiX^{(n)})^{-1} \nabla_{\Lambda} \tiX^{(n)} \tilde{n}_0) \nonumber \\
	&+(-\tiG^{(n)}\tiG^{(n)\top}J(\tiX^{(n)})^{-1} \nabla_{\Lambda} \tiX^{(n)} \tilde{n}_0)\nonumber \\
	=&:\tih^{(n)}_w+\tih^{(n)}_{w^\top }+\tih^{(n)}_\phi+\tih^{(n)}_{\phi^\top }+\tih^{(n)}_q+\tih^{(n)}_G.\label{equ4.17}
\end{align}
As for $ \tif^{(n)}, $ the estimates of $ \tih^{(n)}_w, \tih^{(n)}_{w^\top }, \tih^{(n)}_\phi, \tih^{(n)}_{\phi^\top } $ and  $\tih^{(n)}_q $ are the same as in \cite{castro2019splash} expect $ \tih^{(n)}_G- \tih_{G_0}. $

For the estimate of $ \tih^{(n)}_G- \tih_{G_0} $ in $ L^2H^{s-\frac12}, $ we use the following splitting:
\begin{align}
	\tih_{G_0}-\tih^{(n)}_G=&\tiG^{(n)}\tiG^{(n)\top}J(\tiX^{(n)})^{-1} \nabla_{\Lambda} \tiX^{(n)} \tilde{n}_0-\tiG_0\tiG_0^{(n)\top}J^{-1} \tilde{n}_0\nonumber\\
	=&(\tiG^{(n)}-\tiG_0)(\tiG^{(n)}-\tiG_0)^\top J(\tiX^{(n)})^{-1} \nabla_{\Lambda} \tiX^{(n)}\tilde{n}_0\nonumber\\
	&+\tiG_0(\tiG^{(n)}-\tiG_0)^\top J(\tiX^{(n)})^{-1} \nabla_{\Lambda} \tiX^{(n)}\tilde{n}_0\nonumber\\ &+(\tiG^{(n)}-\tiG_0)\tiG_0^\top J(\tiX^{(n)})^{-1} \nabla_{\Lambda} \tiX^{(n)}\tilde{n}_0\nonumber\\
	&+\tiG_0\tiG_0(J(\tiX^{(n)})^{-1}-J^{-1}) \nabla_{\Lambda} \tiX^{(n)}\tilde{n}_0\nonumber\\
	&+\tiG_0\tiG_0^\top J^{-1} \nabla_{\Lambda} \tiX^{(n)}\tilde{n}_0
	=:\sum_{i=1}^{5} d_i^h.\nonumber
\end{align}
We estimate the most difficult term $ d_1^h. $ Note that these products involve the inverse matrix of $ J. $ By direct computation, we have 
\begin{equation}\label{Jni}
	J(\tiX)^{-1}=
	\left( \begin{array}{cc}
		2 \tiX^{1}	& -2 \tiX^{2}  \\
		2 \tiX^{2}	& 2 \tiX^{1}
	\end{array}\right). 
\end{equation}
Now from  $ \tiX^{(n)}\in B(X), \tiG^{(n)}\in B(G), $  equation \eqref{Jni},    Lemmas \ref{x-omega}, \ref{G-G0} and Theorem \ref{lem3.9}, we obtain
\begin{align}
	&\|d_1^h\|_{L^2H^{s-\frac 12}}\nonumber\\
	=&\|(\tiG^{(n)}-\tiG_0)(\tiG^{(n)}-\tiG_0)^\top J(\tiX^{(n)})^{-1} \nabla_{\Lambda} \tiX^{(n)}\tilde{n}_0\|_{L^2H^{s-\frac 12}}\nonumber\\
	\le &\|\tiG^{(n)}-\tiG_0\|_{L^2H^{s-\frac 12}}\|(\tiG^{(n)}-\tiG_0)^\top \|_{L^\infty H^{s-\frac 12}}\|J(\tiX^{(n)})^{-1}\|_{L^\infty H^{s-\frac 12}}\| \nabla_{\Lambda} \tiX^{(n)}\tilde{n}_0\|_{L^\infty H^{s-\frac 12}}\nonumber\\
	\le &CT^{\frac 12}\|\tiG^{(n)}-\tiG_0\|^2_{L^\infty H^{s-\frac 12}}\|J(\tiX^{(n)})^{-1}\|_{L^\infty H^{s-\frac 12}}\| \nabla_{\Lambda} \tiX^{(n)}\tilde{n}_0\|_{L^\infty H^{s-\frac 12}}\nonumber\\
	\le &CT^{\frac 12}\|\tiG^{(n)}-\tiG_0\|^2_{L^\infty H^s}\|\tiX^{(n)}\|_{L^\infty H^{s}}^2\nonumber\\
	\le &C(N,\tiv_0,\tiG_0)T.\nonumber
\end{align}
Similarly, we deduce
\begin{align*}
	\sum_{i=2}^{5}\|d_i^h\|_{L^2H^{s-\frac 12}}\le C(N,\tiv_0,\tiG_0)T^{\frac 34}.
\end{align*}

For the estimate in $ H_{(0)}^{\frac{s}{2}-\frac 14}L^2, $ we use the following splitting:
\begin{align*}
	&\tih_{G_0}-\tih^{(n)}_G\nonumber\\
	=&\tiG^{(n)}\tiG^{(n)\top}J(\tiX^{(n)})^{-1} \nabla_{\Lambda} \tiX^{(n)} \tilde{n}_0-\tiG_0\tiG_0^{(n)\top}J^{-1} \tilde{n}_0\nonumber\\
	=&[(\tiG^{(n)}-\tiG_0)+\tiG_0][(\tiG^{(n)}-\tiG_0)+\tiG_0]^\top [(J(\tiX^{(n)})^{-1}-J^{-1})+J^{-1}] \\
	&\cdot[(\nabla_{\Lambda} \tiX^{(n)}-\id)+\id]\tilde{n}_0-\tiG_0\tiG_0^{(n)\top}J^{-1} \tilde{n}_0\nonumber\\
	=:& \sum_{i=1}^{15} \bar{d}_i^h.\nonumber
\end{align*}
For the most difficult term $ \bar{d}_1^h, $ from  $ \tiX^{(n)}\in B(X),  \tiG^{(n)}\in B(G), $ equation \eqref{Jni}, Lemmas \ref{x-omega}, \ref{G-G0},   \ref{lem3.4},  \ref{lem3.5} and  Theorem \ref{lem3.9}, we  obtain
\begin{align*}
	&\|\bar{d}_1^h\|_{H_{(0)}^{\frac{s}{2}-\frac 14}L^2}\nonumber\\
	=&\|(\tiG^{(n)}-\tiG_0)(\tiG^{(n)}-\tiG_0)^\top (J(\tiX^{(n)})^{-1}-J^{-1} )(\nabla_{\Lambda} \tiX^{(n)}-\id)\tilde{n}_0\|_{H_{(0)}^{\frac{s}{2}-\frac 14}L^2}\nonumber\\
	\le &\|(\tiG^{(n)}-\tiG_0)(\tiG^{(n)}-\tiG_0)^\top \|_{H_{(0)}^{\frac{s}{2}-\frac 14}L^2}\\
	&\cdot\|(J(\tiX^{(n)})^{-1}-J^{-1} )(\nabla_{\Lambda} \tiX^{(n)}-\id)\|_{H_{(0)}^{\frac{s}{2}-\frac 14}H^{\frac 12+\eta}}\nonumber\\
	\le &\|\tiG^{(n)}-\tiG_0\|_{H_{(0)}^{\frac{s}{2}-\frac 14}H^{\frac 12+\eta}}\|(\tiG^{(n)}-\tiG_0)^\top \|_{H_{(0)}^{\frac{s}{2}-\frac 14}H^{\frac 12-\eta}}\|J(\tiX^{(n)})^{-1}-J^{-1} \|_{H_{(0)}^{\frac{s}{2}-\frac 14}H^{\frac 12+\eta}}\nonumber\\
	&\cdot\|\nabla_{\Lambda} \tiX^{(n)}-\id \|_{H_{(0)}^{\frac{s}{2}-\frac 14}H^{\frac 12+\eta}}\nonumber\\
	\le &\|\tiG^{(n)}-\tiG_0\|_{H_{(0)}^{\frac{s}{2}-\frac 14}H^{1+\eta}}\|(\tiG^{(n)}-\tiG_0)^\top \|_{H_{(0)}^{\frac{s}{2}-\frac 14}H^{1-\eta}}\|J(\tiX^{(n)})^{-1}-J^{-1} \|_{H_{(0)}^{\frac{s}{2}-\frac 14}H^{1+\eta}}\nonumber\\
	&\cdot\|\nabla_{\Lambda} \tiX^{(n)}-\id \|_{H_{(0)}^{\frac{s}{2}-\frac 14}H^{1+\eta}}\nonumber\\
	\le &C(N,\tiv_0,\tiG_0)T^{2\delta}\| \tiX^{(n)}-\tiomega \|_{H_{(0)}^{\frac{s}{2}-\frac 14}H^{2+\eta}}^2\nonumber\\
	\le &C(N,\tiv_0,\tiG_0)T^{4\delta}\nonumber
\end{align*}
for $ \eta, \delta>0 $  small enough.
Similarly, we have
\begin{align*}
	\|\bar{d}_i^h\|_{H_{(0)}^{\frac{s}{2}-\frac 14}L^2}\le C(N,\tiv_0,\tiG_0)T^{\delta_i}
\end{align*}
for some $ \delta_i>0,   i=2,3,\cdots, 15. $ Therefore, we obtain
\begin{align}
	\|\tih^{(n)}_w\|_{\mathcal{K}_{(0)}^{s-\frac1 2}}+&\|\tih^{(n)}_{w^\top }\|_{\mathcal{K}_{(0)}^{s-\frac 12}}+\|\tih^{(n)}_\phi\|_{\mathcal{K}_{(0)}^{s-\frac 12}}+\|\tih^{(n)}_{\phi^\top }\|_{\mathcal{K}_{(0)}^{s-\frac 12}}+\|\tih^{(n)}_q\|_{\mathcal{K}_{(0)}^{s-\frac 12}}\nonumber\\
	+&\|\tih^{(n)}_G-\tih_{G_0}\|_{\mathcal{K}_{(0)}^{s-\frac 12}}\le C(N, \tiv_0, \tiG_0) T^\beta\label{equ4.18}
\end{align}
for some $ \beta>0. $

The proof of the first part  is complete if we substitute  estimates \eqref{equ4.15}, \eqref{equ4.16} and  \eqref{equ4.18} into \eqref{equ4.13}.

\textbf{Part 2.}

Note that
\begin{align}
	&L(\tiw^{(n+1)}-\tiw^{(n)}, \tiq^{(n+1)}_w-\tiq^{(n)}_w)
	=(\tilde{f}^{(n)}-\tilde{f}^{(n-1)}, \bar{g}^{(n)}-\bar{g}^{(n-1)}, \tilde{h}^{(n)}-\tilde{h}^{(n-1)}, 0)\nonumber
\end{align}
and  from Lemma \ref{lem4.1}, we have
\begin{align}
	&\|(\tiw^{(n+1)}-\tiw^{(n)}, \tiq^{(n+1)}_w-\tiq^{(n)}_w)\|_{X_0}\nonumber\\
	=&\|L^{-1}(\tilde{f}^{(n)}-\tilde{f}^{(n-1)}, \bar{g}^{(n)}-\bar{g}^{(n-1)}, \tilde{h}^{(n)}-\tilde{h}^{(n-1)}, 0)\|_{X_0}\nonumber\\
	\le & C(\|\tilde{f}^{(n)}-\tilde{f}^{(n-1)}\|_{\mathcal{K}_{(0)}^{s-1}}+\|\bar{g}^{(n)}-\bar{g}^{(n-1)}\|_{\bar{\mathcal{K}}_{(0)}^{s}}+\| \tilde{h}^{(n)}-\tilde{h}^{(n-1)}\|_{\mathcal{K}_{(0)}^{s-\frac 12}}).\nonumber
\end{align}
Therefore, it is enough to estimate $  \|\tilde{f}^{(n)}-\tilde{f}^{(n-1)}\|_{\mathcal{K}_{(0)}^{s-1}}, \|\bar{g}^{(n)}-\bar{g}^{(n-1)}\|_{\bar{\mathcal{K}}_{(0)}^{s}}$  and  $\| \tilde{h}^{(n)}-\tilde{h}^{(n-1)}\|_{\mathcal{K}_{(0)}^{s-\frac 12}}$. 

\textbf{Estimate for $ \tif^{(n)}-\tif^{(n-1)} $.} 
In \eqref{equ4.14}, we have split the term $ \tif^{(n)} $ into four terms $\tif^{(n)}= \tif^{(n)}_w+\tif^{(n)}_\phi+\tif^{(n)}_q+\tif^{(n)}_G. $ The study for $ \tilde{f}^{(n)}-\tilde{f}^{(n-1)} $ involves estimates of $ \tilde{f}^{(n)}_w-\tilde{f}_w^{(n-1)},  \tilde{f}_\phi^{(n)}-\tilde{f}_\phi^{(n-1)}$  and  $  \tilde{f}_q^{(n)}-\tilde{f}_q^{(n-1)}$ which can be found in \cite[Proposition 5.4]{castro2019splash}.  Therefore, we merely show  the estimates of $ \tif^{(n)}_G- \tif^{(n-1)}_{G}$ in $ L^2H^{s-1} $ and $  H_{(0)}^{\frac{s-1}{2}}L^2. $
For the estimate in $ L^2H^{s-1} $, we use the following splitting
\begin{align*}
	&\tif^{(n)}_G- \tif^{(n-1)}_{G}\nonumber\\
	=&\nabla\tiG^{(n)}\tizeta^{(n)}J(\tiX^{(n)})\tiG^{(n)}-\nabla\tiG^{(n-1)}\tizeta^{(n-1)}J(\tiX^{(n-1)})\tiG^{(n-1)}\nonumber\\
	=&\nabla\tiG^{(n)}[(\tizeta^{(n)}-\tizeta^{(n-1)})+\tizeta^{(n-1)}][(J(\tiX^{(n)})-J(\tiX^{(n-1)}))+J(\tiX^{(n)})]\tiG^{(n)}\nonumber\\
	=&\nabla\tiG^{(n)}[(\tizeta^{(n)}-\tizeta^{(n-1)})J(\tiX^{(n)})+\tizeta^{(n-1)} (J(\tiX^{(n)})-J(\tiX^{(n-1)}))  \nonumber\\
	&+\tizeta^{(n-1)}J(\tiX^{(n-1)})]\tiG^{(n)}-\nabla\tiG^{(n-1)}\tizeta^{(n-1)}J(\tiX^{(n-1)})\tiG^{(n-1)}\nonumber\\
	=&\nabla\tiG^{(n)}(\tizeta^{(n)}-\tizeta^{(n-1)})J(\tiX^{(n)})\tiG^{(n)}+\nabla\tiG^{(n)}\tizeta^{(n-1)} (J(\tiX^{(n)})-J(\tiX^{(n-1)}) \nonumber\\
	&\tiG^{(n)}+\nabla\tiG^{(n)}\tizeta^{(n-1)}J(\tiX^{(n-1)})\tiG^{(n)}-\nabla\tiG^{(n-1)}\tizeta^{(n-1)}J(\tiX^{(n-1)})\tiG^{(n-1)} \nonumber\\
	=&[(\nabla\tiG^{(n)}-\nabla\tiG_0)+\nabla\tiG_0](\tizeta^{(n)}-\tizeta^{(n-1)})J(\tiX^{(n)})[(\tiG^{(n)}-\tiG_0)+\tiG_0]\nonumber\\
	&+[(\nabla\tiG^{(n)}-\nabla\tiG_0)+\nabla\tiG_0]\tizeta^{(n-1)} (J(\tiX^{(n)})-J(\tiX^{(n-1)})[(\tiG^{(n)}-\tiG_0)+\tiG_0] \nonumber\\
	&+[(\nabla\tiG^{(n-1)}-\nabla\tiG_0)+\nabla\tiG_0
	]\tizeta^{(n-1)}J(\tiX^{(n-1)})(\tiG^{(n)}-\tiG^{(n-1)}) \nonumber\\
	&+(\nabla\tiG^{(n)}-\nabla\tiG^{(n-1)})\tizeta^{(n-1)}J(\tiX^{(n-1)})[(\tiG^{(n-1)}-\tiG_0)+\tiG_0] \nonumber\\	
	=&:\sum_{i=1}^{4}\bar{d}_i^f+\sum_{i=5}^{8}\bar{d}_i^f+\sum_{i=9}^{10}\bar{d}_i^f+\sum_{i=11}^{12}\bar{d}_i^f.\nonumber
\end{align*}
In the above, we first avoid multiplying $  (\tizeta^{(n)}-\tizeta^{(n-1)}) $ and $ (J(\tiX^{(n)})-J(\tiX^{(n-1)}))$. Then, we deal with the difference $ \tiG^{(n)}-\tiG^{(n-1)} $ and leave only one item as the difference of the iterative steps $ n-1 $ and $ n $ in each product. Finally, we  subtract the initial values to apply the estimates.

We estimate  the most difficult terms, such as $ \bar{d}_1^f $ and $ \bar{d}_9^f $ and the others are similar. For $ \bar{d}_1^f, $ since $ \tiX^{(n)}, \tiX^{(n-1)}\in B(X) $ and $ \tiG^{(n)}\in B(G) $, we apply
Lemmas \ref{G-G0}, \ref{jx-j} and  \ref{zetan-zetan-1} to obtain
\begin{align}
	&\|\bar{d}_1^f\|_{L^2H^{s-1}}\nonumber\\
	= &\| (\nabla\tiG^{(n)}-\nabla\tiG_0)(\tizeta^{(n)}-\tizeta^{(n-1)})J(\tiX^{(n)})(\tiG^{(n)}-\tiG_0)\|_{L^2H^{s-1}}\nonumber\\
	\le  &\| \nabla\tiG^{(n)}-\nabla\tiG_0\|_{L^\infty H^{s-1}}\|\tizeta^{(n)}-\tizeta^{(n-1)}\|_{L^\infty H^{s-1}}\|J(\tiX^{(n)})\|_{L^\infty H^{s-1}}\|\tiG^{(n)}-\tiG_0\|_{L^2H^{s-1}}\nonumber\\
	\le  & T^{\frac 12}\| \tiG^{(n)}-\tiG_0\|_{L^\infty H^{s}}^2\|\tizeta^{(n)}-\tizeta^{(n-1)}\|_{L^\infty H^{s-1}}\|J(\tiX^{(n)})\|_{L^\infty H^{s-1}}\nonumber\\
	\le  & C(N,\tiv_0, \tiG_0)T\|\tiX^{(n)}-\tiX^{(n-1)}\|_{L^\infty H^{s}} \nonumber\\
	\le  & C(N,\tiv_0, \tiG_0)T^{\frac 54}\|\tiX^{(n)}-\tiX^{(n-1)}\|_{L^\infty_{\frac 14} H^{s+1}} \nonumber\\
	\le  & C(N,\tiv_0, \tiG_0)T^{\frac 54}\|\tiX^{(n)}-\tiX^{(n-1)}\|_{\mathcal{A}^{s+1,\gamma+1}}. \nonumber
\end{align}
For $ \bar{d}_9^f, $ we apply
Lemmas \ref{G-G0},  \ref{jx-j} and  \ref{zeta-} to get
\begin{align}
	&\|\bar{d}_9^f\|_{L^2H^{s-1}}\nonumber\\
	= &\| (\nabla\tiG^{(n-1)}-\nabla\tiG_0)
	\tizeta^{(n-1)}J(\tiX^{(n-1)})(\tiG^{(n)}-\tiG^{(n-1)})\|_{L^2H^{s-1}}\nonumber\\
	\le  &\| \nabla\tiG^{(n)}-\nabla\tiG_0\|_{L^\infty H^{s-1}}\|\tizeta^{(n-1)}\|_{L^\infty H^{s-1}}\|J(\tiX^{(n-1)})\|_{L^\infty H^{s-1}}\|\tiG^{(n)}-\tiG^{(n-1)}\|_{L^2H^{s-1}}\nonumber\\
	\le  & T^{\frac 12}\| \tiG^{(n)}-\tiG_0\|_{L^\infty H^{s}}\|\tizeta^{(n-1)}\|_{L^\infty H^{s-1}}\|J(\tiX^{(n-1)})\|_{L^\infty H^{s-1}}\|\tiG^{(n)}-\tiG^{(n-1)}\|_{L^\infty H^{s-1}}\nonumber\\
	\le  & C(N,\tiv_0, \tiG_0)T^{\frac 34}\|\tiG^{(n)}-\tiG^{(n-1)}\|_{L^\infty H^{s-1}} \nonumber\\ 
	\le  & C(N,\tiv_0, \tiG_0)T \|\tiX^{(n)}-\tiX^{(n-1)}\|_{\mathcal{A}^{s ,\gamma }}. \nonumber
\end{align}
Similarly, we have 
\begin{align}
	\sum_{i=2}^{8}\|\bar{d}_i^f\|_{L^2H^{s-1}}\le& C(N,\tiv_0, \tiG_0)T^{\delta_1}\|\tiX^{(n)}-\tiX^{(n-1)}\|_{\mathcal{A}^{s+1,\gamma+1}},\nonumber\\
	\sum_{i=10}^{12}\|\bar{d}_i^f\|_{L^2H^{s-1}}\le& C(N,\tiv_0, \tiG_0)T^{\delta_2}\|\tiG^{(n)}-\tiG^{(n-1)}\|_{\mathcal{A}^{s,\gamma}}, \nonumber
\end{align}
for some $ \delta_1,\delta_2>0. $

It remains to show the estimates in
$ H_{(0)}^{\frac{s-1}{2}}L^2. $ We use the following splitting
\begin{align}
	&\tif^{(n)}_G- \tif^{(n-1)}_{G}\nonumber\\
	=&\nabla\tiG^{(n)}\tizeta^{(n)}J(\tiX^{(n)})\tiG^{(n)}-\nabla\tiG^{(n-1)}\tizeta^{(n-1)}J(\tiX^{(n-1)})\tiG^{(n-1)}\nonumber\\
	=&(\nabla\tiG^{(n)}-\nabla\tiG^{(n-1)})\tizeta^{(n)}J(\tiX^{(n)})\tiG^{(n)}\nonumber\\
	&+\nabla\tiG^{(n-1)}(\tizeta^{(n)}-\tizeta^{(n-1)})J(\tiX^{(n)})\tiG^{(n)}\nonumber\\
	&+\nabla\tiG^{(n-1)}\tizeta^{(n-1)}(J(\tiX^{(n)})-J(\tiX^{(n-1)}))\tiG^{(n)}\nonumber\\
	&+\nabla\tiG^{(n-1)}\tizeta^{(n-1)}J(\tiX^{(n-1)})(\tiG^{(n)}-\tiG^{(n-1)})\nonumber\\
	=&:\sum_{i=1}^{4}I_i.\nonumber
\end{align}
We only show the estimates of $ I_1 $ and $ I_3. $ For $ I_1, $ we have the following splitting:
\begin{align}
	I_1=&(\nabla\tiG^{(n)}-\nabla\tiG^{(n-1)})\tizeta^{(n)}J(\tiX^{(n)})\tiG^{(n)}\nonumber\\
	=&(\nabla\tiG^{(n)}-\nabla\tiG^{(n-1)})[(\tizeta^{(n)}-\id)+\id][(J(\tiX^{(n)})-J)+J][(\tiG^{(n)}-\tiG_0)+\tiG_0]\nonumber\\
	=&:\sum_{i=1}^{8}d_{1,i}^f.\nonumber
\end{align}
For $ d_{1,1}^f=(\nabla\tiG^{(n)}-\nabla\tiG^{(n-1)})(\tizeta^{(n)}-\id)(J(\tiX^{(n)})-J)(\tiG^{(n)}-\tiG_0),  $ since $ \tiX^{(n-1)}$, $\tiX^{(n)} \in B(X) $ and $ \tiG^{(n-1)}, \tiG^{(n)}\in B(G), $ we apply Lemmas \ref{x-omega},  \ref{G-G0},  \ref{jx-j}, \ref{zeta-} and  \ref{lem3.7} to obtain
\begin{align*}
	&\|d_{1,1}^f\|_{H_{(0)}^{\frac{s-1}{2}}L^2}\nonumber\\
	=&\|(\nabla\tiG^{(n)}-\nabla\tiG^{(n-1)})(\tizeta^{(n-1)}-\id)(J(\tiX^{(n)})-J)(\tiG^{(n)}-\tiG_0)\|_{H_{(0)}^{\frac{s-1}{2}}L^2}\nonumber\\
	\le &\|(\nabla\tiG^{(n)}-\nabla\tiG^{(n-1)})(\tizeta^{(n-1)}-\id)(J(\tiX^{(n)})-J)\|_{H_{(0)}^{\frac{s-1}{2}}L^2}\|\tiG^{(n)}-\tiG_0\|_{H_{(0)}^{\frac{s-1}{2}}H^{1+\eta}}\nonumber\\
	\le&\|\nabla\tiG^{(n)}-\nabla\tiG^{(n-1)}\|_{H_{(0)}^{\frac{s-1}{2}}L^2}\|(\tizeta^{(n-1)}-\id)(J(\tiX^{(n)})-J )\|_{H_{(0)}^{\frac{s-1}{2}}H^{1+\eta}}\\ &\cdot\|\tiG^{(n)}-\tiG_0\|_{H_{(0)}^{\frac{s-1}{2}}H^{1+\eta}}\nonumber\\
	\le&\|\tiG^{(n)}-\tiG^{(n-1)}\|_{H_{(0)}^{\frac{s-1}{2}}H^1}\|\tizeta^{(n-1)}-\id\|_{H_{(0)}^{\frac{s-1}{2}}H^{1+\eta}}\|J(\tiX^{(n)})-J\|_{H_{(0)}^{\frac{s-1}{2}}H^{1+\eta}}\nonumber\\
	&\cdot\|\tiG^{(n)}-\tiG_0\|_{H_{(0)}^{\frac{s-1}{2}}H^{1+\eta}}\nonumber\\
	\le & T^{3\delta}\|\tiG^{(n)}-\tiG^{(n-1)}\|_{\mathcal{A}^{s,\gamma}},\nonumber
\end{align*}
where $ \delta,\eta>0 $ are sufficiently small.
For $ d_{1,8}^f=(\nabla\tiG^{(n)}-\nabla\tiG^{(n-1)})J\tiG_0, $ we apply
Lemmas \ref{lem3.3} and  \ref{lem3.7} to have
\begin{align}
	\|d_{1,8}^f\|_{H_{(0)}^{\frac{s-1}{2}}L^2}
	=&\|(\nabla\tiG^{(n)}-\nabla\tiG^{(n-1)})J\tiG_0\|_{H_{(0)}^{\frac{s-1}{2}}L^2}\nonumber\\
	\le &\|(\nabla\tiG^{(n)}-\nabla\tiG^{(n-1)})\|_{H_{(0)}^{\frac{s-1}{2}}L^2}\|J\tiG_0\|_{H_{(0)}^{\frac{s-1}{2}}H^{1+\eta}}\nonumber\\
	\le&C(\tiG_0)\|\tiG^{(n)}-\tiG^{(n-1)}\|_{H_{(0)}^{\frac{s-1}{2}}H^{1}}\nonumber\\
	\le&C(N, \tiv_0, \tiG_0)\left\|\partial_t \int_{0}^t\tiG^{(n)}-\tiG^{(n-1)}d\tau \right\|_{H_{(0)}^{\frac{s-1}{2}+\delta_1-\delta_1}H^{1}}\nonumber\\
	\le & C(N, \tiv_0, \tiG_0)T^{\delta_1}\|\tiG^{(n)}-\tiG^{(n-1)}\|_{H_{(0)}^{\frac{s-1}{2}+\delta_1}H^{1}}\nonumber\\
	\le & C(N, \tiv_0, \tiG_0)T^{\delta_1}\|\tiG^{(n)}-\tiG^{(n-1)}\|_{\mathcal{A}^{s,\gamma}},\nonumber
\end{align}
where $ \delta_1,\eta>0 $ are sufficiently small. We rewrite $ I_3 $ as
\begin{align*}
	I_3=&\nabla\tiG^{(n-1)}\tizeta^{(n-1)}(J(\tiX^{(n)})-J(\tiX^{(n-1)}))\tiG^{(n)}\nonumber\\
	=&[(\nabla\tiG^{(n-1)}-\nabla\tiG_0)+\nabla\tiG_0][(\tizeta^{(n-1)}-\id)+\id]\\
	&\cdot(J(\tiX^{(n)})-J(\tiX^{(n-1)}))[(\tiG^{(n)}-\tiG_0)+\tiG_0]\nonumber\\
	=&:\sum_{i=1}^{8}d_{3,i}^f.\nonumber
\end{align*}
For $ d_{3,1}^f=(\nabla\tiG^{(n-1)}-\nabla\tiG_0)(\tizeta^{(n-1)}-\id)(J(\tiX^{(n)})-J(\tiX^{(n-1)}))(\tiG^{(n)}-\tiG_0),  $ we apply Lemmas  \ref{x-omega},  \ref{G-G0}, \ref{jx-jy}, \ref{zeta-} and \ref{lem3.7} to obtain
\begin{align*}
	&\|d_{3,1}^f\|_{H_{(0)}^{\frac{s-1}{2}}L^2}\nonumber\\
	=&\|(\nabla\tiG^{(n-1)}-\nabla\tiG_0)(\tizeta^{(n-1)}-\id)(J(\tiX^{(n)})-J(\tiX^{(n-1)}))(\tiG^{(n)}-\tiG_0)\|_{H_{(0)}^{\frac{s-1}{2}}L^2}\nonumber\\
	\le &\|(\nabla\tiG^{(n-1)}-\nabla\tiG_0)(\tizeta^{(n-1)}-\id)(J(\tiX^{(n)})-J(\tiX^{(n-1)}))\|_{H_{(0)}^{\frac{s-1}{2}}L^2}\\
	&\cdot\|\tiG^{(n)}-\tiG_0\|_{H_{(0)}^{\frac{s-1}{2}}H^{1+\eta}}\nonumber\\
	\le&\|\nabla\tiG^{(n-1)}-\nabla\tiG_0\|_{H_{(0)}^{\frac{s-1}{2}}L^2}\|(\tizeta^{(n-1)}-\id)(J(\tiX^{(n)})-J(\tiX^{(n-1)}))\|_{H_{(0)}^{\frac{s-1}{2}}H^{1+\eta}}\nonumber\\
	&\cdot\|\tiG^{(n)}-\tiG_0\|_{H_{(0)}^{\frac{s-1}{2}}H^{1+\eta}}\nonumber\\
	\le&\|\tiG^{(n-1)}-\tiG_0\|_{H_{(0)}^{\frac{s-1}{2}}H^1}\|\tizeta^{(n-1)}-\id\|_{H_{(0)}^{\frac{s-1}{2}}H^{1+\eta}}\|J(\tiX^{(n)})-J(\tiX^{(n-1)})\|_{H_{(0)}^{\frac{s-1}{2}}H^{1+\eta}}\nonumber\\
	&\cdot\|\tiG^{(n)}-\tiG_0\|_{H_{(0)}^{\frac{s-1}{2}}H^{1+\eta}}\nonumber\\
	\le & T^{3\delta}\|\tiX^{(n)}-\tiX^{(n-1)}\|_{\mathcal{A}^{s+1,\gamma+1}},\nonumber
\end{align*}
where $ \delta,\eta>0 $ are sufficiently small.
For $ d_{3,8}^f=\nabla\tiG_0(J(\tiX^{(n)})-J(\tiX^{(n-1)}))\tiG_0, $ we apply
 Lemmas \ref{jx-jy} , \ref{lem3.3} and \ref{lem3.7} to obtain
\begin{align*}
	\|d_{3,8}^f\|_{H_{(0)}^{\frac{s-1}{2}}L^2}
	=&\|\nabla\tiG_0(J(\tiX^{(n)})-J(\tiX^{(n-1)}))\tiG_0\|_{H_{(0)}^{\frac{s-1}{2}}L^2}\nonumber\\
	\le &\|\nabla\tiG_0(J(\tiX^{(n)})-J(\tiX^{(n-1)}))\|_{H_{(0)}^{\frac{s-1}{2}}L^2}\|\tiG_0\|_{H_{(0)}^{\frac{s-1}{2}}H^{1+\eta}}\nonumber\\
	\le&\|\nabla\tiG_0\|_{H_{(0)}^{\frac{s-1}{2}}L^2}\|J(\tiX^{(n)})-J(\tiX^{(n-1)})\|_{H_{(0)}^{\frac{s-1}{2}}H^{1+\eta}}\|\tiG_0\|_{H_{(0)}^{\frac{s-1}{2}}H^{1+\eta}}\nonumber\\
	\le&C(\tiG_0)\|J(\tiX^{(n)})-J(\tiX^{(n-1)})\|_{H_{(0)}^{\frac{s-1}{2}}H^{1+\eta}}\nonumber\\
	\le&C(N, \tiv_0, \tiG_0)\left\|\partial_t \int_{0}^t\tiX^{(n)}-\tiX^{(n-1)}d\tau \right\|_{H_{(0)}^{\frac{s-1}{2}+\delta_1-\delta_1}H^{1+\eta}}\nonumber\\
	\le & C(N, \tiv_0, \tiG_0)T^{\delta_1}\|\tiX^{(n)}-\tiX^{(n-1)}\|_{H_{(0)}^{\frac{s-1}{2}+\delta_1}H^{1+\eta}}\nonumber\\
	\le & C(N, \tiv_0, \tiG_0)T^{\delta_1}\|\tiX^{(n)}-\tiX^{(n-1)}\|_{\mathcal{A}^{s+1,\gamma+1}},\nonumber
\end{align*}
where $ \delta_1,\eta>0 $ are sufficiently small. Similarly, we split $ I_2 $ and $ I_4 $  as follows:
\begin{align*}
	I_2=&\nabla\tiG^{(n-1)}(\tizeta^{(n)}-\tizeta^{(n-1)})J(\tiX^{(n)})\tiG^{(n)}\nonumber\\
	=&[(\nabla\tiG^{(n-1)}-\nabla\tiG_0)+\nabla\tiG_0](\tizeta^{(n)}-\tizeta^{(n-1)})\\
	&\cdot[(J(\tiX^{(n)}-J)+J)][(\tiG^{(n)}-\tiG_0)+\tiG_0],\nonumber\\
	I_4=&\nabla\tiG^{(n-1)}\tizeta^{(n-1)}J(\tiX^{(n-1)})(\tiG^{(n)}-\tiG^{(n-1)})\nonumber\\
	=&[(\nabla\tiG^{(n-1)}-\nabla\tiG_0)+\nabla\tiG_0][(\tizeta^{(n-1)}-\id)+\id]\\
	&\cdot[(J(\tiX^{(n)}-J)+J)](\tiG^{(n)}-\tiG^{(n-1)})\nonumber
\end{align*}
and we conclude  
\begin{align*}
	\|I_1\|_{H_{(0)}^{\frac{s-1}{2}}L^2}\le &C(N,\tiv_0, \tiG_0)T^{\delta_1}\|\tiG^{(n)}-\tiG^{(n-1)}\|_{\mathcal{A}^{s,\gamma}},\nonumber\\
	\|I_2\|_{H_{(0)}^{\frac{s-1}{2}}L^2}\le &C(N,\tiv_0, \tiG_0)T^{\delta_2}\|\tiX^{(n)}-\tiX^{(n-1)}\|_{\mathcal{A}^{s+1,\gamma+1}},\nonumber\\
	\|I_3\|_{H_{(0)}^{\frac{s-1}{2}}L^2}\le &C(N,\tiv_0, \tiG_0)T^{\delta_3}\|\tiX^{(n)}-\tiX^{(n-1)}\|_{\mathcal{A}^{s+1,\gamma+1}},\nonumber\\
	\|I_4\|_{H_{(0)}^{\frac{s-1}{2}}L^2}\le &C(N,\tiv_0, \tiG_0)T^{\delta_4}\|\tiG^{(n)}-\tiG^{(n-1)}\|_{\mathcal{A}^{s,\gamma}}, \nonumber
\end{align*}
for some $ \delta_i>0. $

\textbf{Estimate for $ \tig^{(n)}-\tig^{(n-1)} $.}
As in \cite[Proposition 5.4]{castro2019splash}, we have
\begin{align*}
	&\|\tig^{(n)}-\tig^{(n-1)}\|_{\bar{\mathcal{K}}_{(0)}^s}\\
	\le& C(N,\tiv_0, \tiG_0)T^{\theta}(\|\tiX^{(n)}-\tiX^{(n-1)}\|_{\mathcal{A}^{s+1,\gamma+1}}+\|\tiw^{(n)}-\tiw^{(n-1)}\|_{\mathcal{K}_{(0)}^{s+1}}).
\end{align*}

\textbf{Estimate for $ \tih^{(n)}-\tih^{(n-1)} $.}
In \eqref{equ4.17}, we split $ \tih^{(n)} $ into six terms
$ \tih^{(n)}=\tih^{(n)}_w+\tih^{(n)}_{w^\top }+\tih^{(n)}_\phi+\tih^{(n)}_{\phi^\top }+\tih^{(n)}_q+\tih^{(n)}_G. $
We merely show the estimates of $ \tih^{(n)}_G- \tih^{(n-1)}_{G}$ in $ L^2H^{s-\frac 12} $ and $ H_{(0)}^{\frac{s}{2}-\frac 14}L^2 $, and the other estimates  can be found in \cite[Proposition 5.4]{castro2019splash}.
We start with the estimate in   $ L^2H^{s-\frac 12} $ and we use the following splitting:
\begin{align}
	&\tih^{(n-1)}_G- \tih^{(n)}_{G}\nonumber\\
	=& \tiG^{(n)}\tiG^{(n)\top}J(\tiX^{(n)})^{-1}\nabla_{\Lambda}\tiX^{(n)}\tilde{n}_0-\tiG^{(n-1)}\tiG^{(n-1)\top}J(\tiX^{(n-1)})^{-1}\nabla_{\Lambda}\tiX^{(n-1)}\tilde{n}_0\nonumber\\	
	=& \tiG^{(n)}\tiG^{(n)\top}(J(\tiX^{(n)})^{-1}-J(\tiX^{(n-1)})^{-1})\nabla_{\Lambda}\tiX^{(n)}\tilde{n}_0\nonumber\\
	&+ \tiG^{(n)}\tiG^{(n)\top}J(\tiX^{(n-1)})^{-1}\nabla_{\Lambda}\tiX^{(n)}\tilde{n}_0\nonumber\\
	&-\tiG^{(n-1)}\tiG^{(n-1)\top}J(\tiX^{(n-1)})^{-1}\nabla_{\Lambda}\tiX^{(n-1)}\tilde{n}_0\nonumber\\
	=& [(\tiG^{(n)}-\tiG_0)+\tiG_0][(\tiG^{(n)}-\tiG_0)+\tiG_0]^\top (J(\tiX^{(n)})^{-1}-J(\tiX^{(n-1)})^{-1})\nabla_{\Lambda}\tiX^{(n)}\tilde{n}_0\nonumber\\
	&+ \tiG^{(n)}\tiG^{(n)\top}J(\tiX^{(n-1)})^{-1}[(\nabla_{\Lambda}\tiX^{(n)}-\nabla_{\Lambda}\tiX^{(n-1)})]\tilde{n}_0\nonumber\\
	&+ [\tiG^{(n)}\tiG^{(n)\top}-\tiG^{(n-1)}\tiG^{(n-1)\top}]J(\tiX^{(n-1)})^{-1}\nabla_{\Lambda}\tiX^{(n-1)}\tilde{n}_0\nonumber\\
	=& [(\tiG^{(n)}-\tiG_0)+\tiG_0][(\tiG^{(n)}-\tiG_0)+\tiG_0]^\top (J(\tiX^{(n)})^{-1}-J(\tiX^{(n-1)})^{-1})\nabla_{\Lambda}\tiX^{(n)}\tilde{n}_0\nonumber\\
	&+ [(\tiG^{(n)}-\tiG_0)+\tiG_0][(\tiG^{(n)}-\tiG_0)+\tiG_0]^\top J(\tiX^{(n-1)})^{-1}[(\nabla_{\Lambda}\tiX^{(n)}-\nabla_{\Lambda}\tiX^{(n-1)})]\tilde{n}_0\nonumber\\
	&+ [(\tiG^{(n)}-\tiG_0)+\tiG_0](\tiG^{(n)\top}-\tiG^{(n-1)\top})J(\tiX^{(n-1)})^{-1}\nabla_{\Lambda}\tiX^{(n-1)}\tilde{n}_0\nonumber\\
	&+ (\tiG^{(n)}-\tiG^{(n-1)})[(\tiG^{(n-1)}-\tiG_0)+\tiG_0]^\top J(\tiX^{(n-1)})^{-1}\nabla_{\Lambda}\tiX^{(n-1)}\tilde{n}_0\nonumber\\
	=&:\sum_{i=1}^{4}\bar{d}_i^h+\sum_{i=4}^{8}\bar{d}_i^h+\sum_{i=9}^{10}\bar{d}_i^h+\sum_{i=11}^{12}\bar{d}_i^h.\nonumber
\end{align}
In the above, we first deal with  $(J(\tiX^{(n)}))^{-1}$  and $(J(\tiX^{(n-1)}))^{-1}$. Then, we  split $ \tiG^{(n)}-\tiG^{(n-1)} $ in each product and  subtract the initial values for $ \tiG^{(n)} $ and $ \tiG^{(n-1)} $ to apply the estimates.

We show the estimate for $ \bar{d}_1^h $ and $ \bar{d}_9^h $ since  the others can be deduced similarly.
For $ \bar{d}_1^h, $ from $ \tiX^{(n-1)}, \tiX^{(n)} \in B(X), \tiG^{(n)}\in B(G), $ and  \eqref{Jni},   we apply Lemmas \ref{x-omega} and  \ref{G-G0} to obtain
\begin{align}
	&\|\bar{d}_1^h\|_{L^2H^{s-\frac 12}} \nonumber\\
	= &\| (\tiG^{(n)}-\tiG_0)(\tiG^{(n)}-\tiG_0)^\top (J(\tiX^{(n)})^{-1}-J(\tiX^{(n-1)})^{-1})\nabla_{\Lambda}\tiX^{(n)}\tilde{n}_0\|_{L^2H^{s-\frac 12}}\nonumber\\
	\le  &\| \tiG^{(n)}-\tiG_0\|_{L^2H^{s-\frac 12}}\|(\tiG^{(n)}-\tiG_0)^\top \|_{L^\infty H^{s-\frac 12}}\|J(\tiX^{(n)})^{-1}-J(\tiX^{(n-1)})^{-1}\|_{L^\infty H^{s-\frac 12}}\nonumber\\
	&\|\nabla_{\Lambda}\tiX^{(n)}\tilde{n}_0\|_{L^\infty H^{s-\frac 12}}\nonumber\\
	\le  &C T^{\frac 12}\|\tiG^{(n)}-\tiG_0\|_{L^\infty H^{s-\frac 12}}^2\|\tiX^{(n)}-\tiX^{(n-1)}\|_{L^\infty H^{s-\frac 12}}\|\tiX^{(n)}\|_{L^\infty H^{s+\frac 12}}\nonumber\\
	\le  &C T^{\frac 12}\|\tiG^{(n)}-\tiG_0\|_{L^\infty H^{s}}^2\|\tiX^{(n)}-\tiX^{(n-1)}\|_{L^\infty H^{s}}\|\tiX^{(n)}\|_{L^\infty H^{s+1}}\nonumber\\
	\le  & C(N,\tiv_0, \tiG_0)T\|\tiX^{(n)}-\tiX^{(n-1)}\|_{L^\infty H^{s}} \nonumber\\
	\le  & C(N,\tiv_0, \tiG_0)T^{\frac 54}\|\tiX^{(n)}-\tiX^{(n-1)}\|_{L^\infty_{\frac 14} H^{s+1}} \nonumber\\
	\le  & C(N,\tiv_0, \tiG_0)T^{\frac 54}\|\tiX^{(n)}-\tiX^{(n-1)}\|_{\mathcal{A}^{s+1,\gamma+1}}. \nonumber
\end{align}
For $ \bar{d}_9^h, $ we have from \eqref{Jni},  Lemmas \ref{x-omega} and  \ref{G-G0} that
\begin{align}
	&\|\bar{d}_9^h\|_{L^2H^{s-\frac 12}}\nonumber\\
	= &\| (\tiG^{(n)}-\tiG_0)(\tiG^{(n)\top}-\tiG^{(n-1)\top})J(\tiX^{(n-1)})^{-1}\nabla_{\Lambda}\tiX^{(n-1)}\tilde{n}_0\|_{L^2H^{s-\frac 12}}\nonumber\\
	\le  &\| \tiG^{(n)}-\tiG_0\|_{L^2H^{s-\frac 12}}\|(\tiG^{(n)}-\tiG^{(n-1)})^\top \|_{L^\infty H^{s-\frac 12}}\|J(\tiX^{(n-1)})^{-1}\|_{L^\infty H^{s-\frac 12}}\nonumber\\
	&\cdot\|\nabla_{\Lambda}\tiX^{(n)}\|_{L^\infty H^{s-\frac 12}}\nonumber\\
	\le  &C T^{\frac 12}\|\tiG^{(n)}-\tiG_0\|_{L^\infty H^{s-\frac 12}}\|\tiG^{(n)}-\tiG^{(n-1)}\|_{L^\infty H^{s-\frac 12}}\|\tiX^{(n)}\|_{L^\infty H^{s+\frac 12}}^2\nonumber\\
	\le  &C T^{\frac 12}\|\tiG^{(n)}-\tiG_0\|_{L^\infty H^{s}}\|\tiG^{(n)}-\tiG^{(n-1)}\|_{L^\infty H^{s}}\|\tiX^{(n)}\|_{L^\infty H^{s+1}}\nonumber\\
	\le  & C(N,\tiv_0, \tiG_0)T^{\frac34}\|\tiG^{(n)}-\tiG^{(n-1)}\|_{L^\infty H^{s}} \nonumber\\
	\le  & C(N,\tiv_0, \tiG_0)T \|\tiG^{(n)}-\tiG^{(n-1)}\|_{L^\infty_{\frac 14} H^{s }} \nonumber\\
	\le  & C(N,\tiv_0, \tiG_0)T \|\tiG^{(n)}-\tiG^{(n-1)}\|_{\mathcal{A}^{s,\gamma}}. \nonumber
\end{align}
Similarly, we have
\begin{align*}
	\sum_{i=2}^{8}\|\bar{d}_i^h\|_{L^2H^{s-\frac 12}}\le  & C(N,\tiv_0, \tiG_0)T^{\delta_1} \|\tiG^{(n)}-\tiG^{(n-1)}\|_{\mathcal{A}^{s,\gamma}},
	\nonumber\\
	\sum_{i=10}^{12}\|\bar{d}_i^h\|_{L^2H^{s-\frac 12}}\le  &C(N,\tiv_0, \tiG_0)T^{\delta_2}\|\tiX^{(n)}-\tiX^{(n-1)}\|_{\mathcal{A}^{s+1,\gamma+1}},\nonumber
\end{align*}
for some $ \delta_1,\delta_2>0. $

It remains to show the estimates in $ H_{(0)}^{\frac{s}{2}-\frac 14}L^2. $ For this, we leave only one term as the difference of the iterative steps $ n-1 $ and $ n$  in each product. For the remaining terms, we  subtract the initial values to apply the estimates. Indeed, we need the following splitting:
\begin{align*}
&\tih^{(n-1)}_G- \tih^{(n)}_{G}\nonumber\\
=& \tiG^{(n)}\tiG^{(n)\top}J(\tiX^{(n)})^{-1}\nabla_{\Lambda}\tiX^{(n)}\tilde{n}_0-\tiG^{(n-1)}\tiG^{(n-1)\top}J(\tiX^{(n-1)})^{-1}\nabla_{\Lambda}\tiX^{(n-1)}\tilde{n}_0\nonumber\\
=&[(\tiG^{(n)}-\tiG^{(n-1)})]\tiG^{(n)\top}J(\tiX^{(n)})^{-1}\nabla_{\Lambda}\tiX^{(n)}\tilde{n}_0+ \tiG^{(n-1)}\tiG^{(n)\top}J(\tiX^{(n)})^{-1}\nabla_{\Lambda}\tiX^{(n)}\tilde{n}_0\nonumber\\
&-\tiG^{(n-1)}\tiG^{(n-1)\top}J(\tiX^{(n-1)})^{-1}\nabla_{\Lambda}\tiX^{(n-1)}\tilde{n}_0\nonumber\\
=&[(\tiG^{(n)}-\tiG^{(n-1)})][(\tiG^{(n)\top}-\tiG_0^\top )+\tiG_0^\top ][(J(\tiX^{(n)})^{-1}-J^{-1})+J^{-1}]\\
&\cdot[(\nabla_{\Lambda}\tiX^{(n)}-\id)+\id]\tilde{n}_0\nonumber\\
&+\tiG^{(n-1)}(\tiG^{(n)\top}J(\tiX^{(n)})^{-1}\nabla_{\Lambda}\tiX^{(n)}\tilde{n}_0-\tiG^{(n-1)\top}J(\tiX^{(n-1)})^{-1}\nabla_{\Lambda}\tiX^{(n-1)}\tilde{n}_0)\nonumber\\
=&[(\tiG^{(n)}-\tiG^{(n-1)})][(\tiG^{(n)\top}-\tiG_0^\top )+\tiG_0^\top ][(J(\tiX^{(n)})^{-1}-J^{-1})+J^{-1}]\\
&\quad\cdot[(\nabla_{\Lambda}\tiX^{(n)}-\id)+\id]\tilde{n}_0\nonumber\\
&+ [(\tiG^{(n-1)}-\tiG_0)+\tiG_0](\tiG^{(n)}-\tiG^{(n-1)})^\top [(J(\tiX^{(n)})^{-1}-J^{-1})+J^{-1}]\\
&\quad\cdot[(\nabla_{\Lambda}\tiX^{(n)}-\id)+\id]\tilde{n}_0\nonumber\\
&+ [(\tiG^{(n-1)}-\tiG_0)+\tiG_0][(\tiG^{(n-1)}-\tiG_0)+\tiG_0]^\top (J(\tiX^{(n)})^{-1}-J(\tiX^{(n-1)})^{-1})\\
&\quad\cdot[(\nabla_{\Lambda}\tiX^{(n)}-\id)+\id]\tilde{n}_0\nonumber\\
&+ [(\tiG^{(n-1)}-\tiG_0)+\tiG_0][(\tiG^{(n-1)}-\tiG_0)+\tiG_0]^\top [(J(\tiX^{(n-1)})^{-1}-J^{-1})+J^{-1}]\nonumber\\
&\quad (\nabla_{\Lambda}\tiX^{(n)}-\nabla_{\Lambda}\tiX^{(n-1)}  )\tilde{n}_0\nonumber\\
=&:\sum_{i=1}^{8}\dot{d}_i^h+\sum_{i=9}^{16}\dot{d}_i^h+\sum_{i=17}^{24}\dot{d}_i^h+\sum_{i=25}^{32}\dot{d}_i^h\nonumber
\end{align*}
where we have used 
\begin{align*}
	&\tiG^{(n)\top}J(\tiX^{(n)})^{-1}\nabla_{\Lambda}\tiX^{(n)} -\tiG^{(n-1)\top}J(\tiX^{(n-1)})^{-1}\nabla_{\Lambda}\tiX^{(n-1)}\nonumber\\
	=&(\tiG^{(n)\top}-\tiG^{(n-1)\top})J(\tiX^{(n)})^{-1}\nabla_{\Lambda}\tiX^{(n)} +\tiG^{(n-1)\top}J(\tiX^{(n)})^{-1}\nabla_{\Lambda}\tiX^{(n)} \nonumber\\&-\tiG^{(n-1)\top}J(\tiX^{(n-1)})^{-1}\nabla_{\Lambda}\tiX^{(n-1)}\nonumber\\		=&(\tiG^{(n)\top}-\tiG^{(n-1)\top})[(J(\tiX^{(n)})^{-1}-J^{-1})+J^{-1}][(\nabla_{\Lambda}\tiX^{(n)}-\id)+\id] \nonumber\\
	&+\tiG^{(n-1)\top}(J(\tiX^{(n)})^{-1}\nabla_{\Lambda}\tiX^{(n)} -J(\tiX^{(n-1)})^{-1}\nabla_{\Lambda}\tiX^{(n-1)})\nonumber\\
	=&(\tiG^{(n)\top}-\tiG^{(n-1)\top})[(J(\tiX^{(n)})^{-1}-J^{-1})+J^{-1}][(\nabla_{\Lambda}\tiX^{(n)}-\id)+\id] \nonumber\\
	&+\tiG^{(n-1)\top}(J(\tiX^{(n)})^{-1}-J(\tiX^{(n-1)})^{-1})\nabla_{\Lambda}\tiX^{(n)}\nonumber\\
	&+\tiG^{(n-1)\top}J(\tiX^{(n-1)})^{-1}(\nabla_{\Lambda}\tiX^{(n)}-\nabla_{\Lambda}\tiX^{(n-1)}  ). \nonumber
\end{align*}
For $ \dot{d}_1^h,  $ from $  \tiX^{(n)} \in B(X), \tiG^{(n-1)}, \tiG^{(n)} \in B(G) $ and \eqref{Jni}, we apply Lemmas \ref{x-omega},  \ref{G-G0}, \ref{lem3.3}, \ref{lem3.5},  \ref{lem3.7}  and   Theorem \ref{lem3.9} to have
\begin{align*}
	&\|\dot{d}_1^h\|_{ H_{(0)}^{\frac{s}{2}-\frac 14}L^2}\nonumber\\
	=&\|
	(\tiG^{(n)}-\tiG^{(n-1)})(\tiG^{(n)\top}-\tiG_0^\top )(J(\tiX^{(n)})^{-1}-J^{-1})(\nabla_{\Lambda}\tiX^{(n)}-\id)\tilde{n}_0\|_{ H_{(0)}^{\frac{s}{2}-\frac 14}L^2}\nonumber\\
	\le &\|
	(\tiG^{(n)}-\tiG^{(n-1)})(\tiG^{(n)\top}-\tiG_0^\top )\|_{ H_{(0)}^{\frac{s}{2}-\frac 14}L^2}\\
	&\cdot\|(J(\tiX^{(n)})^{-1}-J^{-1})(\nabla_{\Lambda}\tiX^{(n)}-\id)\|_{ H_{(0)}^{\frac{s}{2}-\frac 14}H^{\frac 12+\eta}}\nonumber\\
	\le &\|
	\tiG^{(n)}-\tiG^{(n-1)}\|_{ H_{(0)}^{\frac{s}{2}-\frac 14}H^{\frac 12-\mu}}\|(\tiG^{(n)\top}-\tiG_0^\top )\|_{ H_{(0)}^{\frac{s}{2}-\frac 14}H^{\frac 12+\mu}}\\
	&\cdot\|J(\tiX^{(n)})^{-1}-J^{-1}\|_{ H_{(0)}^{\frac{s}{2}-\frac 14}H^{\frac 12+\eta}}\|\nabla_{\Lambda}\tiX^{(n)}-\id\|_{ H_{(0)}^{\frac{s}{2}-\frac 14}H^{\frac 12+\eta}}\nonumber\\
	\le &\|
	\tiG^{(n)}-\tiG^{(n-1)}\|_{ H_{(0)}^{\frac{s}{2}-\frac 14}H^{1-\mu}}\|\tiG^{(n)}-\tiG_0\|_{ H_{(0)}^{\frac{s}{2}-\frac 14}H^{1+\mu}}\|J(\tiX^{(n)})^{-1}-J^{-1}\|_{ H_{(0)}^{\frac{s}{2}-\frac 14}H^{1+\eta}}\nonumber\\
	&\cdot\|\nabla_{\Lambda}\tiX^{(n)}-\id\|_{ H_{(0)}^{\frac{s}{2}-\frac 14}H^{1+\eta}}\nonumber\\
	\le &C\|
	\tiG^{(n)}-\tiG^{(n-1)}\|_{ H_{(0)}^{\frac{s}{2}-\frac 14}H^{1-\mu}}\|\tiG^{(n)}-\tiG_0\|_{ H_{(0)}^{\frac{s}{2}-\frac 14}H^{1+\mu}}\|\tiX^{(n)}-\tiomega\|_{ H_{(0)}^{\frac{s}{2}-\frac 14}H^{2+\eta}}^2\nonumber\\
	\le &C(N,\tiv_0,\tiG_0)T^{3\delta}\|
	\tiG^{(n)}-\tiG^{(n-1)}\|_{ H_{(0)}^{\frac{s}{2}-\frac 14}H^{1-\mu}}\nonumber\\
	\le&C(N,\tiv_0,\tiG_0)T^{3\delta}T^\delta\|\tiG^{(n)}-\tiG^{(n-1)}\|_{ H_{(0)}^{\frac{s}{2}-\frac 14+\delta}H^{1-\mu}}\nonumber\\	\le&C(N,\tiv_0,\tiG_0)T^{4\delta} \|\tiG^{(n)}-\tiG^{(n-1)}\|_{ \mathcal{A}^{s,\gamma}}\nonumber
\end{align*}
for $ \delta, \eta$ and $ \mu>0 $ small enough. For 
$ \dot{d}_{17}^h, $ we apply the same Lemmas as $ \dot{d}_1^h $ to obtain
\begin{align*}
	&\|\dot{d}_{17}^h\|_{ H_{(0)}^{\frac{s}{2}-\frac 14}L^2}\nonumber\\
	=&\|
	(\tiG^{(n-1)}-\tiG_0)(\tiG^{(n-1)}-\tiG_0)^\top (J(\tiX^{(n)})^{-1}-J(\tiX^{(n-1)})^{-1})\\
	&\qquad\cdot(\nabla_{\Lambda}\tiX^{(n)}-\id)\tilde{n}_0\|_{ H_{(0)}^{\frac{s}{2}-\frac 14}L^2}\nonumber\\
	\le &\|
	(\tiG^{(n-1)}-\tiG_0)(\tiG^{(n-1)\top}-\tiG_0^\top )\|_{ H_{(0)}^{\frac{s}{2}-\frac 14}L^2}\\
	&\cdot\|(J(\tiX^{(n)})^{-1}-J(\tiX^{(n-1)})^{-1})(\nabla_{\Lambda}\tiX^{(n)}-\id)\|_{ H_{(0)}^{\frac{s}{2}-\frac 14}H^{\frac 12+\eta}}\nonumber\\
	\le &\|
	\tiG^{(n-1)}-\tiG_0\|_{ H_{(0)}^{\frac{s}{2}-\frac 14}H^{\frac 12-\mu}}\|\tiG^{(n-1)}-\tiG_0\|_{ H_{(0)}^{\frac{s}{2}-\frac 14}H^{\frac 12+\mu}}\nonumber\\
	&\cdot\|J(\tiX^{(n)})^{-1}-J(\tiX^{(n-1)})^{-1}\|_{ H_{(0)}^{\frac{s}{2}-\frac 14}H^{\frac 12+\eta}}\|\nabla_{\Lambda}\tiX^{(n)}-\id\|_{ H_{(0)}^{\frac{s}{2}-\frac 14}H^{\frac 12+\eta}}\nonumber\\
	\le &C\|
	\tiG^{(n-1)}-\tiG_0\|_{ H_{(0)}^{\frac{s}{2}-\frac 14}H^{1-\mu}}\|\tiG^{(n-1)}-\tiG_0\|_{ H_{(0)}^{\frac{s}{2}-\frac 14}H^{1+\mu}}\|\tiX^{(n)}-\tiX^{(n-1)}\|_{ H_{(0)}^{\frac{s}{2}-\frac 14}H^{1+\eta}}\nonumber\\
	&\cdot\|\nabla_{\Lambda}\tiX^{(n)}-\id\|_{ H_{(0)}^{\frac{s}{2}-\frac 14}H^{1+\eta}}\nonumber\\
	\le &C(N,\tiv_0,\tiG_0)T^{2\delta}\|\tiX^{(n)}-\tiomega\|_{ H_{(0)}^{\frac{s}{2}-\frac 14}H^{2+\eta}}\|\tiX^{(n)}-\tiX^{(n-1)}\|_{ H_{(0)}^{\frac{s}{2}-\frac 14}H^{1+\eta}}\nonumber\\
	\le &C(N,\tiv_0,\tiG_0)T^{3\delta}\|\tiX^{(n)}-\tiX^{(n-1)}\|_{ H_{(0)}^{\frac{s}{2}-\frac 14}H^{1+\eta}}\nonumber\\
	\le&C(N,\tiv_0,\tiG_0)T^{3\delta} \|
	\tiX^{(n)}-\tiX^{(n-1)}\|_{ \mathcal{A}^{s+1,\gamma+1}}\nonumber
\end{align*}
for $ \delta, \eta$ and $ \mu>0 $ small enough.
Similarly, we have
\begin{align}
	\sum_{i=2}^{16}\|\dot{d}_i^h\|_{ H_{(0)}^{\frac{s}{2}-\frac 14}L^2}\le&C(N,\tiv_0,\tiG_0)T^{\delta_1} \|
	\tiG^{(n)}-\tiG^{(n-1)}\|_{ \mathcal{A}^{s,\gamma}},\nonumber\\
	\sum_{i=18}^{32}\|\dot{d}_i^h\|_{ H_{(0)}^{\frac{s}{2}-\frac 14}L^2}\le&C(N,\tiv_0,\tiG_0)T^{\delta_2} \|
	\tiX^{(n)}-\tiX^{(n-1)}\|_{ \mathcal{A}^{s+1,\gamma+1}}\nonumber
\end{align}
for some $ \delta_1,\delta_2>0. $

The proof is complete if we put all the estimates together. 

Now, we have proved Propositions \ref{pro2} and \ref{pro3}. Let us explain how to prove Proposition \ref{pro1} and Theorem \ref{thm1}. Indeed,  the initial iterative sequence satisfies the conditions in the first part of Lemma \ref{lem1}, Propositions \ref{pro2} and  \ref{pro3}, so we know that all the iterative sequences belong to the corresponding spaces and uniformly lie in the predetermined balls when $ T>0 $ is sufficiently small.  From this, the conditions in the second part of Lemma \ref{lem1}, Propositions \ref{pro2} and  \ref{pro3} automatically hold for any $ n\ge 1. $ Therefore, we obtain the desired estimates and prove that for small $ T>0, $ the iterative sequence is  Cauchy  and its limit satisfies the same estimates. In particular, we have solved system \eqref{MHD Lagrangian coordinate2}. 
\end{proof}

\section{Stability estimates}\label{Stability estimates}

\par Let  $ (\tiw, \tiq_w, \tiX, \tiG) $ be a solution in $ \tiOmega_0 $  with initial data $ \tiv_0 $ and $\tiG_0$ as in Theorem \ref{thm1}. In order to prove the stability, we recall  that $ \tiOmega_0=P(\Omega_0) $ and we choose a family of initial data $ \tiv_{\vare, 0}^\prime,\tiG_{\vare,0}^\prime $  and $ \tiOmega_\vare(0) $ such that
\begin{align*}
	\tiOmega_{\vare}(0)=\tiOmega_0+\vare b,
\end{align*}
where $ |b|=1 $ is a constant vector such that $ P^{-1}(\tiOmega_{\vare}(0)) $ is not a self-intersecting domain as in Fig. \ref{fig:1} and Fig. \ref{fig:4}.

%

Let $ (\tiw_\vare^\prime, \tiq_{w,\vare}^\prime, \tiX_\vare^\prime,\tiG_\vare^\prime) $ be a solution defined in $ \tiOmega_\vare(0) $  with initial data $ \tiv_{\vare, 0}^\prime,\tiG_{\vare,0}^\prime $ and a constant $ N_\vare $  from Theorem \ref{thm1}. Indeed, we need another $ \phi_\vare^\prime $ to apply Theorem \ref{thm1}, i.e., 
\begin{align*}
	\phi_\vare^\prime:=\tiv_{\vare, 0}^\prime+t \hat{\phi}_\vare^\prime
	:=\tiv_{\vare, 0}^\prime+t (Q^2 \Delta \tiv_{\vare, 0}^\prime- J^\top  \nabla \tiq_{\phi,\vare}^\prime+\nabla \tiG_{\vare,0}^\prime J \tiG_{\vare,0}^\prime  ),
\end{align*}
where $\tiq_{\phi,\vare}^\prime$  satisfies
\begin{align*}
	\begin{cases}
		-Q^2 \Delta \tiq_{\phi,\vare}^\prime=\trace (\nabla \tiv_{\vare, 0}^\prime J  \nabla \tiv_{\vare, 0}^\prime J )-\trace (\nabla \tiG_{\vare, 0}^\prime J \nabla \tiG_{\vare, 0}^\prime J ), & \text { in }  \tiOmega_\vare(0), \\ 
		\tiq_{\phi,\vare}^\prime J^{-1} \tilde{n}_0=( \nabla \tiv_{\vare, 0}^\prime J+( \nabla \tiv_{\vare, 0}^\prime J )^\top +\tiG_{\vare, 0}^\prime \tiG_{\vare, 0}^{\prime \top} )J^{-1} \tilde{n}_0, & \text { on }  \partial \tiOmega_\vare(0).
	\end{cases}
\end{align*}
  
In order to compare the solutions, we define $ (\tiw_\vare, \tiq_{w,\vare}, \tiX_\vare,\tiG_\vare) $ by shifting from $ \tiOmega_\vare(0) $ to $ \tiOmega_0 $ 
\begin{align*}
	(\tiw_\vare, \tiq_{w,\vare}, \tiX_\vare,\tiG_\vare)(\tiomega,t):=(\tiw_\vare^\prime, \tiq_{w,\vare}^\prime, \tiX_\vare^\prime,\tiG_\vare^\prime)(\tiomega+\vare b,t),\quad  \tiomega\in \tiOmega_0.  
\end{align*}
Similarly, we can define $ \phi_\vare, v_\vare $ and $ \tiq_\vare. $
Meanwhile, we choose initial data  $  \tiv_{\vare, 0}^\prime,\tiG_{\vare, 0}^\prime  $ such that
\begin{align*}
	(\tiv_{\vare, 0}^\prime,\tiG_{\vare, 0}^\prime)(\tiomega+\vare b)=(\tiv_{0},\tiG_{0})(\tiomega), \quad \tiomega\in \tiOmega_0.
\end{align*}
Then we have
\begin{align*}
	(\tiv_{\vare},\tiG_{\vare})(\tiomega,0)=(\tiv,\tiG)(\tiomega,0).
\end{align*}

By Theorem \ref{thm1}, the solution
$ (\tiw, \tiq_w, \tiX,\tiG) $ satisfies the following equations if we let $ n $ tend to infinity in iterative systems  \eqref{equ4.9}, \eqref{equ4.10} and \eqref{equ4.11} 
\begin{align*}
	\begin{cases}
		\partial_t \tiw-Q^2 \Delta \tiw+ J^\top \nabla \tiq_w=\tilde{f}+ \tilde{f}_\phi^L, \\
		\trace (\nabla \tiw J )=\tig+\tig_\phi^L, \\
		 [-\tiq_w \id+ ((\nabla \tiw J )+ (\nabla \tiw J )^\top  )] J^{-1} \tilde{n}_0=\tilde{h}+\tilde{h}_\phi^L,\\
		\tiw(0, \cdot)=0,
	\end{cases}
\end{align*}
where
\begin{align*}
		\tif= & -Q^2 \Delta (\tiw+\phi)+ J^\top \nabla (\tiq_w+\tiq_\phi)+Q^2(\tiX) \nabla(\nabla(\tiw+\phi)\tizeta)\tizeta\nonumber \\
		&-J(\tiX)^\top  \tizeta^\top \nabla (\tiq_w+\tiq_\phi) +\nabla \tiG\tizeta J(\tiX)\tiG, \nonumber \\
		\tilde{f}_\phi^L =& -\partial_t \phi+Q^2 \Delta \phi- J^\top \nabla \tilde{q}_\phi, \nonumber \\
		\tig=&\trace (\nabla (\tiw+\phi) J)-\trace (\nabla (\tiw+\phi) \tizeta J (\tiX )),\nonumber \\
		\tig_\phi^L =& -\operatorname{Tr}(\nabla \phi J),\nonumber \\
		\tih= & -(\tiq_w+\tiq_\phi)J^{-1} \tilde{n}_0+(\tiq_w+\tiq_\phi)(J(\tiX))^{-1} \nabla_{\Lambda} \tiX \tilde{n}_0 \nonumber \\
		&+((\nabla (\tiw+\phi) J) +(\nabla (\tiw+\phi) J)^\top )J^{-1} \tilde{n}_0 \nonumber \\
		& - ( (\nabla (\tiw+\phi) \tizeta J (\tiX ) )+ (\nabla (\tiw+\phi) \tizeta J (\tiX ) )^\top  ) J (\tiX ) ^{-1} \nabla_{\Lambda} \tiX \tilde{n}_0\nonumber \\
		&-\tiG\tiG^\top J(\tiX)^{-1} \nabla_{\Lambda} \tiX\tilde{n}_0,\nonumber \\
		\tilde{h}_\phi^L=&\tilde{q}_\phi J^{-1} \tilde{n}_0-(\nabla \phi J+(\nabla \phi J)^\top )J^{-1} \tilde{n}_0,\nonumber
\end{align*}
with
\begin{align*} 
	\tiG (t, \tiomega)  =\tiG_0+\int_0^t\nabla(\tiw (\tau,\tiomega)+\phi(\tau,\tiomega))\tizeta (\tiomega)J(\tiX (\tau,\tiomega))\tiG (\tau,\tiomega)d\tau,
\end{align*}
and
\begin{align*} 
	\tiX (t, \tiomega)=\tiomega+\int_0^t J(\tiX (\tau,\tiomega))(\tiw (\tau,\tiomega)+\phi(\tau,\tiomega))d\tau .
\end{align*}
 
Similarly, the shifted solution  
 $ (\tiw_\vare, \tiq_{w,\vare}, \tiX_\vare,\tiG_\vare) $ defined in $ \tiOmega_0 $ satisfies
\begin{align*}
	\begin{cases}
		\partial_t \tiw_\vare-Q_\vare^2 \Delta \tiw_\vare+ J_\vare^\top  \nabla \tiq_{w,\vare}=\tilde{f}_\vare+ \tilde{f}_{\phi,\vare}^L, \\
		\trace (\nabla \tiw_\vare J_\vare )=\tig_\vare+\tig_{\phi,\vare}^L, \\
		[-\tiq_{w,\vare} \id+ ((\nabla \tiw_\vare J_\vare )+ (\nabla \tiw_\vare J_\vare )^\top  )] J_\vare^{-1} \tilde{n}_0=\tilde{h}_\vare+\tilde{h}_{\phi,\vare}^L,\\
		\tiw_\vare(0, \cdot)=0,
	\end{cases}
\end{align*}
where
\begin{align}
		\tif_\vare= & -Q_\vare^2 \Delta (\tiw_\vare+\phi_\vare)+ J_\vare^\top  \nabla (\tiq_{w,\vare}+\tiq_{\phi,\vare})+Q_\vare^2(\tiX_\vare) \nabla(\nabla(\tiw_\vare+\phi_\vare)\tizeta_\vare)\tizeta_\vare \nonumber\\
		& -J_\vare(\tiX_\vare)^\top  \tizeta_\vare^\top \nabla (\tiq_{w,\vare}+\tiq_{\phi,\vare})+\nabla \tiG_\vare\tizeta_\vare J_\vare(\tiX_\vare)\tiG_\vare,\nonumber\\
		\tilde{f}_{\phi,\vare}^L =& -\partial_t \phi_\vare+Q_\vare^2 \Delta \phi_\vare- J_\vare^\top  \nabla \tilde{q}_{\phi,\vare}, \nonumber\\
		\tig_\vare=&\trace (\nabla (\tiw_\vare+\phi_\vare) J_\vare)-\trace (\nabla (\tiw_\vare+\phi_\vare) \tizeta_\vare J_\vare (\tiX_\vare )),\nonumber\\
		\tig_{\phi,\vare}^L =& -\operatorname{Tr}(\nabla \phi_\vare J_\vare),\nonumber\\
		\tih_\vare= & -(\tiq_{w,\vare}+\tiq_{\phi,\vare})J_\vare^{-1} \tilde{n}_0+(\tiq_{w,\vare}+\tiq_{\phi,\vare})(J_\vare(\tiX_\vare))^{-1} \nabla_{\Lambda} \tiX_\vare \tilde{n}_0\nonumber\\
		&+((\nabla (\tiw_\vare+\phi_\vare) J_\vare) +(\nabla (\tiw_\vare+\phi_\vare) J_\vare)^\top )J_\vare^{-1} \tilde{n}_0 \nonumber\\
		& - ( (\nabla (\tiw_\vare+\phi_\vare) \tizeta_\vare J_\vare (\tiX_\vare ) )+ (\nabla (\tiw_\vare+\phi_\vare) \tizeta_\vare J_\vare (\tiX_\vare ) )^\top  ) J_\vare (\tiX_\vare ) ^{-1} \nabla_{\Lambda} \tiX_\vare \tilde{n}_0\nonumber\\
		&-\tiG_\vare\tiG_\vare^\top J_\vare(\tiX_\vare)^{-1} \nabla_{\Lambda} \tiX_\vare\tilde{n}_0,\nonumber\\
		\tilde{h}_{\phi,\vare}^L=&\tilde{q}_{\phi,\vare} J_\vare^{-1} \tilde{n}_0-(\nabla \phi_\vare J_\vare+(\nabla \phi_\vare J_\vare)^\top )J_\vare^{-1} \tilde{n}_0,\nonumber
\end{align}
with $ J_\vare(\tiomega):= J(\tiomega+\vare b), 	Q^2_\vare(\tiomega):=Q^2(\tiomega+\vare b), $ and
\begin{align*}
	\phi_\vare:=\tiv_0+t \hat{\phi}_\vare
	:=\tiv_0+t (Q_\vare^2 \Delta \tiv_0- J_\vare^\top  \nabla \tiq_{\phi,\vare}+\nabla \tiG_0 J_\vare \tiG_0  ),
\end{align*}
where $\tiq_{\phi,\vare}$  satisfies
\begin{align*}
	\begin{cases}
		-Q_\vare^2 \Delta \tiq_{\phi,\vare}=\trace (\nabla \tiv_0 J_\vare  \nabla \tiv_0 J_\vare  )-\trace (\nabla \tiG_0 J_\vare  \nabla \tiG_0 J_\vare  ), & \text { in }  \tiOmega_0, \\ 
		\tiq_{\phi,\vare} J_\vare^{-1} \tilde{n}_0=( \nabla \tiv_0 J_\vare+( \nabla \tiv_0 J_\vare)^\top +\tiG_0 \tiG_0^\top )J_\vare^{-1} \tilde{n}_0, & \text { on }  \partial \tiOmega_0.
	\end{cases}
\end{align*}

Now, we consider the following system for the difference $ (\tiw-\tiw_\vare, \tiq_w-\tiq_{w,\vare},$  $ \tiX- \tiX_\vare, \tiG-\tiG_\vare) $
\begin{align*}
	\begin{cases}
		\partial_t (\tiw-\tiw_\vare)-Q^2 \Delta (\tiw-\tiw_\vare)+ J^\top \nabla (\tiq_w-\tiq_{w,\vare})=\tiF_\vare, \\
		\trace (\nabla (\tiw-\tiw_\vare) J )=\tiK_\vare, \\
		[-(\tiq_w-\tiq_{w,\vare}) \id+ ((\nabla (\tiw-\tiw_\vare) J )+ (\nabla (\tiw-\tiw_\vare) J )^\top  )] J^{-1} \tilde{n}_0=\tiH_\vare,\\
		(\tiw-\tiw_\vare)(0, \cdot)=0,
	\end{cases}
\end{align*}
where 
\begin{align*}
	\tiF_\vare=& \tilde{f}-\tilde{f}_\vare+ \tilde{f}_\phi^L- \tilde{f}_{\phi,\vare}^L+(Q^2-Q_\vare^2) \Delta \tiw_\vare
	+ (J_\vare^\top -J^\top) \nabla \tiq_{w,\vare},\nonumber\\
	\tiK_\vare=&\tig-\tig_\vare+\tig_\phi^L-\tig_{\phi,\vare}^L+	\trace( \nabla \tiw_\vare (J_\vare-J) ),\nonumber\\
	\tiH_\vare=&\tilde{h}-\tilde{h}_\vare+\tilde{h}_\phi^L-\tilde{h}_{\phi,\vare}^L
	+\tiq_{w,\vare}(J^{-1}-J_\vare^{-1})\tilde{n}_0\\
	&+(\nabla \tiw_\vare J_\vare)^\top J_\vare^{-1}\tilde{n}_0-(\nabla \tiw_\vare J)^\top J^{-1}\tilde{n}_0.\nonumber
\end{align*}
For the flux and the magnetic field, we have 
\begin{align*}
	\begin{cases}
		\dfrac{d}{d t} \tiX_\vare (t, \tiomega)=J  (\tiX_\vare (t, \tiomega) ) \tiv_\vare (t, \tiomega), \\ 
		\tiX_\vare(0, \tiomega)=\tiomega+\vare b, \quad \text { in } \tiOmega_0,
	\end{cases}
\end{align*}
and 
\begin{align*}
	\begin{cases}
		\partial_t \tiG_\vare (t, \tiomega)=\nabla \tiv_\vare (t, \tiomega) \tizeta_\vare (\tiomega) J(\tiX_\vare (t, \tiomega) )   \tiG_\vare(t,\tiomega), \\
		\tiG_\vare(0, \tiomega)=\tiG_0(\tiomega), \quad \text { in } \tiOmega_0 .
	\end{cases}
\end{align*}
Therefore, we obtain
\begin{align}
	\tiG_\vare  =&\tiG_0+\int_0^t\nabla(\tiw_\vare +\phi_\vare )\tizeta_\vare J(\tiX_\vare )\tiG_\vare d\tau,\nonumber\\
	\tiX_\vare =&\tiomega+\vare b+\int_0^t J(\tiX_\vare)(\tiw_\vare+\phi_\vare ) d\tau\nonumber,
\end{align}
and so 
\begin{align}
	\tiG-\tiG_\vare  =&\int_0^t \nabla(\tiw +\phi )\tizeta J(\tiX )\tiG-\nabla(\tiw_\vare +\phi_\vare )\tizeta_\vare J(\tiX_\vare )\tiG_\vare d\tau,\nonumber\\
	\tiX-\tiX_\vare =&-\vare b+\int_0^t J(\tiX )(\tiw +\phi )- J(\tiX_\vare)(\tiw_\vare+\phi_\vare )d\tau .\nonumber
\end{align}

Now we prove the following stability result.
\begin{theorem}\label{thm2}
	Let $2<s<\frac{5}{2}, 1<\gamma<s-1$ and a suitable $\delta>0$. If $0<T< C^{-\frac{1}{\delta}},$ then
\begin{align*}
\|\tilde{X}-\tilde{X}_{\vare}\|_{L^{\infty} H^{s+1}} \leq C \vare
\end{align*}
where the constant $ C $ depends only on the initial data. In particular, we have 
\begin{equation}\label{equthm2}
	\operatorname{dist}(\partial\tiOmega(t), \partial\tiOmega_{\vare}(t))\le C\vare 
\end{equation}
for all $ t>0 $ small enough.
\end{theorem}

From Theorem  \ref{thm1}, we know that the solutions belong to the predetermined balls whose radii  depend only on the initial data. With the help of this estimate, the proof of the stability theorem is a  consequence of the following results. 
\begin{lemma}\label{lem5.2}
For $2<s<\frac{5}{2}, 1<\gamma<s-1$ and a suitable $\delta_1>0$ , we have

(1) $\|J-J_{\vare}\|_{H^r} \leq C \vare$  and $ \|Q^2-Q_{\vare}^2\|_{H^r} \leq C \vare$ for all $r\ge 0$.

(2)
 $\|\phi-\phi_{\vare}\|_{L^{\infty} H^{s+1}} \leq C \vare, $  and  $\|\phi-\phi_{\vare}\|_{H_{(0)}^1 H^r} \leq C \vare$  for smooth $\tilde{v}_0, \tilde{G}_0$ and $r\le s+1 $.
 
(3) $\|\tilde{q}_\phi-\tilde{q}_{\phi, \vare}\|_{H^{r+1}} \leq C \vare $ for all $ r \geq 0$.

(4)
\begin{align}
	 &\|\tilde{X}-\tilde{X}_{\vare}+\vare b-t (J-J_{\vare} ) \tilde{v}_0\|_{\mathcal{A}^{s+1, \gamma+1}}\nonumber\\
	  \le& C \vare +C T^{\delta_1} ( \|\tilde{X}-\tilde{X}_{\vare}+\vare b- (J-J_{\vare}  ) \tilde{v}_0 \|_{\mathcal{A}^{s+1, \gamma+1}}+ \|\tilde{w}-\tilde{w}_{\vare} \|_{\mathcal{K}_{(0)}^{s+1}} ) \label{Stability X}
\end{align}
for some constant $C$ depending only on the initial data.
\end{lemma}
\begin{proof}
	The proof can be found in \cite[Lemma 6.1]{castro2019splash}  with small modifications due to the definition of $ \phi $ and $ \phi_\vare, $ which now depend on both $ \tiv_0 $ and $ \tiG_0. $
\end{proof}

\begin{proposition}
For $2<s<\frac{5}{2}, 1<\gamma<s-1$ and a suitable $\delta_2>0$, we have 
	\begin{align} 
	&	\|\tilde{G}-\tilde{G}_{\vare}-t\nabla \tilde{v}_0(J-J_{\vare})  \tilde{G}_0\|_{\mathcal{A}^{s, \gamma}}  \nonumber\\
	\leq&   C T^{\delta_2}(\|\tilde{w}-\tilde{w}_{\vare}\|_{\mathcal{K}_{(0)}^{s+1}}+\|\tilde{X}-\tilde{X}_{\vare}+\vare b-t(J-J_{\vare}) \tilde{v}_0\|_{\mathcal{A}^{s+1, \gamma+1}}\nonumber\\
	& +\|\tilde{G}-\tilde{G}_{\vare}-t\nabla \tilde{v}_0(J-J_{\vare}) \tilde{G}_0\|_{\mathcal{A}^{s, \gamma}})+C \vare,\label{Stability G}
	\end{align}
where the constant $C$ depends only on the initial data.
\end{proposition}
\begin{proof}

First, we deal with the estimate in
$ L^\infty_{\frac 14}H^s $ and  have
\begin{align}
	&\|\tilde{G}-\tilde{G}_{\vare}-t\nabla \tilde{v}_0(J-J_{\vare})  \tilde{G}_0\|_{L^\infty_{\frac 14}H^s}\nonumber\\
	= &\sup_{t\in[0,T]}t^{-\frac 14}\left\|\int_0^t \nabla(\tiw +\phi )\tizeta J(\tiX )\tiG-\nabla(\tiw_\vare +\phi_\vare )\tizeta_\vare J(\tiX_\vare )\tiG_\vare -\nabla \tilde{v}_0(J-J_{\vare})  \tilde{G}_0d\tau \right\|_{H^s}\nonumber\\
	\leq &\sup_{t\in[0,T]}t^{\frac 14}\|  \nabla(\tiw +\phi )\tizeta J(\tiX )\tiG-\nabla(\tiw_\vare +\phi_\vare )\tizeta_\vare J(\tiX_\vare )\tiG_\vare  -\nabla \tilde{v}_0(J-J_{\vare})  \tilde{G}_0 \|_{L^2([0,t];H^s)}\nonumber\\
	\leq &T^{\frac 14}\|  \nabla(\tiw +\phi )\tizeta J(\tiX )\tiG-\nabla(\tiw_\vare +\phi_\vare )\tizeta_\vare J(\tiX_\vare )\tiG_\vare  -\nabla \tilde{v}_0(J-J_{\vare})  \tilde{G}_0 \|_{L^2([0,T];H^s)}.\nonumber
\end{align}	
We split the term $ \nabla(\tiw +\phi )\tizeta J(\tiX )\tiG-\nabla(\tiw_\vare +\phi_\vare )\tizeta_\vare J(\tiX_\vare )\tiG_\vare  -\nabla \tilde{v}_0(J-J_{\vare})  \tilde{G}_0 $ as follows:
\begin{align}
	&\nabla(\tiw +\phi )\tizeta J(\tiX )\tiG-\nabla(\tiw_\vare +\phi_\vare )\tizeta_\vare J(\tiX_\vare )\tiG_\vare  -\nabla \tilde{v}_0(J-J_{\vare})  \tilde{G}_0\nonumber\\
	=&[\nabla  \tiw \tizeta J(\tiX )\tiG-\nabla \tiw_\vare  \tizeta_\vare J(\tiX_\vare )\tiG_\vare]  +[t\nabla  \hat{\phi} \tizeta J(\tiX )\tiG-t\nabla \hat{\phi}_\vare \tizeta_\vare J(\tiX_\vare )\tiG_\vare] \nonumber\\
	&+[\nabla  \tiv_0 \tizeta J(\tiX )\tiG-\nabla \tiv_0  \tizeta_\vare J(\tiX_\vare )\tiG_\vare-\nabla \tilde{v}_0(J-J_{\vare})  \tilde{G}_0 ]\nonumber\\
	=&:\sum_{i=1}^{25}I_i.\nonumber
\end{align}
Indeed, $ I_1,\cdots, I_9 $ are defined by
\begin{align*}
	&\nabla  \tiw \tizeta J(\tiX )\tiG-\nabla \tiw_\vare  \tizeta_\vare J(\tiX_\vare )\tiG_\vare\nonumber\\
	=&\nabla  \tiw \tizeta [(J(\tiX)-J(\tiX_\vare ) - J+J_\vare)+(J-J_\vare)+J(\tiX_\vare )]\tiG-\nabla \tiw_\vare  \tizeta_\vare J(\tiX_\vare )\tiG_\vare\nonumber\\
	=&\nabla  \tiw \tizeta [(J(\tiX)-J(\tiX_\vare ) - J+J_\vare)+(J-J_\vare)]\tiG+\nabla  \tiw [(\tizeta-\tizeta_\vare)+\tizeta_\vare] J(\tiX_\vare )\tiG\\
	&-\nabla \tiw_\vare  \tizeta_\vare J(\tiX_\vare )\tiG_\vare\nonumber\\
	=&\nabla  \tiw \tizeta [(J(\tiX)-J(\tiX_\vare ) - J+J_\vare)]\tiG+\nabla  \tiw \tizeta(J-J_\vare)\tiG+\nabla  \tiw [(\tizeta-\tizeta_\vare)]J(\tiX_\vare )\tiG\nonumber\\
	&+\nabla  [(\tiw-\tiw_\vare)+\tiw_\vare] \tizeta_\vare J(\tiX_\vare )\tiG-\nabla \tiw_\vare  \tizeta_\vare J(\tiX_\vare )\tiG_\vare\nonumber\\
	=&\nabla  \tiw \tizeta [(J(\tiX)-J(\tiX_\vare ) - J+J_\vare)][(\tiG-\tiG_0)+\tiG_0]\nonumber\\
	&+\nabla  \tiw \tizeta(J-J_\vare)[(\tiG-\tiG_0)+\tiG_0]\nonumber\\
	&+\nabla  \tiw (\tizeta-\tizeta_\vare)J(\tiX_\vare )[(\tiG-\tiG_0)+\tiG_0]\nonumber\\
	&+\nabla  (\tiw-\tiw_\vare) \tizeta_\vare J(\tiX_\vare )[(\tiG-\tiG_0)+\tiG_0]\nonumber\\
	&+\nabla \tiw_\vare  \tizeta_\vare J(\tiX_\vare )(\tiG-\tiG_\vare)\nonumber\\
	=&:\sum_{i=1}^{2}I_i+\sum_{i=3}^{4}I_i+\sum_{i=5}^{6}I_i+\sum_{i=7}^{8}I_i+I_9.\nonumber
\end{align*}
$ I_{10},\cdots, I_{18} $ are defined  by
\begin{align*}
	&t\nabla  \hat{\phi} \tizeta J(\tiX )\tiG-t\nabla \hat{\phi}_\vare \tizeta_\vare J(\tiX_\vare )\tiG_\vare\nonumber\\
	=&t\nabla  \hat{\phi} \tizeta [(J(\tiX)-J(\tiX_\vare ) - J+J_\vare)+(J-J_\vare)+J(\tiX_\vare )]\tiG-t\nabla \hat{\phi}_\vare  \tizeta_\vare J(\tiX_\vare )\tiG_\vare\nonumber\\
	=&t\nabla  \hat{\phi} \tizeta [(J(\tiX)-J(\tiX_\vare ) - J+J_\vare)+(J-J_\vare)]\tiG+t\nabla \hat{\phi} [(\tizeta-\tizeta_\vare)+\tizeta_\vare] J(\tiX_\vare )\tiG\\
	&-t\nabla \hat{\phi}_\vare  \tizeta_\vare J(\tiX_\vare )\tiG_\vare\nonumber\\
	=&t\nabla  \hat{\phi} \tizeta [(J(\tiX)-J(\tiX_\vare ) - J+J_\vare)]\tiG+t\nabla  \hat{\phi} \tizeta(J-J_\vare)\tiG+t\nabla  \hat{\phi} (\tizeta-\tizeta_\vare)J(\tiX_\vare )\tiG\nonumber\\
	&+t\nabla  [(\hat{\phi}-\hat{\phi}_\vare)+\hat{\phi}_\vare] \tizeta_\vare J(\tiX_\vare )\tiG-t\nabla \hat{\phi}_\vare  \tizeta_\vare J(\tiX_\vare )\tiG_\vare\nonumber\\
	=&t\nabla  \hat{\phi} \tizeta [(J(\tiX)-J(\tiX_\vare ) - J+J_\vare)][(\tiG-\tiG_0)+\tiG_0]\nonumber\\
	&+t\nabla \hat{\phi} \tizeta(J-J_\vare)[(\tiG-\tiG_0)+\tiG_0]\nonumber\\
	&+t\nabla  \hat{\phi} (\tizeta-\tizeta_\vare)J(\tiX_\vare )[(\tiG-\tiG_0)+\tiG_0]\nonumber\\
	&+t\nabla  (\hat{\phi}-\hat{\phi}_\vare)\tizeta_\vare J(\tiX_\vare )[(\tiG-\tiG_0)+\tiG_0]\nonumber\\
	&+t\nabla \hat{\phi}_\vare  \tizeta_\vare J(\tiX_\vare )(\tiG-\tiG_\vare)\nonumber\\
	=&:\sum_{i=10}^{11}I_i+\sum_{i=12}^{13}I_i+\sum_{i=14}^{15}I_i+\sum_{i=16}^{17}I_i+I_{18}.\nonumber
\end{align*}	
Finally, $ I_{19},\cdots, I_{25} $ are defined by
\begin{align}
	&\nabla  \tiv_0 \tizeta J(\tiX )\tiG-\nabla \tiv_0  \tizeta_\vare J(\tiX_\vare )\tiG_\vare-\nabla \tilde{v}_0(J-J_{\vare})  \tilde{G}_0\nonumber\\
	=&\nabla  \tiv_0 \tizeta [(J(\tiX)-J(\tiX_\vare ) - J+J_\vare)+(J-J_\vare)+J(\tiX_\vare )]\tiG\nonumber\\
	&-\nabla \tiv_0 \tizeta_\vare J(\tiX_\vare )\tiG_\vare-\nabla \tilde{v}_0(J-J_{\vare})  \tilde{G}_0\nonumber\\
	=&\nabla  \tiv_0 \tizeta [(J(\tiX)-J(\tiX_\vare ) - J+J_\vare)+(J-J_\vare)]\tiG+\nabla  \tiv_0 [(\tizeta-\tizeta_\vare)+\tizeta_\vare] J(\tiX_\vare )\tiG\nonumber\\
	&-\nabla \tiv_0 \tizeta_\vare J(\tiX_\vare )\tiG_\vare-\nabla \tilde{v}_0(J-J_{\vare})  \tilde{G}_0\nonumber\\
	=&\nabla  \tiv_0 \tizeta [(J(\tiX)-J(\tiX_\vare ) - J+J_\vare)]\tiG+\nabla  \tiv_0 \tizeta(J-J_\vare)\tiG+\nabla  \tiv_0 (\tizeta-\tizeta_\vare)J(\tiX_\vare )\tiG\nonumber\\
	&+\nabla  \tiv_0 \tizeta_\vare J(\tiX_\vare )\tiG-\nabla \tiv_0  \tizeta_\vare J(\tiX_\vare )\tiG_\vare-\nabla \tilde{v}_0(J-J_{\vare})  \tilde{G}_0\nonumber\\
	=&\nabla  \tiv_0 \tizeta [(J(\tiX)-J(\tiX_\vare ) - J+J_\vare)][(\tiG-\tiG_0)+\tiG_0]\nonumber\\
	&+\nabla  \tiv_0 \tizeta(J-J_\vare)[(\tiG-\tiG_0)+\tiG_0]-\nabla \tilde{v}_0(J-J_{\vare})  \tilde{G}_0\nonumber\\
	&+\nabla \tiv_0 (\tizeta-\tizeta_\vare)J(\tiX_\vare )[(\tiG-\tiG_0)+\tiG_0]\nonumber\\
	&+\nabla \tiv_0  \tizeta_\vare J(\tiX_\vare )(\tiG-\tiG_\vare)\nonumber\\
	=&\nabla  \tiv_0 \tizeta [(J(\tiX)-J(\tiX_\vare ) - J+J_\vare)][(\tiG-\tiG_0)+\tiG_0]\nonumber\\
	&+\nabla  \tiv_0 \tizeta(J-J_\vare)(\tiG-\tiG_0)\nonumber\\
	&+\nabla \tiv_0 (\tizeta-\tizeta_\vare)J(\tiX_\vare )[(\tiG-\tiG_0)+\tiG_0]\nonumber\\
	&+\nabla \tiv_0  \tizeta_\vare J(\tiX_\vare )(\tiG-\tiG_\vare)\nonumber\\
	&+\nabla  \tiv_0 (\tizeta-\id)(J-J_\vare)\tiG_0\nonumber\\
	=&:\sum_{i=19}^{20}I_i+I_{21}+\sum_{i=22}^{23}I_i+I_{24}+I_{25}.\nonumber
\end{align}
In each product, we leave only one item as the difference of the iterative step $ n-1 $ and $ n. $ Then, we  subtract the initial values as before.

We only show the estimate of $ I_1, I_7 $ and $ I_{18}. $
For $ I_1, $ we recall Theorem \ref{thm1} and apply Lemmas \ref{G-G0}, \ref{lem5.2},  \ref{jx-jy} and \ref{zeta-}  to obtain
 \begin{align*}
 	&T^\frac 14\|I_1\|_{L^2H^s}\nonumber\\
 	=&T^\frac 14\|\nabla  \tiw \tizeta [(J(\tiX)-J(\tiX_\vare ) - J+J_\vare)](\tiG-\tiG_0)\|_{L^2H^s}\nonumber\\
 	\le&T^\frac 14\|\nabla  \tiw\|_{L^2H^s}\| \tizeta\|_{L^\infty H^s}\| J(\tiX)-J(\tiX_\vare ) - J+J_\vare\|_{L^\infty H^s}\|\tiG-\tiG_0\|_{L^\infty H^s}\nonumber\\
 	\le&C(N,N_\vare, \tiv_{0},\tiG_{0})T^\frac 12  \| J(\tiX)-J(\tiX_\vare ) - J+J_\vare\|_{L^\infty H^s} \nonumber\\
 	\le&C(N,N_\vare, \tiv_{0},\tiG_{0})T^\frac 12 \\
 	&\cdot (\| J(\tiX+\vare b)-J(\tiX_\vare)\|_{L^\infty H^s}+\|J(\tiX)-J(\tiX+\vare b)+J_\vare-J\|_{L^\infty H^s}) \nonumber\\
 	\le&C(N,N_\vare, \tiv_{0},\tiG_{0})T^\frac 12 \\
 	&\cdot (\| \tiX+\vare b-\tiX_\vare\|_{L^\infty H^s}+\|J(\tiX)-J(\tiX+\vare b)\|_{L^\infty H^s}+\|J_\vare-J\|_{L^\infty H^s}) \nonumber\\
 	\le&C(N,N_\vare, \tiv_{0},\tiG_{0})T^\frac 12  (\| \tiX-\tiX_\vare+\vare b\|_{L^\infty H^s}+C\vare) \nonumber\\
 	\le&C(N,N_\vare, \tiv_{0},\tiG_{0})T^\frac 12  (\|\tilde{X}-\tilde{X}_{\vare}+\vare b-t (J-J_{\vare} ) \tilde{v}_0\|_{L^\infty H^s}+C(\tiv_{0})\vare) \nonumber\\
 	\le&C(N,N_\vare, \tiv_{0},\tiG_{0})T^\frac 12  \|\tilde{X}-\tilde{X}_{\vare}+\vare b-t (J-J_{\vare} ) \tilde{v}_0\|_{\mathcal{A}^{s+1,\gamma+1}}+C(N,N_\vare, \tiv_{0},\tiG_{0})T^\frac 12 \vare .\nonumber
 \end{align*}
For $ I_7$,   by Theorem \ref{thm1}, Lemmas \ref{G-G0}, \ref{jx-j} and  \ref{zeta-}, we have
 \begin{align}
 	&T^\frac14\|I_7\|_{L^2H^s}\nonumber\\
 	=&T^\frac14\|\nabla  (\tiw-\tiw_\vare) \tizeta_\vare J(\tiX_\vare )  (\tiG-\tiG_0) \|_{L^2H^s}\nonumber\\
 	\le&T^\frac14\|\tiw-\tiw_\vare\|_{L^2H^{s+1}}\| \tizeta_\vare\|_{L^\infty H^s}\| J(\tiX_\vare )\|_{L^\infty H^s}\|  (\tiG-\tiG_0) \|_{L^\infty H^s}\nonumber\\
 	\le&C(N,N_\vare,\tiv_{0},\tiG_{0})T^\frac12\|\tiw-\tiw_\vare\|_{\mathcal{K}_{(0)}^{s+1}}. \nonumber
 \end{align}
 From Theorem \ref{thm1}, Lemmas \ref{lem5.2}, \ref{jx-j} and  \ref{zeta-}, we obtain for   $ I_{18} $  that
\begin{align}
 	&T^\frac 14\|I_{18}\|_{L^2 H^s}\nonumber\\
 	=&T^\frac 14\|t\nabla \hat{\phi}_\vare\|_{L^2 H^s}  \|\tizeta_\vare\|_{L^\infty H^s} \|J(\tiX_\vare )\|_{L^\infty H^s}(\|(\tiG-\tiG_\vare)-t\nabla \tilde{v}_0(J-J_{\vare}) \tilde{G}_0\|_{L^\infty H^s}\nonumber\\
 	&+\|t\nabla \tilde{v}_0(J-J_{\vare}) \tilde{G}_0\|_{L^\infty H^s})\nonumber\\
 	\le &C(N,N_\vare,\tiv_{0},\tiG_{0})T^\frac 74  (\|(\tiG-\tiG_\vare)-t\nabla \tilde{v}_0(J-J_{\vare}) \tilde{G}_0\|_{\mathcal{A}^{s,\gamma}} +\vare )\nonumber\\
 	\le &C(N,N_\vare,\tiv_{0},\tiG_{0})T^\frac 74  \|(\tiG-\tiG_\vare)-t\nabla \tilde{v}_0(J-J_{\vare}) \tilde{G}_0\|_{\mathcal{A}^{s,\gamma}}.\nonumber
\end{align}
The  estimates of   other terms are similar since $ \tiv_0,  \hat{\phi} $  and $ \hat{\phi}_\vare  $ depend only on the initial data.
 
For the estimates in $ H_{(0)}^2H^\gamma, $ we first apply Lemma \ref{lem3.3}  to obtain
\begin{align}
	&\|\tilde{G}-\tilde{G}_{\vare}-t\nabla \tilde{v}_0(J-J_{\vare})  \tilde{G}_0\|_{H_{(0)}^2H^\gamma}\nonumber\\
	= &\left\|\int_0^t \nabla(\tiw +\phi )\tizeta J(\tiX )\tiG-\nabla(\tiw_\vare +\phi_\vare )\tizeta_\vare J(\tiX_\vare )\tiG_\vare  -\nabla \tilde{v}_0(J-J_{\vare})  \tilde{G}_0d\tau \right\|_{H_{(0)}^2H^\gamma}\nonumber\\
	\leq &\|  \nabla(\tiw +\phi )\tizeta J(\tiX )\tiG-\nabla(\tiw_\vare +\phi_\vare )\tizeta_\vare J(\tiX_\vare )\tiG_\vare  -\nabla \tilde{v}_0(J-J_{\vare})  \tilde{G}_0 \|_{H_{(0)}^1H^\gamma}.\nonumber
\end{align}	
Then, we rewrite  $  \nabla(\tiw +\phi )\tizeta J(\tiX )\tiG-\nabla(\tiw_\vare +\phi_\vare )\tizeta_\vare J(\tiX_\vare )\tiG_\vare  -\nabla \tilde{v}_0(J-J_{\vare})  \tilde{G}_0 $ as follows:
\begin{align}
	&\nabla(\tiw +\phi )\tizeta J(\tiX )\tiG-\nabla(\tiw_\vare +\phi_\vare )\tizeta_\vare J(\tiX_\vare )\tiG_\vare  -\nabla \tilde{v}_0(J-J_{\vare})  \tilde{G}_0\nonumber\\
	=&(\nabla  \tiw \tizeta J(\tiX )\tiG-\nabla \tiw_\vare  \tizeta_\vare J(\tiX_\vare )\tiG_\vare)  +(t\nabla  \hat{\phi} \tizeta J(\tiX )\tiG-t\nabla \hat{\phi}_\vare \tizeta_\vare J(\tiX_\vare )\tiG_\vare) \nonumber\\
	&+(\nabla  \tiv_0 \tizeta J(\tiX )\tiG-\nabla \tiv_0  \tizeta_\vare J(\tiX_\vare )\tiG_\vare-\nabla \tilde{v}_0(J-J_{\vare})  \tilde{G}_0 ).\nonumber
\end{align}
We focus on $ (\nabla  \tiw \tizeta J(\tiX )\tiG-\nabla \tiw_\vare  \tizeta_\vare J(\tiX_\vare )\tiG_\vare) $ and the others are similar since $ \tiv_0,  \hat{\phi} $  and $ \hat{\phi}_\vare  $ depend only on the initial data. Indeed, we have
\begin{align}
	&\nabla  \tiw \tizeta J(\tiX )\tiG-\nabla \tiw_\vare  \tizeta_\vare J(\tiX_\vare )\tiG_\vare\nonumber\\
	=&\nabla  \tiw \tizeta [(J(\tiX)-J(\tiX_\vare ) - J+J_\vare)][(\tiG-\tiG_0)+\tiG_0] +\nabla  \tiw \tizeta(J-J_\vare)[(\tiG-\tiG_0)+\tiG_0]\nonumber\\
	&+\nabla  \tiw [(\tizeta-\tizeta_\vare)]J(\tiX_\vare )[(\tiG-\tiG_0)+\tiG_0] +\nabla  (\tiw-\tiw_\vare) \tizeta_\vare J(\tiX_\vare )[(\tiG-\tiG_0)+\tiG_0]\nonumber\\
	&+\nabla \tiw_\vare  \tizeta_\vare J(\tiX_\vare )(\tiG-\tiG_\vare)\nonumber\\
	=&\nabla  \tiw [(\tizeta-\id)+\id] [(J(\tiX)-J(\tiX_\vare ) - J+J_\vare)][(\tiG-\tiG_0)+\tiG_0]\nonumber\\
	&+\nabla  \tiw [(\tizeta-\id)+\id](J-J_\vare)[(\tiG-\tiG_0)+\tiG_0]\nonumber\\
	&+\nabla  \tiw (\tizeta-\tizeta_\vare)[(J(\tiX_\vare)-J_\vare)+J_\vare] [(\tiG-\tiG_0)+\tiG_0]\nonumber\\
	&+\nabla  (\tiw-\tiw_\vare) [(\tizeta_\vare-\id)+\id] [(J(\tiX_\vare )-J_\vare)+J_\vare][(\tiG-\tiG_0)+\tiG_0]\nonumber\\
	&+\nabla \tiw_\vare  [(\tizeta_\vare-\id)+\id] [(J(\tiX_\vare )-J_\vare)+J_\vare](\tiG-\tiG_\vare)\nonumber\\
	=&:\sum_{i=1}^{4}I\!\!I_i+\sum_{i=5}^{8}I\!\!I_i+\sum_{i=9}^{12}I\!\!I_i+\sum_{i=13}^{20}I\!\!I_i+\sum_{i=21}^{24}I\!\!I_i.\nonumber
\end{align}
We show the estimate of $ I\!\!I_1 $ and $ I\!\!I_{24}. $ For $ I\!\!I_1, $ we apply Lemmas \ref{x-omega}, \ref{G-G0},  \ref{lem5.2},  \ref{jx-jy},  \ref{zeta-},  \ref{lem3.7}, and Theorem \ref{thm1} to get
\begin{align*}
	&\|I\!\!I_1\|_{H_{(0)}^1H^\gamma}\nonumber\\
	=&\|\nabla  \tiw\|_{H_{(0)}^1H^\gamma}\| \tizeta-\id \|_{H_{(0)}^1H^\gamma}\| J(\tiX)-J(\tiX_\vare ) - J+J_\vare\|_{H_{(0)}^1H^\gamma}\|\tiG-\tiG_0\|_{H_{(0)}^1H^\gamma}
	\nonumber\\
	\le &C(N,N_\vare,\tiv_{0},\tiG_{0})T^{2\delta} \| J(\tiX)-J(\tiX_\vare ) - J+J_\vare\|_{H_{(0)}^1H^\gamma}
	\nonumber\\
	\le&C(N,N_\vare, \tiv_{0},\tiG_{0})T^{2\delta} \\
	&\cdot (\| J(\tiX+\vare b)-J(\tiX_\vare)\|_{H_{(0)}^1H^\gamma}+\|J(\tiX)-J(\tiX+\vare b)+J_\vare-J\|_{H_{(0)}^1H^\gamma}) \nonumber\\
	\le&C(N,N_\vare, \tiv_{0},\tiG_{0})T^{2\delta} \\
	&\cdot (\| \tiX+\vare b-\tiX_\vare\|_{H_{(0)}^1H^\gamma}+\|J(\tiX)-J(\tiX+\vare b)\|_{H_{(0)}^1H^\gamma}+\|J_\vare-J\|_{H_{(0)}^1H^\gamma}) \nonumber\\
	\le&C(N,N_\vare, \tiv_{0},\tiG_{0})T^{2\delta}  (\| \tiX+\vare b-\tiX_\vare-t(J-J_\vare)\tiv_{0}+t(J-J_\vare)\tiv_{0}\|_{H_{(0)}^1H^\gamma})\\
	&+C(N,N_\vare, \tiv_{0},\tiG_{0})T^{2\delta}\vare \nonumber\\
	\le&C(N,N_\vare, \tiv_{0},\tiG_{0})T^{2\delta}  (\| \tiX+\vare b-\tiX_\vare-t(J-J_\vare)\tiv_{0}\|_{H_{(0)}^1H^\gamma}+\|t(J-J_\vare)\tiv_{0}\|_{H_{(0)}^1H^\gamma})\nonumber\\
	&+C(N,N_\vare, \tiv_{0},\tiG_{0})T^{2\delta}\vare \nonumber\\
	\le&C(N,N_\vare, \tiv_{0},\tiG_{0})T^{2\delta}  \| \tiX+\vare b-\tiX_\vare-t(J-J_\vare)\tiv_{0}\|_{H_{(0)}^1H^\gamma}+C(N,N_\vare, \tiv_{0},\tiG_{0})T^{2\delta}\vare \nonumber\\
	\le&C(N,N_\vare, \tiv_{0},\tiG_{0})T^{2\delta}  \| \tiX+\vare b-\tiX_\vare-t(J-J_\vare)\tiv_{0}\|_{\mathcal{A}^{s+1,\gamma+1}}\\
	&+C(N,N_\vare, \tiv_{0},\tiG_{0})T^{2\delta}\vare, \nonumber
\end{align*}
where $ \delta>0 $ is sufficiently small. For $ I\!\!I_{24}, $ we use  Theorem \ref{thm1}, Lemmas  \ref{lem5.2}, \ref{jx-j} and  \ref{lem3.7} to have
\begin{align*}
	&\|I\!\!I_{24}\|_{H_{(0)}^1H^\gamma}\nonumber\\
	=&\|\nabla \tiw_\vare \id  J_\vare(\tiG-\tiG_\vare)\|_{H_{(0)}^1H^\gamma}\nonumber\\
	\le &\|\nabla \tiw_\vare\|_{H_{(0)}^1H^\gamma}\|  J_\vare\|_{H_{(0)}^1H^\gamma}\|\tiG-\tiG_\vare\|_{H_{(0)}^1H^\gamma}\nonumber\\
	\le &C(N,N_\vare,\tiv_{0},\tiG_{0})\|\tiG-\tiG_\vare-t\nabla\tiv_{0}(J-J_\vare)\tiG_{0}\|_{H_{(0)}^1H^\gamma}\\
	&+C(N,N_\vare,\tiv_{0},\tiG_{0})\|t\nabla\tiv_{0}(J-J_\vare)\tiG_{0}\|_{H_{(0)}^1H^\gamma}\nonumber\\
	\le &C(N,N_\vare,\tiv_{0},\tiG_{0})\left\|\int_0^t \partial_\tau(\tiG-\tiG_\vare-\tau\nabla\tiv_{0}(J-J_\vare)\tiG_{0})d\tau \right\|_{H_{(0)}^{1+\eta-\delta}H^\gamma}\\
	&+C(N,N_\vare,\tiv_{0},\tiG_{0})\vare\nonumber\\
	\le &C(N,N_\vare,\tiv_{0},\tiG_{0})T^\delta\| \tiG-\tiG_\vare-t\nabla\tiv_{0}(J-J_\vare)\tiG_{0} \|_{\mathcal{A}^{s,\gamma} }+C(N,N_\vare,\tiv_{0},\tiG_{0})\vare.\nonumber
\end{align*}
Similarly, we have the desired estimates of the other terms. Therefore, this proposition holds if we put all the estimates together.
\end{proof}

For the velocity and the pressure, the following result holds:
\begin{proposition}
	For $2<s<\frac{5}{2}, 1<\gamma<s-1$ and a suitable $\delta_3>0$, we have 
		\begin{align} 
		&\| \tilde{w}-   \tilde{w}_{\vare}\|_{\mathcal{K}_{(0)}^{s+1}}+\| \tilde{q}_w-\tilde{q}_{w, \vare}\|_{\mathcal{K}_{p r}^s(0)}\nonumber\\
		\leq&  C \vare+ C T^{\delta_3}(\|\tilde{w}-\tilde{w}_{\vare}\|_{\mathcal{K}_{(0)}^{s+1}}+\| \tilde{q}_w-\tilde{q}_{w, \vare}\|_{\mathcal{K}_{p r}^s(0)}\nonumber\\
		&+\|\tilde{X}-\tilde{X}_{\vare}+\vare b-t(J-J_{\vare}) \tilde{v}_0\|_{\mathcal{A}^{s+1, \gamma+1}}  \nonumber\\
		&+\|\tilde{G}-\tilde{G}_{\vare}-t\nabla \tilde{v}_0(J-J_{\vare})   \tilde{G}_0\|_{\mathcal{A}^{s, \gamma}})\label{Stability w,q}
		\end{align}
where the constant $C$ depends only on the initial data.
\end{proposition}
\begin{proof}

Since 
\begin{align*}
	(\tiw-\tiw_\vare,\tiq_w-\tiq_{w,\vare})=L^{-1}(\tiF_\vare,\tiK_\vare,\tiH_\vare,0
	),
\end{align*}
we obtain from Lemma \ref{lem4.1} that
\begin{align*}
	\| \tilde{w}-   \tilde{w}_{\vare}\|_{\mathcal{K}_{(0)}^{s+1}}+\| \tilde{q}_w-\tilde{q}_{w, \vare}\|_{\mathcal{K}_{p r}^s(0)}\le C(\|\tiF_\vare\|_{\mathcal{K}_{(0)}^{s-1}}+\|\tiK_\vare\|_{\bar{\mathcal{K}}_{(0)}^{s}}+\|\tiH_\vare\|_{\mathcal{K}_{(0)}^{s-\frac 12}}).
\end{align*}
Therefore, it is sufficient to estimate $ \|\tiF_\vare\|_{\mathcal{K}_{(0)}^{s-1}},\|\tiK_\vare\|_{\bar{\mathcal{K}}_{(0)}^{s}} $ and  $\|\tiH_\vare\|_{\mathcal{K}_{(0)}^{s-\frac 12}}. $

\textbf{Estimate for $ \tiF_\vare $.}
We recall  that 
\begin{align*}
		\tiF_\vare= \tilde{f}-\tilde{f}_\vare+ \tilde{f}_\phi^L- \tilde{f}_{\phi,\vare}^L+(Q^2-Q_\vare^2) \Delta \tiw_\vare
	+ (J_\vare^\top -J^\top) \nabla \tiq_{w,\vare}
\end{align*}
and analyze  $ (Q^2-Q_\vare^2) \Delta \tiw_\vare
+ (J_\vare^\top -J^\top) \nabla \tiq_{w,\vare} $ in the above expression  first. For the $L^2 H^{s-1}$-norm, we apply Lemma \ref{lem5.2} to obtain
\begin{align}
		&  \|  (Q^2-Q_{\vare}^2 ) \Delta \tilde{w}_{\vare} \|_{L^2 H^{s-1}} \leq \|Q^2-Q_{\vare}^2 \|_{L^{\infty} H^{s-1}} \|\tilde{w}_{\vare} \|_{L^2 H^{s+1}} \leq C \vare,\nonumber \\
		&  \| (J^\top- J_{\vare}^\top  ) \nabla \tilde{q}_{w, \vare} \|_{L^2 H^{s-1}} \leq \| J^\top- J_{\vare}^\top  \|_{L^{\infty} H^{s-1}} \|\tilde{q}_{w, \vare} \|_{L^2 H^s} \leq C \vare.\nonumber
\end{align}
For the $H_{(0)}^{\frac{s-1}{2}} L^2$-norm, we use Lemmas \ref{lem3.2}, \ref{lem3.4} and Theorem \ref{thm1} to have
\begin{align*}
	 \| (Q^2-Q_{\vare}^2 ) \Delta \tilde{w}_{\vare} \|_{H_{(0)}^{\frac{s-1}{2}} L^2} \leq& \|Q^2-Q_{\vare}^2 \|_{H^{1+\eta}} \|\Delta \tilde{w}_{\vare} \|_{H_{(0)}^{\frac{s-1}{2}} L^2} \\
	 \leq& C \vare \|\tilde{w}_{\vare} \|_{H_{(0)}^{\frac{s-1}{2}} H^2} \leq C   \vare,\nonumber\\
	 \| (J^\top-J_{\vare}^\top  ) \nabla \tiq_{w,\vare}   \|_{H_{(0)}^{\frac{s-1}{2}} L^2} \leq& \|J^\top-J_{\vare}^\top  \|_{H^{1+\eta}} \|\nabla \tiq_{w,\vare}   \|_{H_{(0)}^{\frac{s-1}{2}} L^2} \leq C \vare.\nonumber 
\end{align*}
Next, we pay attention to the estimate of $  \tilde{f}-\tilde{f}_\vare+ \tilde{f}_\phi^L- \tilde{f}_{\phi,\vare}^L.   $ First, we rewrite $ \tilde{f}_\phi^L- \tilde{f}_{\phi,\vare}^L $ as follows:
\begin{align}
	\tilde{f}_\phi^L- \tilde{f}_{\phi,\vare}^L=&\partial_t \phi_\vare-\partial_t\phi +(Q^2 \Delta \phi-Q_\vare^2 \Delta \phi_\vare)+ (J_\vare^\top  \nabla \tilde{q}_{\phi,\vare}- J^\top \nabla \tilde{q}_{\phi})\nonumber\\
	=&Q_\vare^2\laplace\tiv_{0}+\nabla\tiG_{0}J_\vare\tiG_{0}- Q^2\laplace\tiv_{0}-\nabla\tiG_{0}J\tiG_{0}  +(Q^2 \Delta \phi-Q_\vare^2 \Delta \phi_\vare)\nonumber \\
	=& \nabla\tiG_{0}J_\vare \tiG_{0} -\nabla\tiG_{0}J\tiG_{0}  +t(Q^2 \Delta \hat{\phi}-Q_\vare^2 \Delta \hat{\phi}_\vare)\nonumber \\
	=& \nabla\tiG_{0}(J_\vare -J)\tiG_{0}  +t(Q^2-Q_\vare^2) \Delta \hat{\phi}+t Q_\vare^2 (\Delta\hat{\phi}-\Delta \hat{\phi}_\vare). \nonumber
\end{align}
Then, we split $  \tilde{f}-\tilde{f}_\vare+ \tilde{f}_\phi^L- \tilde{f}_{\phi,\vare}^L   $ as 
\begin{align}
	&\tilde{f}-\tilde{f}_\vare+ \tilde{f}_\phi^L- \tilde{f}_{\phi,\vare}^L\nonumber\\
	=&\tilde{f}_w-\tilde{f}_{w,\vare}+\tilde{f}_\phi-\tilde{f}_{\phi,\vare}+\tilde{f}_q-\tilde{f}_{q,\vare}+\tilde{f}_G-\tilde{f}_{G,\vare}+ \tilde{f}_\phi^L- \tilde{f}_{\phi,\vare}^L\nonumber\\
	=&(\tilde{f}_w-\tilde{f}_{w,\vare}+\tilde{f}_\phi-\tilde{f}_{\phi,\vare}+\tilde{f}_q-\tilde{f}_{q,\vare})+(\tilde{f}_G-\tilde{f}_{G,\vare}+\nabla\tiG_{0}(J_\vare -J)\tiG_{0})\nonumber\\
	&  +t(Q^2-Q_\vare^2) \Delta \hat{\phi}+t Q_\vare^2 (\Delta\hat{\phi}-\Delta \hat{\phi}_\vare)\nonumber\\
	=&:\sum_{i=1}^{4}I_i.\nonumber
\end{align}
We  point out that $ I_1 $ has already been estimated in \cite[Lemma 6.2]{castro2019splash}. We observe that in $ I_3 $ and $ I_4, $ we have all the terms depending on the  initial data and we have $ t $ in front of these terms. Thus, for $ T>0 $ small enough, we apply   Lemma \ref{lem5.2} to obtain $ \|I_3\|_{\mathcal{K}_{(0)}^{s-1}}+\|I_4\|_{\mathcal{K}_{(0)}^{s-1}}\le C(N,N_\vare,\tiv_{0},\tiG_{0})\vare.  $ Now we have to estimate $ I_2 $ in $ L^2H^{s-1} $ and $ H_{(0)}^{\frac{s-1}{2}}L^2. $ From \eqref{equ4.14}, we have
\begin{align}
	I_2= \nabla\tiG \tizeta J(\tiX)\tiG-\nabla\tiG_\vare \tizeta_\vare J(\tiX_\vare)\tiG_\vare+ \nabla\tiG_{0}(J_\vare -J)\tiG_{0}.	\nonumber
\end{align}
For the estimate in  $ L^2H^{s-1}, $ we use the following splitting:
\begin{align*}
	&\nabla\tiG \tizeta J(\tiX)\tiG-\nabla\tiG_\vare \tizeta_\vare J(\tiX_\vare)\tiG_\vare+ \nabla\tiG_{0}(J_\vare -J)\tiG_{0}\nonumber\\
	=&\nabla\tiG \tizeta[(J(\tiX)-J(\tiX_\vare ) - J+J_\vare)+(J-J_\vare)+J(\tiX_\vare )]\tiG\\
	&-\nabla\tiG_\vare \tizeta_\vare J(\tiX_\vare)\tiG_\vare+ \nabla\tiG_{0}(J_\vare -J)\tiG_{0}\nonumber\\
	=&\nabla\tiG \tizeta(J(\tiX)-J(\tiX_\vare ) - J+J_\vare)\tiG+\nabla\tiG \tizeta(J-J_\vare)\tiG  \nonumber\\
	 &+\nabla\tiG \tizeta J(\tiX_\vare )\tiG-\nabla\tiG_\vare \tizeta_\vare J(\tiX_\vare)\tiG_\vare+ \nabla\tiG_{0}(J_\vare -J)\tiG_{0}\nonumber\\
	 =&\nabla\tiG \tizeta(J(\tiX)-J(\tiX_\vare ) - J+J_\vare)\tiG+\nabla\tiG \tizeta(J-J_\vare) \tiG+\nabla\tiG \tizeta J(\tiX_\vare )(\tiG-\tiG_\vare)  \nonumber\\
	 &+\nabla\tiG \tizeta J(\tiX_\vare )\tiG_\vare-\nabla\tiG_\vare \tizeta_\vare J(\tiX_\vare)\tiG_\vare+ \nabla\tiG_{0}(J_\vare -J)\tiG_{0}\nonumber\\
	 =&\nabla\tiG \tizeta(J(\tiX)-J(\tiX_\vare ) - J+J_\vare)\tiG+\nabla\tiG \tizeta (J-J_\vare) \tiG- \nabla\tiG_{0}(J -J_\vare)\tiG_{0} \nonumber\\
	 &+\nabla\tiG \tizeta J(\tiX_\vare )(\tiG-\tiG_\vare)+\nabla\tiG (\tizeta-\tizeta_\vare) J(\tiX_\vare )\tiG_\vare +\nabla(\tiG-\tiG_\vare)\tizeta_\vare J(\tiX_\vare )\tiG_\vare\nonumber\\
	 =&\nabla[(\tiG -\tiG_{0})+\tiG_0] \tizeta(J(\tiX)-J(\tiX_\vare ) - J+J_\vare)[(\tiG -\tiG_{0})+\tiG_0]\nonumber\\
	 &+\nabla[(\tiG -\tiG_{0})+\tiG_0] \tizeta (J-J_\vare) [(\tiG -\tiG_{0})+\tiG_0]-\nabla\tiG_{0} \tizeta (J-J_\vare) \tiG_0\nonumber\\
	 &+\nabla\tiG_{0} (\tizeta-\id) (J-J_\vare) \tiG_0  \nonumber\\
	 &+\nabla[(\tiG -\tiG_{0})+\tiG_0] \tizeta J(\tiX_\vare )(\tiG-\tiG_\vare)\nonumber\\
	 &+\nabla[(\tiG -\tiG_{0})+\tiG_0] (\tizeta-\tizeta_\vare) J(\tiX_\vare )[(\tiG_\vare-\tiG_{0})+\tiG_0]\nonumber\\
	 & +\nabla(\tiG-\tiG_\vare)\tizeta_\vare J(\tiX_\vare )[(\tiG_\vare-\tiG_{0})+\tiG_0]\nonumber\\
	 =&:\sum_{i=1}^{4}I_{2,i}+\sum_{i=5}^{7}I_{2,i}+I_{2,8
	 }+\sum_{i=9}^{10}I_{2,i}+\sum_{i=11}^{14}I_{2,i}+\sum_{i=15}^{16}I_{2,i}.\nonumber
\end{align*}
In the same way,  we leave only one difference of the adjacent iterative steps  and  subtract the initial values in each product.

For $ I_{2,1}, $ we apply Lemmas \ref{G-G0}, \ref{lem5.2},  \ref{jx-jy} and  \ref{zeta-} to obtain
\begin{align*}
	&\|I_{2,1}\|_{L^2H^{s-1}}\nonumber\\
	\le &\|\nabla(\tiG -\tiG_{0})\|_{L^\infty H^{s-1}}\| \tizeta\|_{L^\infty H^{s-1}}\|J(\tiX)-J(\tiX_\vare ) - J+J_\vare\|_{L^\infty H^{s-1}}\|\tiG -\tiG_{0}\|_{L^2H^{s-1}}\nonumber\\
	\le &T^\frac 12\| \tiG -\tiG_{0} \|_{L^\infty H^{s}}\| \tizeta\|_{L^\infty H^{s-1}}\|J(\tiX)-J(\tiX_\vare ) - J+J_\vare\|_{L^\infty H^{s-1}}\|\tiG -\tiG_{0}\|_{L^\infty H^{s-1}}\nonumber\\
	\le &C(N,N_\vare,\tiv_{0},\tiG_{0})T \|J(\tiX)-J(\tiX_\vare ) - J+J_\vare\|_{L^\infty H^{s-1}}\nonumber\\
	\le &C(N,N_\vare,\tiv_{0},\tiG_{0})T [\| J(\tiX+\vare b)-J(\tiX_\vare )  \|_{L^\infty H^{s-1}}\\
	&+\| J(\tiX)-J(\tiX_\vare+\vare b ) - J+J_\vare \|_{L^\infty H^{s-1}}]\nonumber\\
	\le &C(N,N_\vare,\tiv_{0},\tiG_{0})T [\| \tiX+\vare b -  \tiX_\vare \|_{L^\infty H^{s-1}}+\| J(\tiX) - J  \|_{L^\infty H^{s-1}}\\
	&+\| J(\tiX_\vare+\vare b ) -  J_\vare \|_{L^\infty H^{s-1}}]\nonumber\\
	\le &C(N,N_\vare,\tiv_{0},\tiG_{0})T \| \tiX+\vare b -  \tiX_\vare \|_{L^\infty H^{s-1}}+C(N,N_\vare,\tiv_{0},\tiG_{0})T\vare\nonumber\\	
	\le &C(N,N_\vare,\tiv_{0},\tiG_{0})T [\| \tiX+\vare b -  \tiX_\vare-t(J-J_\vare)\tiv_{0} \|_{L^\infty H^{s-1}}+\| t(J-J_\vare)\tiv_{0} \|_{L^\infty H^{s-1}}]\nonumber\\
	&+C(N,N_\vare,\tiv_{0},\tiG_{0})T\vare\nonumber\\
	\le &C(N,N_\vare,\tiv_{0},\tiG_{0})T \| \tiX+\vare b -  \tiX_\vare-t(J-J_\vare)\tiv_{0} \|_{L^\infty H^{s-1}}+C(N,N_\vare,\tiv_{0},\tiG_{0})T\vare\nonumber\\
	\le &C(N,N_\vare,\tiv_{0},\tiG_{0})T^\frac 54 \| \tiX+\vare b -  \tiX_\vare-t(J-J_\vare)\tiv_{0} \|_{\mathcal{A}^{s+1,\gamma+1}}+C(N,N_\vare,\tiv_{0},\tiG_{0})T\vare.\nonumber
\end{align*}
For $ I_{2,15}, $ we use Lemmas  \ref{G-G0}, \ref{lem5.2},  \ref{jx-j} and \ref{zeta-} to get
\begin{align*}
	&\|I_{2,15}\|_{L^2H^{s-1}}\nonumber\\
	\le &\|\nabla(\tiG-\tiG_\vare)\|_{L^\infty H^{s-1}}\|\tizeta_\vare\|_{L^\infty H^{s-1}}\| J(\tiX_\vare )\|_{L^\infty H^{s-1}}\|\tiG_\vare-\tiG_{0}\|_{L^2 H^{s-1}}\nonumber\\
	\le &T^\frac 14\| \tiG-\tiG_\vare \|_{L^\infty H^{s}}\|\tizeta_\vare\|_{L^\infty H^{s-1}}\| J(\tiX_\vare )\|_{L^\infty H^{s-1}}\|\tiG_\vare-\tiG_{0}\|_{L^\infty H^{s-1}}\nonumber\\
	\le &C(N,N_\vare,\tiv_{0},\tiG_{0})T^\frac 12[\| \tiG-\tiG_\vare-t\nabla \tiv_{0}(J-J_\vare)\tiG_0 \|_{L^\infty H^{s}}+\| t\nabla \tiv_{0}(J-J_\vare)\tiG_0 \|_{L^\infty H^{s}}] \nonumber\\
	\le &C(N,N_\vare,\tiv_{0},\tiG_{0})T^\frac 34 \| \tiG-\tiG_\vare-t\nabla \tiv_{0}(J-J_\vare)\tiG_0 \|_{\mathcal{A}^{s,\gamma}}+C(N,N_\vare,\tiv_{0},\tiG_{0})T^\frac 12\vare.  \nonumber
\end{align*}
The other terms are estimated similarly.

For the estimate in $ H_{(0)}^{\frac{s-1}{2}}L^2, $ we rewrite $ I_2 $ as follows:
\begin{align}
	&\nabla\tiG \tizeta J(\tiX)\tiG-\nabla\tiG_\vare \tizeta_\vare J(\tiX_\vare)\tiG_\vare+ \nabla\tiG_{0}(J_\vare -J)\tiG_{0}\nonumber\\
	=&\nabla\tiG \tizeta J(\tiX)(\tiG-\tiG_\vare) +\nabla\tiG (\tizeta-\tizeta_\vare) J(\tiX) \tiG_\vare  +\nabla(\tiG-\tiG_\vare) \tizeta_\vare  J(\tiX) \tiG_\vare\nonumber\\
	&+[\nabla \tiG_\vare  \tizeta_\vare  (J(\tiX)- J(\tiX_\vare))\tiG_\vare- \nabla\tiG_{0}(J -J_\vare)\tiG_{0}]\nonumber\\
	=&:\sum_{i=1}^{4}I'_{2,i}.\nonumber
\end{align}
These terms need to be estimated  separately, and we show the main idea. We split $ I'_{2,1} $ as follows
\begin{align}
	I'_{2,1}
	=&\nabla\tiG \tizeta J(\tiX)(\tiG-\tiG_\vare)\nonumber\\
	=&\nabla[(\tiG-\tiG_0)+\tiG_0] [(\tizeta-\id)+\id] [(J(\tiX)-J)+J](\tiG-\tiG_\vare)\nonumber\\
	=&:\sum_{i=1}^{8}I'_{2,1,i}\nonumber
\end{align}
For $ I'_{2,1,1},  $ we use
Lemmas  \ref{x-omega},  \ref{G-G0},  \ref{lem5.2}, \ref{jx-j}, \ref{zeta-}, \ref{lem3.4},  \ref{lem3.5} and \ref{lem3.7} to obtain
\begin{align*}
	&\|I'_{2,1,1}\|_{H_{(0)}^{\frac{s-1}{2}}L^2}\nonumber\\
	=&\|\nabla (\tiG-\tiG_0)   (\tizeta-\id)   (J(\tiX)-J) (\tiG-\tiG_\vare)\|_{H_{(0)}^{\frac{s-1}{2}}L^2}\nonumber\\
	\le &\|\nabla (\tiG-\tiG_0)   (\tizeta-\id) \|_{H_{(0)}^{\frac{s-1}{2}}L^2}\|  (J(\tiX)-J) (\tiG-\tiG_\vare)\|_{H_{(0)}^{\frac{s-1}{2}}H^{1+\eta}}\nonumber\\
	\le &\|\nabla (\tiG-\tiG_0)\|_{H_{(0)}^{\frac{s-1}{2}}H^{1-\mu}}\|  \tizeta-\id \|_{H_{(0)}^{\frac{s-1}{2}}H^{1+\mu}}\\
	&\cdot\|  J(\tiX)-J\|_{H_{(0)}^{\frac{s-1}{2}}H^{1+\eta}}\| \tiG-\tiG_\vare\|_{H_{(0)}^{\frac{s-1}{2}}H^{1+\eta}}\nonumber\\
	\le &\|  \tiG-\tiG_0\|_{H_{(0)}^{\frac{s-1}{2}}H^{1+(1-\mu)}}\|   \tizeta-\id \|_{H_{(0)}^{\frac{s-1}{2}}H^{1+\mu}}\\
	&\cdot\|  J(\tiX)-J\|_{H_{(0)}^{\frac{s-1}{2}}H^{1+\eta}}\| \tiG-\tiG_\vare\|_{H_{(0)}^{\frac{s-1}{2}}H^{1+\eta}}\nonumber\\
	\le &C(N,N_\vare,\tiv_0,\tiG_0)T^{2\delta}  \bigg[\| \tiG-\tiG_\vare-t\nabla \tiv_{0}(J-J_\vare)\tiG_0\|_{H_{(0)}^{\frac{s-1}{2}}H^{1+\eta}}\\
	&\qquad\qquad\qquad\qquad\qquad+\| t\nabla \tiv_{0}(J-J_\vare)\tiG_0\|_{H_{(0)}^{\frac{s-1}{2}}H^{1+\eta}}\bigg]\nonumber\\
	\le &C(N,N_\vare,\tiv_0,\tiG_0)T^{2\delta}  \| \tiG-\tiG_\vare-t\nabla \tiv_{0}(J-J_\vare)\tiG_0\|_{\mathcal{A}^{s,\gamma}}+C(N,N_\vare,\tiv_0,\tiG_0)T^{2\delta}\vare.\nonumber
\end{align*}
The other terms are estimated similarly. 

For $ i=2,3 $ and $ 4,$  we split $ I'_{2,i} $ as follows 
\begin{align*}
	I'_{2,2}=&\nabla\tiG (\tizeta-\tizeta_\vare) J(\tiX) \tiG_\vare\nonumber\\
	=&[(\nabla\tiG-\nabla\tiG_0)+\nabla\tiG_0] (\tizeta-\tizeta_\vare) [(J(\tiX)-J)+J] [(\tiG_\vare-\tiG_0)+\tiG_0]\nonumber\\
	=&:\sum_{i=1}^{8}I'_{2,2,i},\nonumber\\
	I'_{2,3}=&\nabla(\tiG-\tiG_\vare) \tizeta_\vare  J(\tiX) \tiG_\vare\nonumber\\
	=&\nabla(\tiG-\tiG_\vare) [(\tizeta_\vare-\id)+\id]  [(J(\tiX)-J)+J] [(\tiG_\vare-\tiG_0)+\tiG_0]\nonumber\\
	=&:\sum_{i=1}^{8}I'_{2,3,i},\nonumber\\
	I'_{2,4}=&\nabla \tiG_\vare  \tizeta_\vare  (J(\tiX)- J(\tiX_\vare))\tiG_\vare- \nabla\tiG_{0}(J -J_\vare)\tiG_{0}\nonumber\\
	=&\nabla [(\tiG_\vare-\tiG_0)+\tiG_0]  [(\tizeta_\vare-\id)+\id]  [(J(\tiX)- J(\tiX_\vare)-J+J_\vare)+(J-J_\vare)]\\
	&\cdot[(\tiG_\vare-\tiG_0)+\tiG_0]\nonumber\\
	&- \nabla\tiG_{0}(J -J_\vare)\tiG_{0}\nonumber\\
	=&:\sum_{i=1}^{15}I'_{2,4,i}.\nonumber
\end{align*}
Finally, we conclude that
\begin{align*}
	\sum_{i,j}\|I'_{2,i,j}\|_{H_{(0)}^{\frac{s-1}{2}}L^2}\le &C(N,N_\vare,\tiv_0,\tiG_0)T^{\delta_1}  \| \tiX-\tiX_\vare-t\nabla \tiv_{0}(J-J_\vare)\tiG_0\|_{\mathcal{A}^{s+1,\gamma+1}}\nonumber\\
	& +C(N,N_\vare,\tiv_0,\tiG_0)T^{\delta_2}\| \tiG-\tiG_\vare-t\nabla \tiv_{0}(J-J_\vare)\tiG_0\|_{\mathcal{A}^{s,\gamma}}\\
	&+C(N,N_\vare,\tiv_0,\tiG_0)\vare\nonumber
\end{align*}
for some $ \delta_1, \delta_2>0. $

\textbf{Estimate for $ \tiK_\vare $.}
It can be found in \cite[Lemma 6.2]{castro2019splash}, and we omit the details.

\textbf{Estimate for $ \tiH_\vare $.}
We rewrite $ \tiH_\vare $ as 
\begin{align}
	\tiH_\vare=&\tilde{h}-\tilde{h}_\vare+\tilde{h}_\phi^L-\tilde{h}_{\phi,\vare}^L
	+\tiq_{w,\vare}(J^{-1}-J_\vare^{-1})\tilde{n}_0+(\nabla \tiw_\vare J_\vare)^\top J_\vare^{-1}\tilde{n}_0-(\nabla \tiw_\vare J)^\top J^{-1}\tilde{n}_0\nonumber\\
	=&\tilde{h}-\tilde{h}_\vare+\tilde{h}_\phi^L-\tilde{h}_{\phi,\vare}^L
	+\bar{H}_\vare.\nonumber
\end{align}
We start by estimating  $ \bar{H}_\vare $ by rewriting it as
\begin{align}
	\bar{H}_\vare=&\tiq_{w,\vare}(J^{-1}-J_\vare^{-1})\tilde{n}_0+(\nabla \tiw_\vare J_\vare)^\top J_\vare^{-1}\tilde{n}_0-(\nabla \tiw_\vare J)^\top J^{-1}\tilde{n}_0\nonumber\\
	=&\tiq_{w,\vare}(J^{-1}-J_\vare^{-1})\tilde{n}_0 +[(\nabla \tiw_\vare J_\vare)^\top J_\vare^{-1} -(\nabla \tiw_\vare J)^\top J_\vare^{-1}]\tilde{n}_0\nonumber\\ &+[(\nabla \tiw_\vare J)^\top J_\vare^{-1} -(\nabla \tiw_\vare J)^\top J^{-1}]\tilde{n}_0\nonumber\\		=&\tiq_{w,\vare}(J^{-1}-J_\vare^{-1})\tilde{n}_0 +(\nabla \tiw_\vare (J_\vare-J))^\top J_\vare^{-1} \tilde{n}_0+ (\nabla \tiw_\vare J)^\top (J_\vare^{-1} - J^{-1})\tilde{n}_0\nonumber\\
	=&:\sum_{i=1}^{3}\bar{I}_i.\nonumber
\end{align}
We notice that 
\begin{align*}
	\|\bar{I}_1\|_{\mathcal{K}_{(0)}^{s-\frac 12}}+\|\bar{I}_2\|_{\mathcal{K}_{(0)}^{s-\frac 12}}+\|\bar{I}_3\|_{\mathcal{K}_{(0)}^{s-\frac 12}}\le C(N,N_\vare,\tiv_0,\tiG_0)\vare
\end{align*}
by Lemma \ref{lem5.2}.

Now we estimate $ \tilde{h}-\tilde{h}_\vare+\tilde{h}_\phi^L-\tilde{h}_{\phi,\vare}^L. $ From the definition of $ \tiq_\phi $ and $ \tiq_{\phi,\vare}, $ we have
\begin{align*}
	&\tilde{h}-\tilde{h}_\vare+\tilde{h}_\phi^L-\tilde{h}_{\phi,\vare}^L\nonumber\\
	=&\tih_w+\tih_{w^\top }+\tih_\phi+\tih_{\phi^\top }+\tih_q+\tih_G-(\tih_{w,\vare}+\tih_{w^\top ,\vare}+\tih_{\phi,\vare}+\tih_{\phi^\top ,\vare}+\tih_{q,\vare}+\tih_{G,\vare})\nonumber\\
	&+\tilde{q}_{\phi } J^{-1} \tilde{n}_0-(\nabla \phi  J +(\nabla \phi  J )^\top )J^{-1} \tilde{n}_0\\
	&-(\tilde{q}_{\phi,\vare} J_\vare^{-1} \tilde{n}_0-(\nabla \phi_\vare J_\vare+(\nabla \phi_\vare J_\vare)^\top )J_\vare^{-1} \tilde{n}_0)\nonumber\\
	=&\tih_w-\tih_{w,\vare}+\tih_{w^\top }-\tih_{w^\top ,\vare}  +\tih_\phi-\tih_{\phi,\vare} +\tih_{\phi^\top }-\tih_{\phi^\top ,\vare}+\tih_q-\tih_{q,\vare}\nonumber\\
	&+\tilde{q}_{\phi } J^{-1} \tilde{n}_0-(\nabla \phi  J +(\nabla \phi  J )^\top )J^{-1} \tilde{n}_0\nonumber\\
	&-(\tilde{q}_{\phi,\vare} J_\vare^{-1} \tilde{n}_0-(\nabla \phi_\vare J_\vare+(\nabla \phi_\vare J_\vare)^\top )J_\vare^{-1} \tilde{n}_0)\nonumber\\
	&+\tih_G-\tih_{G,\vare}\nonumber\\
	=&\tih_w-\tih_{w,\vare}+\tih_{w^\top }-\tih_{w^\top ,\vare}  +\tih_\phi-\tih_{\phi,\vare} +\tih_{\phi^\top }-\tih_{\phi^\top ,\vare}+\tih_q-\tih_{q,\vare}\nonumber\\
	&+( \nabla \tiv_0 J+( \nabla \tiv_0 J)^\top +\tiG_0 \tiG_0^\top )J^{-1} \tilde{n}_0-(\nabla \phi  J +(\nabla \phi  J )^\top )J^{-1} \tilde{n}_0\nonumber\\
	&-(( \nabla \tiv_0 J_\vare+( \nabla \tiv_0 J_\vare)^\top +\tiG_0 \tiG_0^\top )J_\vare^{-1} \tilde{n}_0-(\nabla \phi_\vare J_\vare+(\nabla \phi_\vare J_\vare)^\top )J_\vare^{-1} \tilde{n}_0)\nonumber\\
	&+\tih_G-\tih_{G,\vare}\nonumber\\
	=&\tih_w-\tih_{w,\vare}+\tih_{w^\top }-\tih_{w^\top ,\vare}  +\tih_\phi-\tih_{\phi,\vare} +\tih_{\phi^\top }-\tih_{\phi^\top ,\vare}+\tih_q-\tih_{q,\vare}\nonumber\\
	&+ \tiG_0 \tiG_0^\top J^{-1} \tilde{n}_0-t(\nabla \hat{\phi}  J +(\nabla \hat{\phi}  J )^\top )J^{-1} \tilde{n}_0\nonumber\\
	&-(  \tiG_0 \tiG_0^\top  J_\vare^{-1} \tilde{n}_0-t(\nabla \hat{\phi}_\vare J_\vare+(\nabla \hat{\phi}_\vare J_\vare)^\top )J_\vare^{-1} \tilde{n}_0)\nonumber\\
	&+\tih_G-\tih_{G,\vare}\nonumber\\
	=&\tih_w-\tih_{w,\vare}+\tih_{w^\top }-\tih_{w^\top ,\vare}  +\tih_\phi-\tih_{\phi,\vare} +\tih_{\phi^\top }-\tih_{\phi^\top ,\vare}+\tih_q-\tih_{q,\vare}\nonumber\\
	& +t\nabla (\hat{\phi}_\vare-\hat{\phi})\tilde{n}_0  \nonumber\\
	&+t[ (\nabla \hat{\phi}_\vare J_\vare)^\top  J_\vare^{-1}-(\nabla \hat{\phi}  J )^\top  J^{-1}] \tilde{n}_0\nonumber\\
	&+\tiG_0 \tiG_0^\top (J^{-1}-J_\vare^{-1}) \tilde{n}_0-\tiG\tiG^\top J(\tiX)^{-1} \nabla_{\Lambda} \tiX \tilde{n}_0+\tiG_\vare\tiG_\vare^\top J(\tiX_\vare)^{-1} \nabla_{\Lambda} \tiX_\vare \tilde{n}_0\nonumber\\
	=&:\sum_{i=1}^{4}\tilde{I}_i.\nonumber
\end{align*}

For $ \tilde{I}_1,$ the estimate of $ \tih_w-\tih_{w,\vare},\tih_{w^\top }-\tih_{w^\top ,\vare},\tih_\phi-\tih_{\phi,\vare},\tih_{\phi^\top }-\tih_{\phi^\top ,\vare} $ and $\tih_q-\tih_{q,\vare}$ have already been studied in \cite[Lemma 6.2]{castro2019splash}. For $  \tilde{I}_2 $ and $ \tilde{I}_3, $ we use Lemma \ref{lem5.2} to obtain
\begin{align*}
	\|\tilde{I}_2\|_{\mathcal{K}_{(0)}^{s-\frac 12}}+	\|\tilde{I}_3\|_{\mathcal{K}_{(0)}^{s-\frac 12}}\le C(N,N_\vare,\tiv_0,\tiG_0)\vare.
\end{align*}

For $ \tilde{I}_4, $ we have the following splitting:
\begin{align}
	\tilde{I}_4
	=& \tiG_0 \tiG_0^\top (J^{-1}-J_\vare^{-1}) \tilde{n}_0-\tiG\tiG^\top J(\tiX)^{-1} \nabla_{\Lambda} \tiX \tilde{n}_0+\tiG_\vare\tiG_\vare^\top J(\tiX_\vare)^{-1} \nabla_{\Lambda} \tiX_\vare \tilde{n}_0\nonumber\\
	=&\tiG_0 \tiG_0^\top (J^{-1}-J_\vare^{-1}) \tilde{n}_0+\tiG_\vare\tiG_\vare^\top (J(\tiX_\vare)^{-1}-J(\tiX)^{-1}) \nabla_{\Lambda} \tiX\tilde{n}_0\nonumber\\
	&+\tiG_\vare\tiG_\vare^\top  J(\tiX)^{-1}\nabla_{\Lambda} \tiX \tilde{n}_0-\tiG\tiG^\top J(\tiX)^{-1} \nabla_{\Lambda} \tiX \tilde{n}_0 \nonumber\\
	&+\tiG_\vare\tiG_\vare^\top J(\tiX_\vare)^{-1} \nabla_{\Lambda} \tiX_\vare \tilde{n}_0-\tiG_\vare\tiG_\vare^\top J(\tiX_\vare)^{-1}   \nabla_{\Lambda} \tiX \tilde{n}_0\nonumber\\
	=&\tiG_0 \tiG_0^\top (J^{-1}-J_\vare^{-1}) \tilde{n}_0+\tiG_\vare\tiG_\vare^\top (J(\tiX_\vare)^{-1}-J(\tiX)^{-1}) \nabla_{\Lambda} \tiX\tilde{n}_0\nonumber\\
	& +\tiG_\vare\tiG_\vare^\top J(\tiX_\vare)^{-1} (\nabla_{\Lambda} \tiX_\vare-\nabla_{\Lambda} \tiX )\tilde{n}_0+\tiG_\vare\tiG_\vare^\top  J(\tiX)^{-1}\nabla_{\Lambda} \tiX \tilde{n}_0\nonumber\\
	&- \tiG\tiG_\vare^\top  J(\tiX)^{-1}\nabla_{\Lambda} \tiX \tilde{n}_0+ \tiG\tiG_\vare^\top  J(\tiX)^{-1}\nabla_{\Lambda} \tiX \tilde{n}_0-\tiG\tiG^\top J(\tiX)^{-1} \nabla_{\Lambda} \tiX \tilde{n}_0 \nonumber\nonumber\\
	=&\tiG_0 \tiG_0^\top (J^{-1}-J_\vare^{-1}) \tilde{n}_0 +\tiG_\vare\tiG_\vare^\top (J(\tiX_\vare)^{-1}-J(\tiX)^{-1}) \nabla_{\Lambda} \tiX\tilde{n}_0\nonumber\\
	&+\tiG_\vare\tiG_\vare^\top J(\tiX_\vare)^{-1} (\nabla_{\Lambda} \tiX_\vare-\nabla_{\Lambda} \tiX )\tilde{n}_0\nonumber\\
	&+(\tiG_\vare-\tiG)\tiG_\vare^\top  J(\tiX)^{-1}\nabla_{\Lambda} \tiX \tilde{n}_0\nonumber\\
	&+ \tiG(\tiG_\vare^\top -\tiG^\top ) J(\tiX)^{-1}\nabla_{\Lambda} \tiX \tilde{n}_0 \nonumber \\
	=&:\sum_{i=1}^{4}\tilde{I}_{4,i}\label{equ666}
\end{align}
We start with the estimates in $ L^2H^{s-\frac 12}. $ We only show the estimates of $ \tilde{I}_{4,1} $ and $ \tilde{I}_{4,3} $, and the others are similar.   We split $ \tilde{I}_{4,1} $ as follows:
\begin{align}
	&\tiG_\vare\tiG_\vare^\top (J(\tiX)^{-1}-J(\tiX_\vare)^{-1}) \nabla_{\Lambda} \tiX\tilde{n}_0-\tiG_0 \tiG_0^\top (J^{-1}-J_\vare^{-1}) \tilde{n}_0\nonumber\\
	=&\tiG_\vare\tiG_\vare^\top (J(\tiX)^{-1}-J(\tiX_\vare)^{-1}) \nabla_{\Lambda} \tiX\tilde{n}_0-\tiG_0 \tiG_0^\top (J^{-1}-J_\vare^{-1}) \nabla_{\Lambda} \tiX\tilde{n}_0\nonumber\\
	+&\tiG_0 \tiG_0^\top (J^{-1}-J_\vare^{-1}) \nabla_{\Lambda} \tiX\tilde{n}_0-\tiG_0 \tiG_0^\top (J^{-1}-J_\vare^{-1})  \tilde{n}_0\nonumber\\
	=&[(\tiG_\vare-\tiG_0)+\tiG_0][(\tiG_\vare-\tiG_0)+\tiG_0]^\top [(J(\tiX)^{-1}-J(\tiX_\vare)^{-1})-(J^{-1}-J_\vare^{-1})\nonumber\\
	&+(J^{-1}-J_\vare^{-1}) ]\nabla_{\Lambda} \tiX\tilde{n}_0-\tiG_0 \tiG_0^\top (J^{-1}-J_\vare^{-1}) \nabla_{\Lambda} \tiX\tilde{n}_0\nonumber\\
	&+\tiG_0 \tiG_0^\top (J^{-1}-J_\vare^{-1}) (\nabla_{\Lambda} \tiX-\id)\tilde{n}_0\nonumber\\
	=&:\sum_{i=1}^{7}\tilde{I}_{4,1,i}+\tilde{I}_{4,1,8}.\nonumber
\end{align}
For $ \tilde{I}_{4,1,1},  $ we apply  Theorem \ref{lem3.9}, Lemmas  \ref{x-omega}, \ref{G-G0}, \ref{lem5.2} and \ref{jx-jy} to obtain
\begin{align}
	&\|\tilde{I}_{4,1,1}\|_{L^2H^{s-\frac 12}}\nonumber\\
	=&\|(\tiG_\vare-\tiG_0)(\tiG_\vare-\tiG_0)[(J(\tiX)^{-1}-J(\tiX_\vare)^{-1})-(J^{-1}-J_\vare^{-1})]\nabla_{\Lambda} \tiX\tilde{n}_0\|_{L^2H^{s-\frac 12}}\nonumber\\
	\le &\|\tiG_\vare-\tiG_0\|_{L^2H^{s-\frac 12}}\|\tiG_\vare-\tiG_0\|_{L^\infty H^{s-\frac 12}}\|(J(\tiX)^{-1}-J(\tiX_\vare)^{-1})-(J^{-1}-J_\vare^{-1})\|_{L^\infty H^{s-\frac 12}}\nonumber\\
	&\cdot\|\nabla_{\Lambda} \tiX\tilde{n}_0\|_{L^\infty H^{s-\frac 12}}\nonumber\\
	\le &T^\frac 12\|\tiG_\vare-\tiG_0\|_{L^\infty H^{s }}\|\tiG_\vare-\tiG_0\|_{L^\infty H^{s }}\|(J(\tiX)^{-1}-J(\tiX_\vare)^{-1})-(J^{-1}-J_\vare^{-1})\|_{L^\infty H^{s }}\nonumber\\
	&\cdot\|\nabla_{\Lambda} \tiX\tilde{n}_0\|_{L^\infty H^{s }}\nonumber\\ 	
	\le &T C(N,N_\vare,\tiv_{0},\tiG_0)\|(J(\tiX)^{-1}-J(\tiX_\vare)^{-1})-(J^{-1}-J_\vare^{-1})\|_{L^\infty H^{s }}\nonumber\\ 
	\le &T C(N,N_\vare,\tiv_{0},\tiG_0)\|J(\tiX+\vare b)^{-1}-J(\tiX_\vare)^{-1}\|_{L^\infty H^{s }}\nonumber\\
	&+T C(N,N_\vare,\tiv_{0},\tiG_0)(\|J(\tiX)^{-1}-J(\tiX+\vare b)^{-1}\|_{L^\infty H^{s }}+\|J^{-1}-J_\vare^{-1}\|_{L^\infty H^{s }})\nonumber\\
	\le &T C(N,N_\vare,\tiv_{0},\tiG_0)\|  \tiX+\vare b  -  \tiX_\vare \|_{L^\infty H^{s }} +T C(N,N_\vare,\tiv_{0},\tiG_0)\vare\nonumber\\
	\le &T^\frac 54 C(N,N_\vare,\tiv_{0},\tiG_0)\|  \tiX+\vare b  -  \tiX_\vare-t(J-J_\vare)\tiv_{0} \|_{\mathcal{A}^{s+1,\gamma+1}} +T C(N,N_\vare,\tiv_{0},\tiG_0)\vare.\nonumber
\end{align}
For $ \tilde{I}_{4,3}, $ we use the following splitting:
\begin{align}
	\tilde{I}_{4,3}=&(\tiG_\vare-\tiG)\tiG_\vare^\top  J(\tiX)^{-1}\nabla_{\Lambda} \tiX \tilde{n}_0\nonumber\\
	=&(\tiG_\vare-\tiG)(\tiG_\vare-\tiG_0)^\top  J(\tiX)^{-1}\nabla_{\Lambda} \tiX \tilde{n}_0+(\tiG_\vare-\tiG)\tiG_0^\top  J(\tiX)^{-1}\nabla_{\Lambda} \tiX \tilde{n}_0\nonumber\\
	=&:\tilde{I}_{4,3,1}+\tilde{I}_{4,3,2}.\nonumber
\end{align}
For $ \tilde{I}_{4,3,1},  $ we use   Lemmas \ref{x-omega}, \ref{G-G0}, \ref{lem5.2}, \ref{jx-j} and  Theorem \ref{lem3.9} to have
\begin{align*}
	&\|\tilde{I}_{4,3,1}\|_{L^2H^{s-\frac 12}}\nonumber\\
	=&\|(\tiG_\vare-\tiG)(\tiG_\vare-\tiG_0)^\top  J(\tiX)^{-1}\nabla_{\Lambda} \tiX \tilde{n}_0\|_{L^2H^{s-\frac 12}}\nonumber\\
	\le &\|\tiG_\vare-\tiG\|_{L^\infty H^{s-\frac 12}}\|(\tiG_\vare-\tiG_0)^\top \|_{L^2H^{s-\frac 12}}\| J(\tiX)^{-1}\|_{L^\infty H^{s-\frac 12}}\|\nabla_{\Lambda} \tiX \|_{L^\infty H^{s-\frac 12}}\nonumber\\
	\le &T^\frac 12\|\tiG_\vare-\tiG\|_{L^\infty H^{s}}\|(\tiG_\vare-\tiG_0)^\top \|_{L^\infty H^{s}}\| J(\tiX)^{-1}\|_{L^\infty H^{s}}\|\nabla_{\Lambda} \tiX \|_{L^\infty H^{s}}\nonumber\\
	\le &C(N,N_\vare,\tiv_0,\tiG_0)T^\frac 34\|\tiG_\vare-\tiG\|_{L^\infty H^{s}}\nonumber\\
	\le &C(N,N_\vare,\tiv_0,\tiG_0)T^\frac 34(\|\tiG-\tiG_\vare-t\nabla \tiv_{0}(J-J_\vare)\tiG_0\|_{L^\infty H^{s}}+\|t\nabla \tiv_{0}(J-J_\vare)\tiG_0\|_{L^\infty H^{s}})\nonumber\\
	\le &C(N,N_\vare,\tiv_0,\tiG_0)T \| \tiG-\tiG_\vare-t\nabla \tiv_{0}(J-J_\vare)\tiG_0\|_{\mathcal{A}^{s,\gamma}}+C(N,N_\vare,\tiv_0,\tiG_0)T^\frac 34\vare.\nonumber
\end{align*}
The other terms $ \tilde{I}_{4,1,i} $ and $ \tilde{I}_{4,3,j} $ are estimated similarly. Also, we can estimate $ \tilde{I}_{4,2} $ and $ \tilde{I}_{4,4} $ in the same way.

For the estimates in $ H_{(0)}^{\frac s2-\frac 14}L^2, $ we use the same splitting \eqref{equ666}. We show only the estimate of $ \tilde{I}_{4,2} $ and $ \tilde{I}_{4,4} $. We split $ \tilde{I}_{4,2} $ as follows:
\begin{align}
	\tilde{I}_{4,2}=&[(\tiG_\vare-\tiG_0)+\tiG_0][(\tiG_\vare-\tiG_0)+\tiG_0]^\top [(J(\tiX_\vare)^{-1}-J^{-1})+J^{-1}] (\nabla_{\Lambda} \tiX_\vare-\nabla_{\Lambda} \tiX )\tilde{n}_0\nonumber\\
	=&:\sum_{i=1}^{8}\tilde{I}_{4,2,i}.\nonumber
\end{align}
For $ \tilde{I}_{4,2,1}, $ we apply Lemmas \ref{G-G0}, \ref{lem5.2},  \ref{jx-j}, \ref{lem3.4}, \ref{lem3.5} and  \ref{lem3.7} to obtain
\begin{align*}
	&\|\tilde{I}_{4,2,1}\|_{H_{(0)}^{\frac s2-\frac 14}L^2}\nonumber\\
	=&\|(\tiG_\vare-\tiG_0)(\tiG_\vare-\tiG_0)^\top (J(\tiX_\vare)^{-1}-J^{-1}) (\nabla_{\Lambda} \tiX_\vare-\nabla_{\Lambda} \tiX )\tilde{n}_0\|_{H_{(0)}^{\frac s2-\frac 14}L^2}\nonumber\\
	\le&\|(\tiG_\vare-\tiG_0)(\tiG_\vare-\tiG_0)^\top \|_{H_{(0)}^{\frac s2-\frac 14}H^{\frac12 +\eta}}\\
	&\cdot\|(J(\tiX_\vare)^{-1}-J^{-1}) (\nabla_{\Lambda} \tiX_\vare-\nabla_{\Lambda} \tiX )\tilde{n}_0\|_{H_{(0)}^{\frac s2-\frac 14}L^2}\nonumber\\
	\le &\|\tiG_\vare-\tiG_0\|_{H_{(0)}^{\frac s2-\frac 14}H^{\frac12 +\eta}}\|(\tiG_\vare-\tiG_0)^\top \|_{H_{(0)}^{\frac s2-\frac 14}H^{\frac12 +\eta}}\|J(\tiX_\vare)^{-1}-J^{-1}\|_{H_{(0)}^{\frac s2-\frac 14}H^{\frac12+\mu}}\nonumber\\
	&\cdot\| (\nabla_{\Lambda} \tiX_\vare-\nabla_{\Lambda} \tiX )\tilde{n}_0\|_{H_{(0)}^{\frac s2-\frac 14}H^{\frac12-\mu}}\nonumber\\
	\le &\|\tiG_\vare-\tiG_0\|_{H_{(0)}^{\frac s2-\frac 14}H^{1 +\eta}}\|(\tiG_\vare-\tiG_0)^\top \|_{H_{(0)}^{\frac s2-\frac 14}H^{1 +\eta}}\|J(\tiX_\vare)^{-1}-J^{-1}\|_{H_{(0)}^{\frac s2-\frac 14}H^{1+\mu}}\nonumber\\
	&\cdot\| (\nabla_{\Lambda} \tiX_\vare-\nabla_{\Lambda} \tiX )\tilde{n}_0\|_{H_{(0)}^{\frac s2-\frac 14}H^{1-\mu}}\nonumber\\
	\le &C(N,N_\vare,\tiv_{0},\tiG_0)T^{3\delta}\|  \tiX_\vare- \tiX \|_{H_{(0)}^{\frac s2-\frac 14}H^{1+(1-\mu)}}\nonumber\\
	\le &C(N,N_\vare,\tiv_{0},\tiG_0)T^{3\delta}\|  \tiX_\vare- \tiX -\vare b-t(J_\vare-J)\tiv_{0}\|_{H_{(0)}^{\frac s2-\frac 14}H^{1+(1-\mu)}}\\
	&+C(N,N_\vare,\tiv_{0},\tiG_0)T^{3\delta}\vare\nonumber\\
	\le &C(N,N_\vare,\tiv_{0},\tiG_0)T^{3\delta}\|  \tiX_\vare- \tiX -\vare b-t(J_\vare-J)\tiv_{0}\|_{\mathcal{A}^{s+1,\gamma+1}}\\
	&+C(N,N_\vare,\tiv_{0},\tiG_0)T^{3\delta}\vare\nonumber
\end{align*}
for some $ \eta,\delta $ and $\mu>0 $  small enough. For $ \tilde{I}_{4,4}, $ we use the following splitting:
\begin{align}
	\tilde{I}_{4,4} 
	=&[(\tiG-\tiG_0)+\tiG_0](\tiG_\vare^\top -\tiG^\top ) [(J(\tiX)^{-1}-J^{-1})+J^{-1}][(\nabla_{\Lambda} \tiX-\id)+\id] \tilde{n}_0\nonumber\\
	=&:\sum_{i=1}^{8}\tilde{I}_{4,4,i}.\nonumber
\end{align}
For $ \tilde{I}_{4,4,1}, $ we apply Lemmas \ref{G-G0}, \ref{lem5.2}, \ref{jx-j}, \ref{lem3.4},  \ref{lem3.5} and \ref{lem3.7} to have
\begin{align*}
	&\|\tilde{I}_{4,4,1}\|_{H_{(0)}^{\frac s2-\frac 14}L^2}\nonumber\\
	=&\|(\tiG-\tiG_0)(\tiG_\vare^\top -\tiG^\top ) (J(\tiX)^{-1}-J^{-1})(\nabla_{\Lambda} \tiX-\id) \tilde{n}_0\|_{H_{(0)}^{\frac s2-\frac 14}L^2}\nonumber\\
	\le&\|(\tiG-\tiG_0)(\tiG_\vare-\tiG)^\top \|_{H_{(0)}^{\frac s2-\frac 14}H^{\frac12 +\eta}}\|(J(\tiX)^{-1}-J^{-1}) (\nabla_{\Lambda} \tiX-\id )\tilde{n}_0\|_{H_{(0)}^{\frac s2-\frac 14}L^2}\nonumber\\
	\le &\|\tiG-\tiG_0\|_{H_{(0)}^{\frac s2-\frac 14}H^{\frac12 +\eta}}\|\tiG_\vare-\tiG\|_{H_{(0)}^{\frac s2-\frac 14}H^{\frac12 +\eta}}\|J(\tiX)^{-1}-J^{-1}\|_{H_{(0)}^{\frac s2-\frac 14}H^{\frac12+\mu}}\nonumber\\
	&\| (\nabla_{\Lambda} \tiX-\id )\tilde{n}_0\|_{H_{(0)}^{\frac s2-\frac 14}H^{\frac12-\mu}}\nonumber\\
	\le &\|\tiG-\tiG_0\|_{H_{(0)}^{\frac s2-\frac 14}H^{1 +\eta}}\|\tiG_\vare-\tiG\|_{H_{(0)}^{\frac s2-\frac 14}H^{1 +\eta}}\|J(\tiX)^{-1}-J^{-1}\|_{H_{(0)}^{\frac s2-\frac 14}H^{1+\mu}}\nonumber\\
	&\| (\nabla_{\Lambda} \tiX-\id )\tilde{n}_0\|_{H_{(0)}^{\frac s2-\frac 14}H^{1-\mu}}\nonumber\\
	\le &C(N,N_\vare,\tiv_{0},\tiG_0)T^{3\delta}\|\tiG_\vare-\tiG\|_{H_{(0)}^{\frac s2-\frac 14}H^{1 +\eta}}\nonumber\\
	\le &C(N,N_\vare,\tiv_{0},\tiG_0)T^{3\delta}\|\tiG-\tiG_\vare-t\nabla \tiv_{0}(J-J_\vare)\tiG_0\|_{H_{(0)}^{\frac s2-\frac 14}H^{1 +\eta}}\\
	&+C(N,N_\vare,\tiv_{0},\tiG_0)T^{3\delta}\vare\nonumber\\
	\le &C(N,N_\vare,\tiv_{0},\tiG_0)T^{3\delta}\| \tiG-\tiG_\vare-t\nabla \tiv_{0}(J-J_\vare)\tiG_0\|_{\mathcal{A}^{s,\gamma}}+C(N,N_\vare,\tiv_{0},\tiG_0)T^{3\delta}\vare\nonumber
\end{align*}
for some $ \eta,\delta$ and $\mu>0 $ small enough. The other terms $ \tilde{I}_{4,2,i} $ and $ \tilde{I}_{4,4,j} $ are estimated similarly. Also, we can estimate $ \tilde{I}_{4,1} $ and $ \tilde{I}_{4,3} $ in the same manner.

This proposition holds if we put together the estimates of  $ \tiF_\vare $ in  ${\mathcal{K}_{(0)}^{s-1}},  \tiK_\vare$ in  ${\bar{\mathcal{K}}_{(0)}^{s}} $ and  $ \tiH_\vare $ in ${\mathcal{K}_{(0)}^{s-\frac 12}}. $

From \eqref{Stability X}, \eqref{Stability G} and \eqref{Stability w,q} we conclude that
\begin{align}
	&\|\tilde{X}-\tilde{X}_{\vare}+\vare b-t (J-J_{\vare} ) \tilde{v}_0\|_{\mathcal{A}^{s+1, \gamma+1}}+\| \tilde{w}-   \tilde{w}_{\vare}\|_{\mathcal{K}_{(0)}^{s+1}}+\| \tilde{q}_w-\tilde{q}_{w, \vare}\|_{\mathcal{K}_{p r}^s(0)}\nonumber\\
	&+\|\tilde{G}-\tilde{G}_{\vare}-t\nabla \tilde{v}_0(J-J_{\vare})  \tilde{G}_0\|_{\mathcal{A}^{s, \gamma}}\nonumber\\
	\leq & C \vare+ C T^{\delta}(\|\tilde{X}-\tilde{X}_{\vare}+\vare  b-t(J-J_{\vare}) \tilde{v}_0\|_{\mathcal{A}^{s+1, \gamma+1}}+\|\tilde{w}-\tilde{w}_{\vare}\|_{\mathcal{K}_{(0)}^{s+1}}\nonumber\\
	& +\| \tilde{q}_w-\tilde{q}_{w, \vare}\|_{\mathcal{K}_{p r}^s(0)}+\|\tilde{G}-\tilde{G}_{\vare}-t\nabla \tilde{v}_0(J-J_{\vare})   \tilde{G}_0\|_{\mathcal{A}^{s, \gamma}})\nonumber
\end{align}
where $ \delta=\min_{i=1,2,3}\delta_i>0 $ and $ C $ depends only on the initial data. If $ 0<T<(\frac C2)^{-\frac 1\delta}, $ we have
\begin{align}
	&\|\tilde{X}-\tilde{X}_{\vare}+\vare b-t (J-J_{\vare} ) \tilde{v}_0\|_{\mathcal{A}^{s+1, \gamma+1}}+\| \tilde{w}-   \tilde{w}_{\vare}\|_{\mathcal{K}_{(0)}^{s+1}}+\| \tilde{q}_w-\tilde{q}_{w, \vare}\|_{\mathcal{K}_{p r}^s(0)}\nonumber\\
	&+\|\tilde{G}-\tilde{G}_{\vare}-t\nabla \tilde{v}_0(J-J_{\vare})  \tilde{G}_0\|_{\mathcal{A}^{s, \gamma}}
	\leq  2 C \vare.\nonumber
\end{align}
In particular, we obtain
\begin{align}
	\|\tilde{X}-\tilde{X}_{\vare}\|_{L^\infty H^{s+1}}
	\leq  C\vare,\nonumber
\end{align}
thus we have proved Theorem \ref{thm2}.
\end{proof}

\section{Existence of splash singularity}\label{Existence of splash singularity}
\par The choice of the initial velocity is essential to form splash singularities. We are looking for initial velocity such that the compatibility conditions \eqref{equinitialdata} is satisfied and the inner product of the velocity field and the positive normal vector is positive as in Fig. \ref{fig:2}.  The evolution of the splash curve is demonstrated in Fig. \ref{fig:3}.

\par For our problem, we extend the analysis for the choice of the initial velocity already
made in \cite[Section 7] {castro2019splash}. For the convenience of the reader, we recall the arguments in \cite{castro2019splash}.

In the following discussion, we denote the typical domain by $ \Omega $ and consider the following parametrization of the boundary $ \partial \Omega: $
\begin{equation}\label{eque1}
	z(r)=(z_1(r),z_2(r)),\quad | \dot{z}(r)|=1.
\end{equation}
We also choose a small enough neighborhood $ U $ of $ \partial\Omega. $ In $ U, $ we use the coordinates $ (r,\lambda) $ given by
\begin{align}
	x(r, \lambda):=&z(r)+\lambda  \dot{z}(r)^\perp,\nonumber\nonumber\\
	(x_1(r, \lambda),x_2(r, \lambda)):=&(z_1(r)-\lambda\dot{z}_2(r),z_2(r)+\lambda\dot{z}_1(r)).\nonumber
\end{align}
Then   the stream function $ \psi(x_1,x_2) $ in $ U $ will be given by
\begin{align}
	\bar{\psi}(r,\lambda):=&\psi_0(r)+\psi_1(r)\lambda+\frac 12\psi_2(r)\lambda^2,\nonumber\\
	\psi(x(r,\lambda)):=&\bar{\psi}(r,\lambda),\label{eque2}
\end{align}
and we extend smoothly  $ \psi $ to the rest of the domain $ \Omega. $ 

Now we take $ u_0(x_1,x_2):=\nabla^\perp\psi(x_1,x_2) $ and notice that $ u_0 $ is divergence free 
\begin{align*}
	\dive u_0=\partial_1(-\partial_2\psi)+\partial_2(\partial_1\psi)=0.
\end{align*}
From \eqref{equinitialdata}, the initial velocity $ u_0 $ must satisfy
\begin{align*}
	n^\perp ( (\nabla u_0+\nabla u_0^\top  )+ H_0 \otimes H_0 ) n=0,
\end{align*} 
where $ n^\perp=(-n^2, n^1) $ and $ n=(n^1, n^2)^\top  $ are the tangential and normal vectors to $ \partial\Omega, $ respectively.
If $ T $ and $ N $ are an extension of $ n^\perp $ and $ n $ to $ U, $ respectively, we can rewrite it as
\begin{equation}\label{eque4}
(	T  (\nabla u_0+\nabla u_0^\top  ) N)|_{\partial\Omega}=-(T H_0 \otimes H_0  N)|_{\partial\Omega}.
\end{equation}
From \eqref{eque1}, we have 
\begin{align*}
	\ddot{z}_2\dot{z}_2+\ddot{z}_1\dot{z}_1=0,
\end{align*}
then we can take 
\begin{align}
	T(r,\lambda):=&\partial_r x(r,\lambda)=\dot{z}(r)+\lambda\ddot{z}^\perp(r)=(1-\lambda\kappa(r))\dot{z}(r),\nonumber\\
	N(r,\lambda):=&\partial_\lambda x(r,\lambda)=\dot{z}^\perp(r),\nonumber
\end{align}
where 
\begin{align*}
	\kappa(r):=\ddot{z}(r)\cdot\dot{z}^\perp(r).
\end{align*}
We also notice that
\begin{align*}
 	(\ddot{z}_2,-\ddot{z}_1)=\kappa(\dot{z}_1, \dot{z}_2).
\end{align*}

From \eqref{eque1} and \eqref{eque2}, we have
\begin{align*}
	(\nabla\psi)(x(r,\lambda))=\frac{1}{1-\kappa(r)\lambda}\partial_r\bar{\psi}(r,\lambda) \dot{z}(r)+\partial_\lambda\bar{\psi}(r,\lambda)\dot{z}^\perp(r).
\end{align*}
By defining 
\begin{equation}\label{eque3}
	\bar{u}_0(r,\lambda):=u_0(x(r,\lambda)),
\end{equation}
we have
\begin{align*}
	\bar{u}_0(r,\lambda)=	(\nabla\psi)^\perp(x(r,\lambda))=\frac{1}{1-\kappa(r)\lambda}\partial_r\bar{\psi}(r,\lambda) \dot{z}^\perp(r)+\partial_\lambda\bar{\psi}(r,\lambda)\dot{z}(r).
\end{align*}
Taking derivatives of \eqref{eque3}, we obtain
\begin{align}
	(\partial_r\bar{u}_0^j)(r,\lambda)=&T^i(r,\lambda)(\partial_iu_0^j)(x(r,\lambda)),\nonumber\\	(\partial_\lambda\bar{u}_0^j)(r,\lambda)=&N^i(r,\lambda)(\partial_iu_0^j)(x(r,\lambda)).\nonumber
\end{align}
Now the terms on the left side of \eqref{eque4} defined in $ U $ become
\begin{align}
	T(r,\lambda)\nabla u_0(x(r,\lambda)) N(r,\lambda)=&T^i(r,\lambda)(\partial_iu_0^j)(x(r,\lambda))N^j(r,\lambda) =(\partial_r\bar{u}_0^j)(r,\lambda)N^j(r,\lambda),\nonumber\\
	T(r,\lambda)\nabla u_0^\top (x(r,\lambda)) N(r,\lambda)=&T^i(r,\lambda)(\partial_ju_0^i)(x(r,\lambda))N^j(r,\lambda) =(\partial_\lambda\bar{u}_0^i)(r,\lambda)T^i(r,\lambda),\nonumber
\end{align}
i.e.,
\begin{align}
	T\nabla u_0(x(r,\lambda)) N=& \partial_r (\bar{u}_0\cdot N) -\bar{u}_0\cdot \partial_r N,\nonumber\\
	T\nabla u_0^\top (x(r,\lambda)) N=&\partial_\lambda (\bar{u}_0\cdot T) -\bar{u}_0\cdot \partial_\lambda T.\nonumber
\end{align}
Note that
\begin{align*}
	\begin{cases}
		\bar{u}_0\cdot N=\frac{1}{1-\lambda\kappa}\partial_r\bar{\psi},\\
		\bar{u}_0\cdot T=-(1-\lambda\kappa)\partial_\lambda\bar{\psi},\\
		\partial_r N=\ddot{z}^\perp=-\kappa\dot{z},\\
		\partial_\lambda T=\ddot{z}^\perp=-\kappa\dot{z},
	\end{cases}
\end{align*}
and we obtain
\begin{align*}
	\begin{cases}
		\partial_r(\bar{u}_0\cdot N)|_{\lambda=0}= \partial_r^2\bar{\psi}(r,0),\\
		\partial_\lambda(\bar{u}_0\cdot T)|_{\lambda=0}=\kappa\partial_\lambda\bar{\psi}(r,0)-\partial_\lambda^2\bar{\psi}(r,0),\\
		(\bar{u}_0\cdot \partial_r N)|_{\lambda=0}=\kappa\partial_\lambda\bar{\psi}(r,0),\\
		(\bar{u}_0\cdot \partial_\lambda T)|_{\lambda=0}=\kappa\partial_\lambda\bar{\psi}(r,0).
	\end{cases}
\end{align*}
Then the left side of \eqref{eque4}
becomes
\begin{align}
	(T\nabla u_0(x(r,\lambda)) N+T\nabla u_0^\top (x(r,\lambda)) N)|_{\lambda=0}=\partial_r^2\bar{\psi}(r,0)-\kappa\partial_\lambda\bar{\psi}(r,0)-\partial_\lambda^2\bar{\psi}(r,0),\nonumber
\end{align}
i.e.,
\begin{align*}
	\partial_r^2\bar{\psi}(r,0)-\kappa\partial_\lambda\bar{\psi}(r,0)-\partial_\lambda^2\bar{\psi}(r,0)=-(T H_0 \otimes H_0  N)|_{\lambda=0}.
\end{align*}
More precisely, we have
\begin{equation}\label{eque5}
	\partial_r^2{\psi}_0(r)-\kappa{\psi}_1(r)-{\psi}_2(r)=-(T H_0 \otimes H_0  N)|_{\lambda=0}.
\end{equation}
Thus for any given $\psi_0$ such that $u_0\cdot n|_{\partial\Omega}=\partial_r \psi_0>0$ and for any $H_0$, there exist $\psi_1$ and $\psi_2$, such that \eqref{eque5} is satisfied and in particular,  compatibility condition \eqref{eque4} is satisfied. For convenience, we can  choose $ \psi_1=0$ and $ \psi_2= \partial_r^2{\psi}_0+(T H_0 \otimes H_0  N)|_{\partial\Omega}.$

Therefore, we state the following proposition:
\begin{proposition}\label{pro4}
	For a given $ \psi_0\in C^2(U), $ such that $ \partial_r \psi_0>0 $ and for any divergence free magnetic field $ H_0\in C^1(U), $ there exists  a stream function of type \eqref{eque2} with $ \psi_1=0 $ and $ \psi_2= \partial_r^2{\psi}_0+(T H_0 \otimes H_0  N)|_{\partial\Omega}  $ such that the initial velocity defined by $ u_0(x_1,x_2)=\nabla^\perp\psi(x_1,x_2) $ is divergence free and the compatibility conditions \eqref{equinitialdata} are satisfied. Moreover, $ $ we have $ u_0\cdot n>0 $ on $ {\partial\Omega}. $
\end{proposition}
This choice of the initial velocity $ \tiv_0 $ allows us to obtain  domain $ \tiOmega(t) $ with $ 0<t<T $ for $ T $ small enough from  (local existence) Theorem \ref{thm1}. Therefore, we can choose sufficiently small $ \bar{t}\in (0,T) $ such that $ P^{-1}(\partial\tiOmega(\bar{t})) $ is a self-intersecting domain since $ \tiv_{0}\cdot\tilde{n}_0>0 $ on $ \partial\tiOmega_0. $ But \eqref{equthm2} in  (stability) Theorem \ref{thm2} ensures that $ P^{-1}(\partial\tiOmega_\vare(\bar{t})) $ is also self-intersecting for $ \vare $  small enough as in Fig. \ref{fig:5}. Note that $ P^{-1}(\partial\tiOmega_\vare(0)) $ is regular due to our choice of $ b $ and for a later time we end up in a self-intersecting domain as in Fig. \ref{fig:4}. Finally, the continuity argument guarantees the existence of a splash time $ t^*_\vare\in (0,\bar{t}) $ and $ P^{-1}(\partial\tiOmega_{\vare}(t^*_\vare)) $ has a splash singularity. 

Indeed, we fix $ \vare>0 $ as above and define
\begin{align*}
	t^*_\vare:=\inf\{ t\in [0,\bar{t}] :P^{-1}(\partial\tiOmega_\vare(t)) \text{ and } \partial\tiOmega_\vare(t) \text{ are as in Fig. \ref{fig:proof-of-final-theorem} (2) and (3)} \}.
\end{align*}
\begin{figure}[h]
	\centering
	\includegraphics[width=0.8\linewidth]{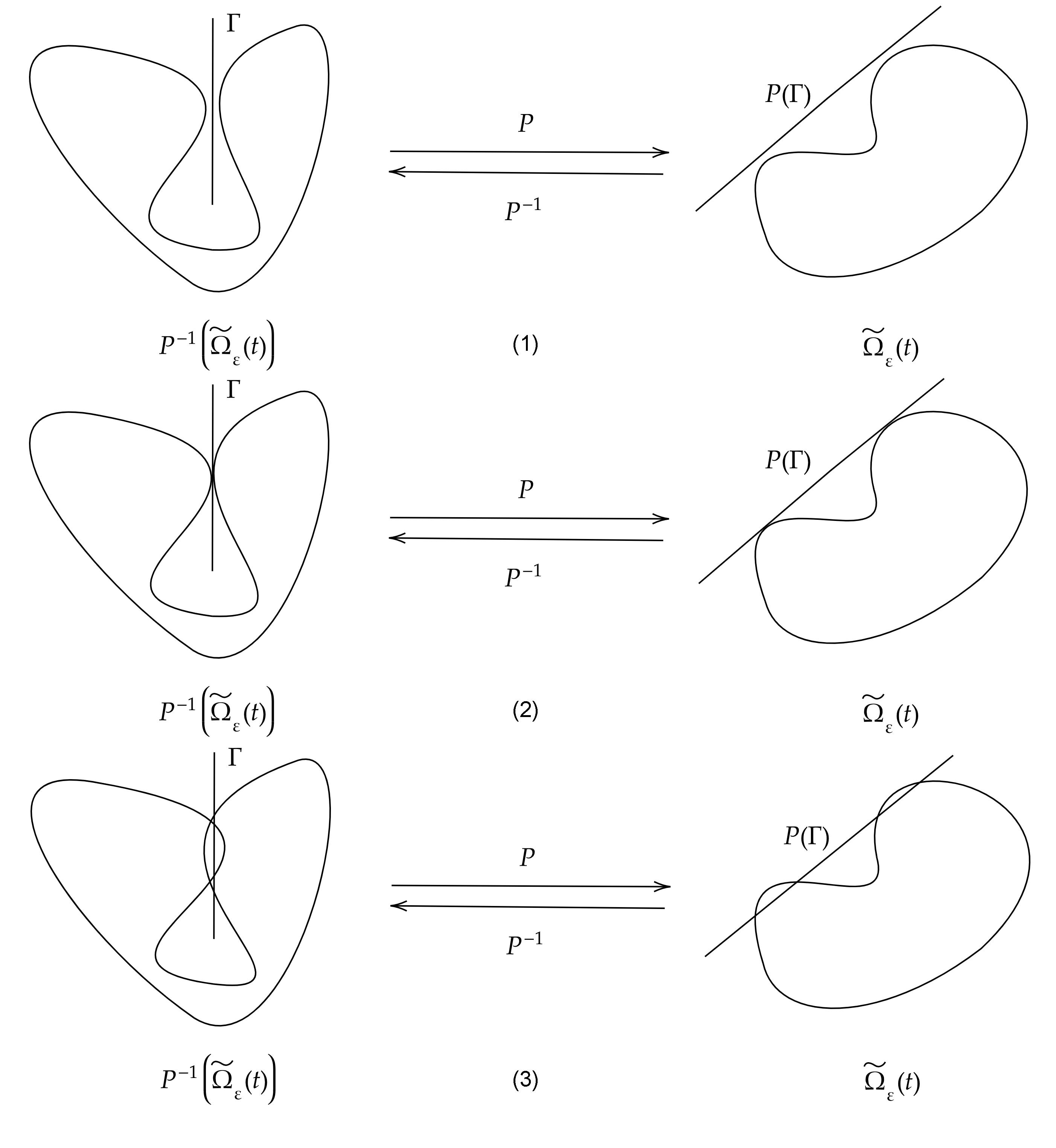}
	\caption{Possibilities for time $ t>0 .$}
	\label{fig:proof-of-final-theorem}
\end{figure}
Then, we have $ 0<t^*_\vare<\bar{t} $ and $ P^{-1}(\partial\tiOmega_\vare(t^*_\vare)) $ is as in Fig. \ref{fig:proof-of-final-theorem} (2). For $ 0\le t< t^*_\vare, $ $ P^{-1}(\partial\tiOmega_\vare(t^*_\vare)) $ is as in Fig. \ref{fig:proof-of-final-theorem} (1). Therefore, $ (\tiOmega_{\vare}(t), \tiw^\prime_\vare, \tiq^\prime_\vare, \tiX^\prime_\vare, \tiG^\prime_\vare) $ gives rise to a solution of the viscous MHD equation, for $ t\in[0,t^*_\vare). $ Moreover, the domain $ \tiOmega_{\vare}(t) $ form a splash singularity at time $ t^*_\vare. $

Finally, we state the above argument as the following theorem.
\begin{theorem}\label{thm3}
	There exists a bounded domain $ \Omega_0= P^{-1}(\partial\tiOmega_\vare(0))  $ with a sufficiently  smooth boundary,  as in the above discussion, such that for any divergence free $ H_0\in H^k(\Omega_0)$ with $ k $ large enough, we can construct a suitable initial  velocity  $ u_0\in H^k(\Omega_0),$ and  there exists  a  solution $  (u, p, H) $ for the viscous MHD \eqref{mhd} in $ [0, t^*_\vare) $ for  $ t^*_\vare>0  $ such that $ (\tiw, \tiq_w, \tiG-\hat{G})  \in   \mathcal{K}_{(0)}^{s+1} \times \mathcal{K}_{p r(0)}^s \times \mathcal{A}^{s, \gamma}$ for $ 2<s<\frac 52 $ and $ 1<\gamma<s-1. $  Moreover, the interface $ \partial\Omega(t^*_\vare) $   remains regular but self-intersects  in at least one point which creates a splash singularity.
\end{theorem}

\appendix
\renewcommand{\thesection}{\Alph{section}}
\section{Further estimates and key lemmas}\label{Further estimates and key lemmas}
\begin{lemma}\label{jx-j} (\cite[Lemma 3.10, Lemma 3.11]{castro2019splash}, \cite[Lemma A.1]{di2020splash}).
	Let $2<s<\frac{5}{2}, 1<\gamma< s-1, \delta, \mu>0 $ small enough and $\tilde{X}-\hat{X} \in \mathcal{A}^{s+1, \gamma+1}$. Then, for $T>0$ small enough, we have
\begin{align*}
	 &\|J(\tilde{X}) \|_{L^{\infty} H^{s+1}} \leq  C (M, \tilde{v}_0,  \|\tilde{X}-\hat{X} \|_{\mathcal{A}^{s+1, \gamma+1}} ),\nonumber\\
	 &\|J(\tilde{X})-J \|_{L^{\infty} H^{s+1}} 
    \leq   C (M, \tilde{v}_0, \|\tilde{X}-\hat{X} \|_{\mathcal{A}^{s+1, \gamma+1}} ) ( \|\tilde{X}-\hat{X} \|_{L^{\infty} H^{s+1}}   + \|t J\tilde{v}_0 \|_{L^{\infty} H^{s+1}} ),\nonumber\\
	 &\|J (\tilde{X})-J \|_{H_{(0)}^1 H^{\gamma+1}}  \leq C (M, \tilde{v}_0, \|\tilde{X}-\hat{X} \|_{\mathcal{A}^{s+1, \gamma+1}} )\|\tilde{X}-\tiomega\|_{H_{(0)}^1 H^{\gamma+1}},\nonumber\\
	 &\|J (\tilde{X})-J \|_{H_{(0)}^{\frac{s-1}{2}} H^{1+\mu}}  \leq C (M, \tilde{v}_0, \|\tilde{X}-\hat{X} \|_{\mathcal{A}^{s+1, \gamma+1}} )(\|\tilde{X}-\hat{X}\|_{H_{(0)}^{\frac{s-1}{2}+\delta} H^{1+\mu}}+T),\nonumber
\end{align*}
with 
\begin{align*}
	M=\frac{1}{\inf_{\tiomega}|\tiomega|-C(\tiv_0) T-T^{\frac 14}\|\tiX-\hat{X}\|_{\mathcal{A}^{s+1,\gamma+1}}}.
\end{align*}
\end{lemma}
\begin{lemma}\label{jx-jy}(\cite[Lemma 3.12]{castro2019splash}).	Let $2<s<\frac{5}{2}, 1<\gamma< s-1,  \delta,\mu>0 $ small enough and $\tilde{X}-\hat{X}, \tilde{Y}-\hat{X} \in \mathcal{A}^{s+1, \gamma+1}$. Then, for $T>0$ small enough, we have
	\begin{align}
		 \|J(\tilde{X})-J(\tilde{Y}) \|_{L^{\infty} H^{s+1}} 
		&\leq  C (M, \tilde{v}_0, \|\tilde{X}-\hat{X} \|_{\mathcal{A}^{s+1, \gamma+1}},  \|\tilde{Y}-\hat{X} \|_{\mathcal{A}^{s+1, \gamma+1}} ) \|\tilde{X}-\tilde{Y} \|_{L^{\infty} H^{s+1}},\nonumber\\
		 \|J (\tilde{X})-J(\tilde{Y}) \|_{H_{(0)}^1 H^{\gamma+1}} & \leq C (M, \tilde{v}_0, \|\tilde{X}-\hat{X} \|_{\mathcal{A}^{s+1, \gamma+1}},  \|\tilde{Y}-\hat{X} \|_{\mathcal{A}^{s+1, \gamma+1}} )\|\tilde{X}-\tilde{Y}\|_{H_{(0)}^1 H^{\gamma+1}},\nonumber
	\end{align}
with 
\begin{align*}
	M=\max  \bigg\{&\frac{1}{\inf|\tiomega|-C(\tiv_0) T-T^{\frac 14}\|\tiX-\hat{X}\|_{\mathcal{A}^{s+1,\gamma+1}}}, \\
	&\frac{1}{\inf|\tiomega|-C(\tiv_0) T-T^{\frac 14}\|\tilde{Y}-\hat{X}\|_{\mathcal{A}^{s+1,\gamma+1}}} \bigg\} .
\end{align*}		
\end{lemma}
\begin{lemma}\label{zeta-}(\cite[Lemma 3.13]{castro2019splash}, \cite[Lemma A.4]{di2020splash}).
	Let $2<s<\frac{5}{2}, 1<\gamma< s-1,  \delta,\mu>0 $ small enough and $\tilde{X}-\hat{X} \in \mathcal{A}^{s+1, \gamma+1},$ and $\tilde{\zeta}=(\nabla \tilde{X})^{-1}$. Then for $T>0$ small enough, we have
\begin{align}
	\|\tilde{\zeta}\|_{L^{\infty} H^s}+\sum_{i=1}^2 \|\partial_i \tilde{\zeta} \|_{L^{\infty} H^s} &\leq C (M, \|\tilde{X}-\hat{X} \|_{\mathcal{A}^{s+1, \gamma+1}} ),\nonumber\\
	\|\tilde{\zeta}-\mathcal{I}\|_{L^{\infty} H^s} &\leq C (M, \|\tilde{X}-\hat{X} \|_{\mathcal{A}^{s+1, \gamma+1}} )\|\tilde{X}-\tiomega\|_{L^{\infty} H^{s+1}},\nonumber\\	
	\|\tilde{\zeta}-\mathcal{I}\|_{H_{(0)}^{\frac{s-1}{2}+\delta} H^{1+\mu}} &\leq
	C (M, \|\tilde{X}-\hat{X} \|_{\mathcal{A}^{s+1, \gamma+1}} ) \|\tilde{X}-\tiomega\|_{H_{(0)}^{\frac{s-1}{2}+\delta} H^{2+\mu}},\nonumber\\
	\|\tilde{\zeta}-\mathcal{I}\|_{H_{(0)}^1 H^{\gamma}} &\leq C (M, \|\tilde{X}-\hat{X} \|_{\mathcal{A}^{s+1, \gamma+1}} )\|\tilde{X}-\tiomega\|_{H_{(0)}^1 H^{\gamma+1}},\nonumber
\end{align}
where 
\begin{align*}
	M=\frac{1}{1-C(\tiv_0) T-CT^{\frac 14}\|\tiX-\hat{X}\|_{\mathcal{A}^{s+1,\gamma+1}} -CT^{\frac 12}\|\tiX-\hat{X}\|^2_{\mathcal{A}^{s+1,\gamma+1}} }.
\end{align*}
\end{lemma}
\begin{lemma}\label{zetan-zetan-1}(\cite[Lemma 3.15]{castro2019splash},  \cite[Lemma A.5]{di2020splash}).
	Let $2<s<\frac{5}{2}, 1<\gamma< s-1,  \delta,\mu>0 $ small enough and $\tilde{X}^{(n)}-\hat{X}, \tilde{X}^{(n-1)}-\hat{X} \in \mathcal{A}^{s+1, \gamma+1}$ and $\tilde{\zeta}^{(n)}=(\nabla \tilde{X}^{(n)})^{-1}$, $\tilde{\zeta}^{(n-1)}=(\nabla \tilde{X}^{(n-1)})^{-1}$. Then, for $T>0$ small enough we have
\begin{align}
	 \|\tilde{\zeta}^{(n)}-\tilde{\zeta}^{(n-1)} \|_{H_{(0)}^{\frac{s-1}{2}+\delta} H^{1+\mu}} &\leq C (M, \tilde{v}_0 ) \|\tilde{X}^{(n)}-\tilde{X}^{(n-1)} \|_{H_{(0)}^{\frac{s-1}{2}+\delta} H^{2+\mu}},\nonumber\\
	 \|\tilde{\zeta}^{(n)}-\tilde{\zeta}^{(n-1)} \|_{H_{(0)}^1 H^{\gamma}} &\leq C (M, \tilde{v}_0 ) \|\tilde{X}^{(n)}-\tilde{X}^{(n-1)} \|_{H_{(0)}^1 H^{\gamma+1}},\nonumber
\end{align}
	where
	\begin{align*}
		M=\max _{m=n-1, n} \frac{1}{1-C(\tiv_0) T-CT^{\frac 14}\|\tiX^{(m)}-\hat{X}\|_{\mathcal{A}^{s+1,\gamma+1}} -CT^{\frac 12}\|\tiX^{(m)}-\hat{X}\|^2_{\mathcal{A}^{s+1,\gamma+1}} }.
	\end{align*}
\end{lemma}
\begin{lemma}\label{lem3.2}(\cite[Lemma2.3]{beale1981initial}).
	Suppose $ 0 \leq r \leq 4. $
	
	(1) The Identity extends to a bounded operator
	\begin{align*}
		\mathcal{K}^r((0, T) ; \Omega) \rightarrow H^p(0, T) H^{r-2 p}(\Omega)
	\end{align*}
	for $p \leq \frac{r}{2}$.
	
	(2) If $r$ is not an odd integer, the restriction of this operator to the subspace with $\partial_t^k v(0)=0,0 \leq k<\frac{r-1}{2}$ is bounded independently on $T$, indeed
	\begin{align*}
		\|v\|_{H_{(0)}^p H^{r-2 p}} \leq C\|v\|_{\mathcal{K}_{(0)}^r} .
	\end{align*}
\end{lemma}
\begin{lemma}\label{lem3.3}(\cite[Lemma 3.3]{di2020splash}).
	Let $\bar{T}>0$ be arbitrary, $B$ a Hilbert space and choose $T \leq \bar{T}$.
	
	(1) For $v \in L^2((0, T) ; B)$, we define $V \in H^1((0, T) ; B)$ by
	\begin{align*}
		V(t)=\int_0^t v(\tau) d \tau .
	\end{align*}
	For $0<s<\frac{1}{2}$ and $0 \leq \vare<s$, then the map $v \rightarrow V$ is a bounded operator from $H^s((0, T) ; B)$ to $H^{s+1-\vare}((0, T) ; B)$, and
	\begin{align*}
		\|V\|_{H^{s+1-\vare}((0, T) ; B)} \leq C_0 T^{\vare}\|v\|_{H^s((0, T) ; B)},
	\end{align*}
	where $C_0$ is independent of $T$ for $0<T \leq \bar{T}$.
	
	(2) For $\frac{1}{2}<s<1$, we impose $v(0)=0$ and $0 \leq \vare<s$. Then $v \rightarrow V$ is a bounded operator from $H_{(0)}^s((0, T) ; B)$ to $H_{(0)}^{s+1-\vare}((0, T) ; B)$ and
	\begin{align*}
		\|V\|_{H_{(0)}^{s+1-\vare}((0, T) ; B)} \leq C_0 T^{\vare}\|v\|_{H_{(0)}^s((0, T) ; B)},
	\end{align*}
	where $C_0$ is independent of $T$ for $0<T<\bar{T}$.
\end{lemma}
\begin{lemma}\label{lem3.4}
	Suppose $\Omega \in \mathbb{R}^n, r>\frac n2$ and $r \geq s \geq 0$. If $v \in H^r(\Omega)$ and $w \in H^s(\Omega)$, then  $vw \in$ $H^s(\Omega)$ and
	\begin{align*}
		\|v w\|_{H^s(\Omega)} \leq C\|v\|_{H^r(\Omega)}\|w\|_{H^s(\Omega)}.
	\end{align*}
\end{lemma}
\begin{lemma}\label{lem3.5}(\cite[Lemma 3.6]{castro2019splash}).
	If $v \in H^{\frac{1}{q}}$ and $w \in H^{\frac{1}{p}}$ with $\frac{1}{p}+\frac{1}{q}=1$ and $1<p<\infty$ then
	\begin{align*}
		\|v w\|_{L^2} \leq C\|v\|_{H^{\frac{1}{q}}}\|w\|_{H^{\frac{1}{p}}}.
	\end{align*}
\end{lemma}
\begin{lemma}\label{lem3.6}(\cite[Lemma 2.6]{beale1981initial}).
	Suppose $B, Y, Z$ are Hilbert spaces, and $M: B \times Y \rightarrow Z$ is a bounded, bilinear operator. Suppose $w \in H^s((0, T) ; B)$ and $v \in H^s$ $((0, T) ; Y)$, where $s>\frac{1}{2}$. If $v w$ is defined by $M(v, w)$, then $v w \in H^s((0, T) ; Z)$ and the following hold:
	
	(1)	$		\|v w\|_{H^s((0, T) ; Z)} \leq C\|v\|_{H^s((0, T) ; Y)}\|w\|_{H^s((0, T) ; B)}$ .
		
	(2) In addition, if $s \leq 2$ and $\partial_t^k v(0)=\partial_t^k w(0)=0,0 \leq k<s-\frac{1}{2}$ and $s-\frac{1}{2}$ is not an integer, then the constant $C$ in (1) can be chosen independently on $T$. Indeed
	\begin{align*}
		\|v w\|_{H_{(0)}^s((0, T) ; Z)} \leq C\|v\|_{H_{(0)}^s((0, T) ; Y)}\|w\|_{H_{(0)}^s((0, T) ; B)}.
	\end{align*}
\end{lemma}
\begin{lemma}\label{lem3.7}(\cite[Lemma 3.8]{castro2019splash},  \cite[Lemma 3.7]{di2020splash}).
	For $2<s<\frac{5}{2}$, $\vare, \delta>0$ small enough and $v \in \mathcal{A}^{s+1, \gamma}$ the following estimates hold:
	
	(1) $\|v\|_{H_{(0)}^{\frac{s+1}{2}} H^{1-\vare}} \leq C\|v\|_{\mathcal{A}^{s+1, \gamma}},$
	
	(2) $\|v\|_{H_{(0)}^{\frac{s+1}{2}+\vare} H^{1+\delta}} \leq C\|v\|_{\mathcal{A}^{s+1, \gamma}},$
	
	(3) $\|v\|_{H_{(0)}^{\frac{s-1}{2}+\vare} H^{2+\delta}} \leq C\|v\|_{\mathcal{A}^{s+1, \gamma}},$
	
	(4) $\|v\|_{H_{(0)}^{\frac{s-1}{2}+\vare} H^{1+\delta}} \leq C\|v\|_{\mathcal{A}^{s, \gamma-1}},$
	
	(5) $\|v\|_{H_{(0)}^{\frac{s}{2}-\frac{1}{4}+\vare} H^{2+\delta}} \leq C\|v\|_{\mathcal{A}^{s+1, \gamma}}$,
	
	(6) $\|v\|_{H_{(0)}^{\frac{s}{2}-\frac{1}{4}+\vare} H^{1+\delta}} \leq C\|v\|_{\mathcal{A}^{s, \gamma-1}}$,
	
	(7) $\|v\|_{H_{(0)}^1 H^{s-1}} \leq C\|v\|_{\mathcal{A}^{s+1, \gamma}},$
	
	(8) $\|v\|_{H_{(0)}^{\frac{1}{2}+2 \vare} H^s} \leq C\|v\|_{\mathcal{A}^{s+1, \gamma}}.$
\end{lemma}
\begin{lemma}\label{lem3.9}(\cite[Lemma 2.1]{beale1981initial}).
	Let $\Omega$ be a bounded set with a sufficiently smooth boundary then the following trace theorems hold:
	
	(1) Suppose $\frac{1}{2}<s \leq 5$. The mapping $v \rightarrow \partial_n^j v$ extends to a bounded operator $\mathcal{K}^s([0, T] ; \Omega) \rightarrow \mathcal{K}^{s-j-\frac{1}{2}}([0, T] ; \partial \Omega)$, where $j$ is an integer $0 \leq j<s-\frac{1}{2}$. The mapping $v \rightarrow \partial_t^k v(\alpha, 0)$ extends to a bounded operator $\mathcal{K}^s([0, T] ; \Omega) \rightarrow H^{s-2 k-1}(\Omega)$, if $k$ is an integer $0 \leq k<\frac{1}{2}(s-1)$.
	
	(2) Suppose $\frac{3}{2}<s<5, s \neq 3$ and $s-\frac{1}{2}$ not an integer. Let
	\begin{align*}
		\mathcal{W}^s=\prod_{0 \leq j \leq s-\frac{1}{2}} \mathcal{K}^{s-j-\frac{1}{2}}([0, T] ; \partial \Omega) \times \prod_{0 \leq k<\frac{s-1}{2}} H^{s-2 k-1}(\Omega),
	\end{align*}
	and let $\mathcal{W}_0^s$ the subspace consisting of $ \{a_j, w_k \}$, which are the traces described in the previous point, so that $\partial_t^k a_j(\alpha, 0)=\partial_n^j w_k(\alpha), \alpha \in \partial \Omega$, for $j+2 k<s-\frac{3}{2}$. Then the traces in the previous point form a bounded operator $\mathcal{K}^s([0, T] ; \Omega) \rightarrow \mathcal{W}_0^s$ and this operator has a bounded right inverse.
\end{lemma}

\bigskip

{\bf Acknowledgments.} Both the authors were partially supported by NSF of China under Grants 12171460. Hao was also partially supported by the CAS Project for Young Scientists in Basic Research under Grant YSBR-031 and the National Key R\&D Program of China under Grant 2021YFA1000800.

%


\end{document}